\documentclass[twoside, reqno, 10pt]{amsart}

\usepackage[left=4cm, right=4cm, top=3cm, bottom=2.5cm]{geometry}

\usepackage[colorlinks=true, pdfstartview=FitV, linkcolor=blue,citecolor=blue, urlcolor=blue]{hyperref}

\usepackage[usenames]{xcolor} 

\definecolor{labelkey}{rgb}{0,0,1}
\definecolor{Red}{rgb}{0.7,0,0.1}
\definecolor{Green}{rgb}{0,0.7,0}

\usepackage{amsfonts, amssymb, amsmath, mathrsfs, bbm, cjhebrew, gensymb, textcomp, mathtools, dsfont, calligra}

\usepackage[normalem]{ulem}

\usepackage{commath, setspace, subcaption, parcolumns, multirow, multicol, accents, comment, marginnote, verbatim, empheq, enumerate, stackrel, enumitem, float, physics}
\usepackage{amsthm}
\usepackage{thmtools}
\usepackage[capitalize,nameinlink,noabbrev]{cleveref}

\usepackage{graphicx, psfrag, tikz, tikz-cd}
\usepackage[notransparent]{svg} 

\numberwithin{equation}{section}

\usepackage{tikz}

\usepackage{algorithm}
\usepackage{algpseudocode}

\newtheorem{Thm}{Theorem}[subsection]

\newtheorem{Prop}[Thm]{Proposition}
\newtheorem{Cor}[Thm]{Corollary}

\newtheorem{Rmk}[Thm]{Remark}

\newtheorem*{Thm*}{Theorem}

\crefname{Thm}{Theorem}{Theorems}
\crefname{Lem}{Lemma}{Lemmas}
\crefname{Prop}{Proposition}{Propositions}
\crefname{Rmk}{Remark}{Remarks}
\Crefname{Thm}{Theorem}{Theorems}
\Crefname{Lem}{Lemma}{Lemmas}
\Crefname{Prop}{Proposition}{Propositions}
\Crefname{Rmk}{Remark}{Remarks}

\newcommand{\veps}{\varepsilon}

\usepackage{cancel}

\newcommand{\tv}{\tilde{v}}

\newcommand{\mut}{\tilde{\mu}}

\usepackage[backgroundcolor=white]{todonotes}
\usepackage{amssymb}

\usepackage{multicol}
\usepackage[normalem]{ulem}

\usepackage{placeins}
\usepackage{afterpage}
\usepackage{cite}

\graphicspath{{figures/BFN_variants}}

\title{A critical note on back-and-forth Data Assimilation Nudging Algorithm}

\author{Aseel Farhat}
\address{(Aseel Farhat) Department of Mathematics, University of Virginia, Charlottesville, VA 22904, USA}
\email{af7py@virginia.edu}

\author{Edriss S. Titi}
 \address{(Edriss S. Titi) Department of Applied Mathematics and Theoretical Physics, University of Cambridge, Cambridge CB3 0WA, UK; Department of Mathematics, Texas A\&M University, College Station, TX 77843, USA; and Department of Computer Science and Applied Mathematics, Weizmann Institute of Science, Rehovot, 76100, Israel}
\email{edriss.titi@maths.cam.ac.uk, and titi@math.tamu.edu}  

\author{Collin Victor}
\address{(Collin Victor) Department of Mathematics, Texas A\&M University, College Station, TX 77843, USA}
\email{collin.victor@tamu.edu}
\date{April 21, 2026}

\begin{document}

\begin{abstract}
This work investigates the effectiveness of the Back-and-Forth Nudging (BFN) data assimilation algorithm, specifically its performance when employing the Azouani-Olson-Titi (AOT) continuous data assimilation downscaling nudging algorithm, for recovering initial conditions of dissipative dynamical systems. Contrary to previous reports in the literature, we show that, for several systems of interest, one can construct initial conditions that BFN cannot reliably recover.

Our key finding is the construction of infinitely many distinct solutions for certain dissipative systems that share identical spatially sparse observational data. Since these observations are indistinguishable, no data assimilation method relying only on them can differentiate between these solutions or recover the correct initial condition. We illustrate these pathological initial conditions for the Lorenz 1963 model and several 1D partial differential equations: the heat equation, viscous linear transport, and viscous and inviscid Burgers equations. Our analytical results are supported by numerical simulations.

To address the numerical ill-posedness of backward-in-time iterations, an essential step of BFN for dissipative models, we introduce a regularized backward step, the Voigt-regularized BFN (VBFN). We investigate its performance for the 2D Navier-Stokes equations and the viscous 1D KdV equation, comparing it with standard BFN and Diffusive BFN (dBFN). While VBFN improves numerical stability and reduces model bias relative to dBFN, it still cannot reconstruct the unobserved fine spatial scales of the reference solution. This reinforces our main conclusion: even with regularization, BFN-type algorithms are limited in recovering the full state, and in particular the initial data, from sparse spatial observations.
\end{abstract}


\maketitle

\begin{center}
\em In memory of Professor Peter Lax
\end{center} 

\section{Introduction}

Accurate computational prediction of nonlinear evolution systems encounters two major challenges. That of initialization and of long-time accuracy, namely, how do we initialize the system and how do we ensure that our computed solution remains accurate for long intervals of time. Data assimilation is a class of techniques that attempts to resolve both of these issues by incorporating coarse spatial observational data over time into the underlying evolutionary physical model. The purpose of data assimilation is to alleviate the reliance on knowing the complete state of the initial data and achieve synchronization of our simulated solution with the true dynamics by utilizing adequately collected sparse spatio-temporal  observational data (measurements). This study focuses on the former issue, specifically on using the Back-and-Forth Nudging (BFN) data assimilation algorithm to recover initial conditions from sparse spatial observations over a finite time interval.

Our method adapts the Azouani-Olson-Titi (AOT) continuous data assimilation nudging algorithm given in \cite{Azouani_Olson_Titi_2014, Azouani_Titi_2014} to the iterative BFN framework. The BFN approach \cite{auroux2003etude, Auroux_Bansart_Blum_2013, auroux2005back, Auroux_Blum_2008, Auroux_Blum_Ruggiero_2016, Auroux_Nodet_2012, Peng_Wu_Shiue_2023} attempts to recover the true initial condition of an unknown reference solution by iteratively nudging the forward-in-time and backward-in-time equations of a specified dynamical system toward spatially sparse observational data collected over a finite time interval. However, this study concludes that the BFN AOT nudging algorithm is ill-suited for this task. We revisit and investigate the efficacy of the BFN algorithm on some of the models investigated earlier in \cite{auroux2003etude, Auroux_Bansart_Blum_2013, auroux2005back, Auroux_Blum_2008, Auroux_Blum_Ruggiero_2016, Auroux_Nodet_2012, Peng_Wu_Shiue_2023}, namely, the Lorenz 1963 model and several 1D PDEs, specifically the heat equation, the inviscid/viscous linear transport equation, and the inviscid/viscous Burgers equations.

We will now describe the general framework of the BFN nudging algorithm that we consider in this work.
Consider a generic dynamical system governed by the evolution equation
\begin{align}\label{eq:generic system}
    \begin{cases}u_t = F(u),\\
    u(0,x) = u_0(x).
    \end{cases}
\end{align}
Here, the mapping $F$ is a potentially nonlocal and nonlinear differential operator. In the literature on forward nudging, such systems are typically assumed to be dissipative with a finite-dimensional global attractor and finitely many determining parameters—such as determining modes, nodal values, and local spatial averages (see, e.g., \cite{C_O_T, F_M_R_T, Foias_Prodi,Foias_Temam_2, Foias_Titi, Holst_Titi, Jones_Titi, Jones_Titi_2} and the references therein). Roughly speaking, dissipation, in this context, means that the dynamical system has a compact absorbing set which trajectories enter after a finite amount of time and remain within thereafter. We will assume this equation is well-posed on the time interval $[0,T]$, a condition that is particularly necessary for the inviscid Burgers equation, which is known to develop shocks in finite time. While the initial condition is included in the model, its value is assumed to be unknown, as our goal is to use the BFN algorithm to recover the true initial state from partial observations of a reference solution to the system.

The BFN algorithm as described, for example in \cite{auroux2003etude, Auroux_Blum_2008, Auroux_Nodet_2012}, is given as follows:
\begin{align} (F) &
\begin{cases}
        v_t^n = F(v^n) + \mu I_h(u - v^n),\\
        v^n(0, x) = \tv^{n-1}(0, x),
\end{cases}\label{eq:BFN Forward}
\\
(B) &
\begin{cases}
        \tilde{v}_t^n = F(\tilde{v}^n) - \mut I_h(u - \tilde{v}^n),\\
         \tv^n(T, x) = v^{n}(T, x),\\   
\end{cases} \label{eq:BFN Backward}
\end{align}
for $t\in [0,T]$. Here $\mu, \mut \geq 0$ are the nudging parameters which control the strength of the feedback-control nudging term $I_h(u-v)$. Here the true solution $u$ is observed at the observational length scale $h$, and then the observed data is interpolated to the entire spatial domain via a spatial approximate identity interpolant operator $I_h$. 
The back-and-forth nudging algorithm is so named due to the iteration between the ``forward'' equation, \cref{eq:BFN Forward}, and the ``backward'' equation, \cref{eq:BFN Backward}, which we denote with (F) and (B), respectively, throughout this work.
It is worth noting that \cref{eq:BFN Backward} and \cref{eq:BFN Forward} are equivalent up to a sign change on the nudging term and the initial/terminal condition.
This sign change arises due to the fact that \cref{eq:BFN Backward} has a specified terminal condition and is nudged backwards-in-time.
We note that the forward equation \cref{eq:BFN Forward} and the backward equation \cref{eq:BFN Backward} are the equations for the AOT nudging algorithm (up to a sign change on the nudging term), which we will refer to generically as nudging. 
It should be noted, however, that, in the context of classical methods of data assimilation, nudging refers to a method developed by Hoke and Anthes \cite{Hoke_Anthes_1976_MWR}. This classical method of nudging is superficially similar to \cref{eq:BFN Forward} but lacks the spatial interpolant $I_h$ which has crucial implications for recovery of solutions to dynamical systems.  We note that the BFN algorithm does require us to choose $\tv^0(0)$ in order to initialize the algorithm, which we will typically choose to be identically zero. 
The goal of the algorithm is to improve an estimate for the initial state of $u$  by iteratively assimilating observations both forward and backwards in time.
As we will see in the later sections, this is not the case, as the BFN algorithm is unable to recover transient states that are under-resolved at the observational length scale. 
In fact, as we will review below, the only reported successful cases of the BFN algorithm occur when the reference solutions are fully observed at every point in space and time. In such instances, the supposedly unknown reference solution is already known and given; consequently, using a data assimilation algorithm to recover it is redundant and of little practical utility.

The convergence of the BFN algorithm \cref{eq:BFN Forward}-\cref{eq:BFN Backward} was proven for a linear model with full observations (at every spatial point of the domain) of the exact solution in \cite{auroux2005back}. Subsequently, its numerical convergence with discrete (perfect or noisy) data was successfully demonstrated on the Lorenz model, the viscous Burgers equation, and the layered quasi-geophysics model \cite{Auroux_Blum_2008}.

The theoretical convergence properties of the BFN algorithm with perfect observations (no error) for transport equations, whether viscous or inviscid, were explored in \cite{Auroux_Nodet_2012}. Three cases were considered: 1) full observation for any point $x$, in the spatial domain, and any time $t$, in $[0,T]$, for the exact solution $u$ ; 2) observations on a time window $[t_1, t_2]$ for any point $x$, in the spatial domain, of the exact solution $u$; and 3) observations at any time $t$, in $[0,T]$, over a given subdomain in $x$ of the exact solution $u$. In the linear viscous case, it was shown that the BFN algorithm converges with an exponentially decreasing error if the solution $u$ is fully observed in at each spatial point on some time subinterval $[a,b] \subseteq [0,T]$. Conversely, the backward equation is ill-posed, even in the distribution sense, when the system is partially observed in space $x$. The proof for this was provided for the special case of linear transport with constant transport speed and when the system is observed on a compact subdomain in space $x$. In the inviscid linear case, the algorithm was shown to converge in all three cases when the system is partially observed in subdomains in space or time.

The BFN algorithm was also investigated for the viscous and inviscid Burgers equations in \cite{Auroux_Nodet_2012}. For the viscous case, the backward equation is shown to be ill-posed due to the presence of the viscosity term. For the inviscid case, and for a time $T$ where there is no shock in the interval $[0,T]$, the BFN algorithm was shown to converge for a sufficiently large choice of the nudging parameter, which depends on bounds on the true solution of the Burgers equation. This convergence was demonstrated only for the case when the solution of the inviscid Burgers equation is fully observed in every $x$ in the domain, at least in a subinterval in time. 

A recent paper, \cite{Peng_Wu_Shiue_2023}, investigated the convergence of the BFN algorithm for the Lorenz 63 system both analytically and numerically. For the continuous-in-time case, where the system is observed fully in time, i.e., all the three components are observed or given, it was shown theoretically that BFN converges with a large enough nudging parameter, which depends on bounds of the true reference solution. For the discrete-in-time case, the convergence of the BFN algorithm was proven under an additional assumption: that the assimilated time step is small enough, also depending on bounds of the true reference solution. Their numerical experiments showed that the assimilated time step significantly impacts the algorithm's accuracy. The inverse relationship between the assimilated time step and the nudging parameter suggests that for a smaller time step, the nudging parameter must be larger, which in turn leads to a larger error. The authors conjectured that nonlinear nudging might be a better choice in this scenario, leaving it as a topic for future study. A more recent work, \cite{Amraoui_2023}, theoretically and numerically investigated the convergence properties of the BFN algorithm for the quasi-geostrophic ocean model. It was shown that the algorithm will asymptotically converge to the true initial condition in the case of continuous-in-time, full observations of the exact solution $u$ at every point in the spatial domain. This result was verified numerically. It was also shown that numerical convergence is not always perfect for discrete-in-time observations, with or without noise. 

We will consider in particular the standard BFN algorithm \cref{eq:BFN Forward}-\cref{eq:BFN Backward} to illustrate its shortcomings, which can make it unsuitable for recovering solutions to certain dynamical systems—even those for which the forward AOT algorithm is successful. Based on a review of the literature, the BFN algorithm appears to perform and converge effectively when provided with full observations of the exact solution at every point in space and time. However, in such cases, the initial condition is already known, making the use of a data assimilation algorithm to recover it of little practical utility and extraneous. Therefore, our discussion will focus on the more challenging and realistic scenarios of sparse spatial observations or incomplete data, where these algorithms are truly needed, and where they face significant issues that are common to all data assimilation methods. 

To highlight this limitation, we examine the BFN algorithm's application to a selection of relevant dynamical systems, including the Lorenz 1963 model and several 1D PDEs: the heat, inviscid/viscous linear transport, and inviscid/viscous Burgers equations. These classic examples were chosen to illustrate a key limitation for the community: the algorithm's performance is intrinsically tied to the system's observability, a condition often not met with sparse spatial data. For the sake of brevity, we have limited our discussion to these systems, although similar results have been found computationally for other evolution equations in higher dimensions. In summary, this note is a focused critique that uses a rigorous combination of analytical and numerical methods to expose a key limitation and shortcoming of the BFN data assimilation algorithm.

\subsection{Diffusive and Voigt Regularization of the BFN Algorithm}

The presence of diffusion in a model's equations generally makes the backward equation \cref{eq:BFN Backward} in the BFN algorithm ill-posed. However, in meteorology and oceanography, many models are theoretically diffusion-free, with diffusion terms added primarily to stabilize numerical integrations forward in time. Based on this, \cite{Auroux_Bansart_Blum_2013} proposed modifying the BFN algorithm to change the sign of the diffusion term during the backward nudging step to ensure stability. The main advantage of this new approach, known as the diffusive Back and Forth Nudging (dBFN) or BFN2, is that both the forward and backward equations are well-posed, whereas in the original BFN algorithm, the backward solution might not depend continuously on the forward solution.

The new dBFN algorithm was numerically tested on the viscous Burgers equation, both with and without shocks, using two scenarios: full, noise-free observations and sparse, noisy observations. The results showed that the dBFN algorithm consistently converged to a solution closer to the true initial condition (achieving a better relative error) compared to the standard BFN algorithm. It also converged in fewer iterations, was less sensitive to observation sparsity, and was more robust to the length of the assimilation time window. While it may not always converge to the exact initial condition, it often provides the best initial condition estimates compared to other data assimilation schemes. The dBFN algorithm was later investigated numerically in \cite{Auroux_Blum_Ruggiero_2016} for a 2D shallow water model and a 3D primitive equation ocean model, with similar conclusions reached.

At this point, one might consider other modifications of the algorithm that regularize the backward-in-time iteration step, thereby increasing numerical stability. One such natural modification is \textit{Voigt regularization}, which we describe here in the context of the 2D Navier-Stokes equations and the viscous 1D Korteweg-de Vries (KdV) equation.

The incompressible 2D Navier-Stokes equations (NSE) describe the motion of a viscous fluid and are given by:
\begin{equation}\label{eq:NSE}
u_t - \nu \Delta u + (u \cdot \nabla) u + \nabla p = f, \quad \nabla \cdot u = 0,
\end{equation}
where $u(x,t)$ is the unknown velocity field, $p$ is the unknown pressure, and $\nu > 0$ is the kinematic viscosity which is given. While globally well-posed in two dimensions, the numerical simulation of the 2D NSE at high Reynolds numbers, e.g., small viscosity, remains challenging due to the cascade of energy to small scales. This necessitates robust regularization techniques for both theoretical analysis and computational efficiency.

The Navier-Stokes-Voigt (NSV) equations model Kelvin-Voigt viscoelastic fluids \cite{oskolkov1973} by incorporating the regularization term, $-\alpha^2 \Delta u_t$, to \cref{eq:NSE}:
\begin{equation}\label{eq:NSE:VOIGT}
u_t - \alpha^2 \Delta u_t - \nu \Delta u + (u \cdot \nabla) u + \nabla p = f, \quad \nabla \cdot u = 0.
\end{equation}
This regularization acts as a mollifier of the nonlinear term, providing global well-posedness even in 3D for any $\alpha > 0$ \cite{Cao_Lunasin_Titi_2006, kalantarov2009}. For any finite time interval $[0, T]$ where a smooth NSE solution exists, the trajectories of the NSV model converge strongly to those of the NSE as the regularization parameter $\alpha \to 0$. Furthermore, it has been rigorously established that the global attractors $\mathcal{A}_\alpha$ are upper semicontinuous to the NSE attractor, and the stationary statistical solutions (invariant measures) $\mu_\alpha$ of the NSV equations converge to those of the 3D NSE \cite{Ramos_Titi_2010}. Numerical investigations using the Sabra shell model have further demonstrated that for small values of $\alpha$ the NSV equations accurately capture the energy-containing scales and exhibit a multiscaling inertial range consistent with standard sub-grid turbulence theory \cite{Levantetal2010}.

A mathematically compelling feature of Voigt regularization is its behavior in backward-in-time problems. Unlike the standard NSE, which is classically ill-posed due to the exponential growth of high-frequency modes under reversed diffusion, the NSV model provides a robust framework for backward continuation. The regularization term $-\alpha^2 \Delta u_t$ effectively mollifies the nonlinear term and inhibits the forward cascade of energy \cite{Levantetal2010}. Crucially, the Voigt term ensures that the backward-in-time problem remains well-posed in the sense of Hadamard \cite{Cao_Lunasin_Titi_2006, kalantarov2009}. This property ensures continuous dependence on the final data in the $V$ (or $H^1$) norm, which has significant implications for state estimation and the reconstruction of historical flow data.

Voigt regularization has also been successfully generalized to other nonlinear dispersive equations, most notably the Korteweg-de Vries (KdV) equation. The resulting \textit{KdV-Voigt} equation, which is structurally related to the Benjamin–Bona–Mahony (BBM) equation \cite{BBM_1972}, is given by:
\begin{equation}
u_t - \alpha^2 u_{xxt} + \gamma u - \nu u_{xx} + u_{xxx} + uu_x = 0,
\end{equation}
where $\gamma, \nu \geq 0$. It is essential to distinguish between the purely dispersive KdV equation ($\gamma = \nu = 0$) and its damped ($\gamma > 0$) or viscous ($\nu > 0$) counterparts. 

In the inviscid, purely dispersive case, the backward-in-time evolution remains well-posed, although it is often plagued by numerical oscillations near solitary waves. By incorporating the term $-\alpha^2 u_{xxt}$, this model replaces the purely dispersive nature of the KdV equation with a regularized temporal structure. 
A primary result in this field is that the Voigt term acts as a low-pass filter on the time derivative, providing a smoothing effect that guarantees global regularity in higher-order Sobolev spaces and prevents the wave steepening that typically leads to numerical instabilities \cite{marquez2019, kalantarov2009}. 

However, the introduction of linear damping ($\gamma > 0$) or viscosity ($\nu > 0$) renders the backward step ill-posed or unstable due to the resulting anti-dissipative terms. In the damped-driven regime, reversing the sign of time turns the damping term into an energy-generating source that amplifies all modes uniformly. In the viscous regime, the effect is even more problematic; the resulting anti-diffusive heat operator causes high-wavenumber modes to grow exponentially. In these dissipative-dispersive settings, the Voigt regularization is necessary to bound the growth of these modes. Specifically, the operator $(I - \alpha^2 \partial_x^2)^{-1}$ acts to mollify the nonlinearity as well as the growth rates. It maps the $L^2$ growth of the anti-diffusion to a bounded evolution in $H^1$ and similarly tempers the anti-damping growth, providing a stable framework for the BFN backward pass.

Notably, while it has been established that the continuous data assimilation algorithm for the 2D NSE \cite{Azouani_Olson_Titi_2014} and the weakly damped KdV equation \cite{Jolly_Sadigov_Titi_2017} can recover the full reference solution at an exponential rate from a finite number of low Fourier modes, our results indicate that this asymptotic synchronization does not ensure the successful iterative recovery of initial conditions within the BFN framework. In this context, we investigate a regularized version of the algorithm for both the 2D NSE and 1D dispersive-dissipative models, including the viscous and damped KdV equations. For these systems, we utilize the standard equations for the forward steps but substitute the Navier-Stokes-Voigt and the regularized KdV-Voigt models for the backward-in-time iterations. This approach is designed to mitigate the inherent ill-posedness of the backward evolution.

Our analysis and numerical experiments demonstrate that while the BFN variant that uses the Voigt regularization for the backwards step remains stable and converges, it does so with a residual error floor directly related to the model discrepancy introduced by the Voigt regularization. To minimize this discrepancy, we also investigate ``filtered'' variants of the regularized schemes that restrict stabilization solely to the unobserved spectral modes. Although we do not pursue a full analytical treatment of the Voigt-regularized schemes in this work, we present numerical experiments in \cref{sect:NSE} and \cref{sect:KdV} comparing the performance of the standard BFN, its diffusive/damped modifications (dBFN), and the proposed Voigt-regularized algorithm (VBFN). 

The remainder of this paper is organized as follows. In \cref{sect:Lorenz}, we examine the applicability of the BFN algorithm to the Lorenz 1963 system and identify a specific class of solutions where the algorithm fails to recover the true initial state, particularly the unobserved third component of the Lorenz system, i.e., when provided with sparse observational data, in this case only part of the components are observed. In \cref{sect:1D}, we extend these results to show that analytical counterexamples can be constructed for select 1D PDEs, such as the heat, linear transport, and Burgers equations. These counterexamples prove that distinct initial conditions can produce identical spatial sparse observational data, thereby precluding the recovery of the true initial state.

In \cref{sect:NSE}, we introduce the VBFN for the 2D Navier-Stokes equations and present numerical experiments comparing its performance against standard and diffusive BFN schemes. We examine the ability of these methods to reconstruct the energy and enstrophy spectra across observed and unobserved scales. 

Similarly, \cref{sect:KdV} explores the application of the VBFN to the 1D viscous and damped KdV equations. We evaluate various stabilization strategies, contrasting the reversal of linear damping and viscosity signs with Voigt-Helmholtz smoothing. For the damped KdV system, we show that while damping reversal prevents blow-up, it introduces a significant bias in the observed modes. In the viscous KdV case, where standard BFN exhibits catastrophic growth across over 100 orders of magnitude, we demonstrate that VBFN successfully tames the instability. Notably, we show that while regularization and spectral filtering achieve numerical stability, they remain fundamentally unable to recover unobserved high-frequency wavemodes. Finally, we conclude with a summary and discussion of our results in \cref{sect:Conclusion}.

\section{The Lorenz 1963 Model}\label{sect:Lorenz}
Lorenz introduced in 1963 a simplified/reduced model of atmospheric convection \cite{Lorenz_1963}, and so it is of particular interest for studying methods of data assimilation, see, e.g., \cite{Hayden_Olson_Titi_2011,Blocher_Martinez_Olson_2018,Du_Shiue_2021,Hayden_2007,Pecora_Carroll_1990,Law_Shukla_Stuart_2014} and the references within.
The Lorenz 1963 system is given by:
\begin{align}
    \begin{cases}\label{eq:Lorenz}
    \dot{u_1} = -\sigma u_1 + \sigma u_2,\\
    \dot{u_2} = -\sigma u_1 - u_2 - u_1u_3, \\
    \dot{u_3} = -bu_3 + u_1u_2 - b(\rho+\sigma),\\
    (u_1(0),u_2(0),u_3(0)) = U_0.
\end{cases}
\end{align}
Here $\sigma>0$ is the Prandtl number, $\rho>0$ is the Rayleigh number, and $b >0$ is a geometric factor.
It is well-known that, for certain parameter choices (e.g., $\sigma  = 10$, $\rho = 28$, $b = 8/3$), that the Lorenz system has a chaotic global attractor \cite{Tucker_1999}. 

We note that in the remainder of this section we will be utilizing the Euclidean $\ell^2$ norm on $\mathbb{R}^3$ defined for $f = (f_1,f_2,f_3)\in \mathbb{R}^3$ as $\norm{f}_{\ell^2} = \left(f_1^2 + f_2^2 + f_3^2\right)^{\frac{1}{2}}$.

\subsection{BFN with Partial Components Observations for the whole interval of time} 
We recall that the BFN algorithm for the Lorenz system with observations of all the three components $u_1(t), u_2(t)$, and $u_3(t)$ of the exact reference solution of \cref{eq:Lorenz}, for $t\in [0,T]$, is given by:
\begin{align}
    (F) &
    \begin{cases}\label{eq:Lorenz-XYZ-Nudging-full-F}
    \dot{v}_1^n = -\sigma v_1^n + \sigma v_2^n + \mu_1(u_1 - v_1^n),\\
    \dot{v}_2^n = -\sigma v_1^n - v_2^n - v_1^nv_3^n + \mu_2(u_2 - v_2^n), \\
    \dot{v}_3^n = -bv_3^n + v_1^nv_2^n - b(\rho+\sigma) +  \mu_3(u_3 - v_3^n),\\
    \left(v_1^n(0),v_2^n(0),v_3^n(0)\right) = \tilde{v}_0^{n-1}:= \left(\tilde{v}_1^{n-1}(0),\tilde{v}_2^{n-1}(0),\tilde{v}_3^{n-1}(0) \right),
\end{cases}\\
(B) &
    \begin{cases}\label{eq:Lorenz-XYZ-Nudging-full-B}
    \dot{\tilde{v}}_1^n = -\sigma \tilde{v}_1^n + \sigma \tilde{v}_2^n - \mut_1(u_1 - \tilde{v}_1^n),\\
    \dot{\tilde{v}}_2^n = -\sigma \tilde{v}_1^n - \tilde{v}_2^n - \tilde{v}_1^n\tilde{v}_3^n - \mut_2(u_2 - \tilde{v}_2^n), \\
    \dot{\tilde{v}}_3^n = -b\tilde{v}_3^n + \tilde{v}_1^n\tilde{v}_2^n - b(\rho+\sigma) - \mut_3(u_3 - \tilde{v}_3^n),\\
    \left(\tilde{v}_1^n(T),\tilde{v}_2^n(T),\tilde{v}_3^n(T)\right) = v_T^{n-1}:= \left(v_1^{n-1}(T), {v}_2^{n-1}(T),v_3^{n-1}(T) \right),
\end{cases}
\end{align}
 for $n = 1,2,...$. Here $\mu_1,\mu_2, \mu_3, \mut_1, \mut_2, \mut_3 \in \mathbb{R}^{+}$ are the nudging coefficients, which modulate the strength of the feedback-control nudging terms. In this section, we will consider the following modified BFN algorithm in which we observe and nudge only part of the state variables (the first two components) in the Lorenz system:
\begin{align}
    (F) &
    \begin{cases}\label{eq:Lorenz-XY-Nudging-F}
    \dot{v}_1^n = -\sigma v_1^n + \sigma v_2^n + \mu_1(u_1 - v_1^n),\\
    \dot{v}_2^n = -\sigma v_1^n - v_2^n - v_1^nv_3^n + \mu_2(u_2 - v_2^n), \\
    \dot{v}_3^n = -bv_3^n + v_1^nv_2^n - b(\rho+\sigma),\\
    \left(v_1^n(0),v_2^n(0),v_3^n(0)\right) = \tilde{v}_0^{n-1}:= \left(\tilde{v}_1^{n-1}(0),\tilde{v}_2^{n-1}(0),\tilde{v}_3^{n-1}(0) \right),
\end{cases}\\
(B) &
    \begin{cases}\label{eq:Lorenz-XY-Nudging-B}
    \dot{\tilde{v}}_1^n = -\sigma \tilde{v}_1^n + \sigma \tilde{v}_2^n - \mut_1(u_1 - \tilde{v}_1^n),\\
    \dot{\tilde{v}}_2^n = -\sigma \tilde{v}_1^n - \tilde{v}_2^n - \tilde{v}_1^n\tilde{v}_3^n - \mut_2(u_2 - \tilde{v}_2^n), \\
    \dot{\tilde{v}}_3^n = -b\tilde{v}_3^n + \tilde{v}_1^n\tilde{v}_2^n - b(\rho+\sigma),\\
    \left(\tilde{v}_1^n(T),\tilde{v}_2^n(T),\tilde{v}_3^n(T)\right) = v_T^{n-1}:= \left(v_1^{n-1}(T), {v}_2^{n-1}(T),v_3^{n-1}(T) \right),
\end{cases}
\end{align}
 for $n = 1,2,...$. 
The BFN algorithm is initialized with $\tilde{v}_0^0 \in \mathbb{R}^3$. While initialization for the forward AOT algorithm is generally arbitrary, we utilize observations of the first two components of $u$ and set $\tilde{v}_0^0 = (u_1(0), u_2(0), z_0)$ for some $z_0 \in \mathbb{R}$. The nudging coefficients $\mu_1, \mu_2, \tilde{\mu}_1, \tilde{\mu}_2 \in \mathbb{R}^{+}$ modulate the strength of the feedback-control nudging terms. In \cref{eq:Lorenz-XY-Nudging-B} we assume continuous observations of $u_1(t)$ and $u_2(t)$ (from \cref{eq:Lorenz}) are available throughout the interval $[0, T]$.

Previous studies establish that observing the first two components of the solution, $u$, to \cref{eq:Lorenz} is sufficient for the asymptotic recovery in time of the true solution. Although the algorithms in \cite{Hayden_2007, Pecora_Carroll_1990} directly enforced equality in the observed components rather than using AOT-style nudging, analogous results for the AOT algorithm have been demonstrated in \cite{Hayden_Olson_Titi_2011, Peng_Wu_Shiue_2023, Titi_Victor_2025}.

For clarity, the iteration index $n$ is omitted in the subsequent discussion of the first BFN iteration. The following results summarize the properties of \cref{eq:Lorenz,eq:Lorenz-XY-Nudging-B,eq:Lorenz-XY-Nudging-F}.
\begin{Thm}\label{thm:Lorenz-uniqueness}
    For any initial data $U_0$, there exists a unique solution $u$ to \cref{eq:Lorenz} such that for all $t\geq 0$ we have
    \begin{align}
        \norm{u(t)}_{\ell^2}^2\leq e^{-2\alpha t}\norm{U_0}_{\ell^2}^2 + \frac{b(\rho+\sigma)^2}{2\alpha} \leq K^2.
    \end{align}
    Here $\alpha = \min\{\sigma, 1, \frac{b}{2}\}$ and $K^2 = \norm{U_0}_{\ell^2}^2 + \frac{b(\rho+\sigma)^2}{2\alpha} $.
\end{Thm}
The proof of this fact can be found in, e.g., \cite{Temam_1997_IDDS}.

\begin{Thm}
    Let $T>0$. For any initial $v(0)$ (or terminal $\tv(T)$) data in $\mathbb{R}^3$, both the forward (\cref{eq:Lorenz-XY-Nudging-F}) and backward (\cref{eq:Lorenz-XY-Nudging-B}) equations admit a unique solution in $C^1([0,T];\mathbb{R}^3)$. Moreover, the solution $v(t)$ of \cref{eq:Lorenz-XY-Nudging-F} satisfies the bound
    \begin{align}
        \norm{v(t)}_{\ell^2}^2 \leq e^{-ct}\norm{v(0)}_{\ell^2}^2 + \frac{b(\rho+\sigma)^2 + (\mu_1 + \mu_2) K^2}{c},
    \end{align}
    for $t\in [0,T]$, where $c = \min\{2\sigma + \mu_1, 2 + \mu_2, b\}$.\\
    Similarly, the solution $\tilde{v}(t)$ of \cref{eq:Lorenz-XY-Nudging-B} enjoys the bound
    \begin{align}
        \norm{\tilde{v}(t)}_{\ell^2}^2 \leq e^{\tilde{c}(T-t)}\norm{\tilde{v}(T)}_{\ell^2}^2 + \frac{b(\rho+\sigma)^2 + (\mut_1 + \mut_2) K^2}{\tilde{c}}e^{\tilde{c}T},
    \end{align}
    for $t\in [0,T]$, where $\tilde{c} = \max\{2\sigma - \mut_1, 2 - \mut_2, 3b\}$.
\end{Thm}
The existence and uniqueness of solutions for \cref{eq:Lorenz-XY-Nudging-B,eq:Lorenz-XY-Nudging-F} follow from classical ODE theory. Solution bounds are obtained via elementary energy estimates based on the analysis in \cite{Peng_Wu_Shiue_2023}.

The modified BFN algorithm, \cref{eq:Lorenz-XY-Nudging-F,eq:Lorenz-XY-Nudging-B}, was previously applied to the Lorenz 1963 system in \cite{Peng_Wu_Shiue_2023}, where it was shown to achieve synchronization of initial data for the observed nudged variables (the first two components):
$$\lim_{n\to \infty} \left( \left(\tv^n_1(0) - u_1(0)\right)^2 + \left(\tv^n_2(0) - u_2(0)\right)^2\right) = 0,$$
where $\tv^n$ is the solution to \cref{eq:Lorenz-XY-Nudging-B} for $n = 1,2,...$. Notably, the recovery of $u_3(0)$ was not addressed in \cite{Peng_Wu_Shiue_2023} when $u_3(t)$ is not observed. We demonstrate here that partial observations of only the first two components of $u(t)$, as presented in \cref{eq:Lorenz-XY-Nudging-F,eq:Lorenz-XY-Nudging-B}, are insufficient to recover certain initial conditions for \cref{eq:Lorenz}.

\subsubsection{Analytical Discussion on the Failure of BFN} 
To illustrate why back-and-forth nudging should not be expected to be effective for partial observations (only the first two components of $u(t)$) of the Lorenz 1963 system, we will first begin with some elementary energy estimates of the difference between the exact reference and nudged solutions to \cref{eq:Lorenz-XY-Nudging-F,eq:Lorenz-XY-Nudging-B}.

Let $\delta = u - v$, where $u = \left(u_1,u_2,u_3\right)$ and $v = \left(v_1,v_2,v_3\right)$ are solutions to \cref{eq:Lorenz} and \cref{eq:Lorenz-XY-Nudging-F} (for $n=1$), respectively.

We note that $\delta$ solves the following system:
\begin{align}\label{eq:Lorenz-delta}
\begin{cases}
    \dot{\delta}_1 &= -\sigma \delta_1 + \sigma \delta_2 - \mu_1 \delta_1,\\
    \dot{\delta}_2 &= -\sigma \delta_1 - \delta_2 - u_1u_3 + v_1v_3 - \mu_2 \delta_2,\\
    \dot{\delta}_3 &= -b \delta_3 + u_1u_2 - v_1v_2.
    \end{cases}
\end{align}
Taking the Euclidean $\ell^2$ inner product of \cref{eq:Lorenz-delta} with $\delta$, we obtain, for $\varepsilon > 0$ to be chosen later, the following energy estimates:
\begin{align*}
    \frac{1}{2}\frac{d}{dt} \norm{\delta}^2_{\ell^2} &= (-\sigma - \mu_1 )\delta_1^2 + (-1 - \mu_2 )\delta_2^2 - b \delta_3^2 -\delta_1\delta_2 u_3 + \delta_1 u_2 \delta_3,\\
    &\leq (-\sigma - \mu_1 )\delta_1^2 + (-1 - \mu_2 )\delta_2^2 - b \delta_3^2 + \frac{K^2}{2}\delta_1^2 + \frac{K^2}{2}\delta_2^2 + \frac{K^2}{4\varepsilon}\delta_1^2 + \varepsilon\delta_3^2,\\
    &\leq (-\sigma - \mu_1 + \frac{K^2}{2} + \frac{K^2}{4\varepsilon})\delta_1^2 + (-1 - \mu_2 + \frac{K^2}{2})\delta_2^2 + (\varepsilon - b)\delta_3^2.
\end{align*}
We note that, above, we have used Young's inequality and the bound $K^2$ as given in \cref{thm:Lorenz-uniqueness}.
Now, for simplicity, we will assume the following lower bounds on $\mu_1$ and $\mu_2$:
\begin{align}\label{ineq:mu requirements}
\begin{cases}\mu_1 \geq b  + \frac{K^2}{2} + \frac{K^2}{4\veps} - \veps - \sigma,\\
\mu_2 \geq b + \frac{K^2}{2} - \veps -1.
\end{cases}
\end{align}
These bounds ensure that the global error is dominated by the $\delta_3^2$ term. Applying Gr\"onwall's inequality, we obtain:
\begin{align}\label{ineq:forward decay}
    \norm{\delta(T)}_{\ell^2}^2 \leq \norm{\delta(0)}_{\ell^2}^2e^{2(\veps - b)T}.
\end{align}
Thus, for a fixed $0<\veps < b$, we obtain exponential (in time for large $T$) convergence in $\ell^2$ of $v(t)$ to $u(t)$, provided the conditions in \cref{ineq:mu requirements} are satisfied. This is a standard result for the AOT algorithm, which guarantees the recovery of the reference solution $u$ of \cref{eq:Lorenz} asymptotically in time, i.e., $\lim\limits_{T\to \infty} \norm{u(T) - v(T)}_{\ell^2} = 0$.

We note that for any fixed $0<\veps < b$, we can choose $\mu_1$ and $\mu_2$ that satisfy \cref{ineq:mu requirements}, but as we send $\veps \to 0$, $\mu_1$ becomes unbounded.
We emphasize that these particular energy estimates require us to utilize the dissipation in the third component in \cref{eq:Lorenz,eq:Lorenz-XY-Nudging-F} to absorb a portion of the nonlinearity, as the nudged variables cannot adequately do so.

We now turn to performing similar estimates on the backward nudging algorithm, \cref{eq:Lorenz-XY-Nudging-B}. Let $\tilde{\delta} = u - \tilde{v}$, where $\tilde{v}$ is the solution to \cref{eq:Lorenz-XY-Nudging-B} from the first backward iteration ($n=1$). Using the substitution $\tau = T-t$, we can derive the following energy estimates, similar to our previous analysis:
\begin{align*}
    \frac{1}{2}\frac{d}{d\tau} \norm{\tilde{\delta}}^2_{\ell^2} &= (\sigma - \mut_1 )\tilde{\delta}_1^2 + (1 - \mut_2 )\tilde{\delta}_2^2 + b \tilde{\delta}_3^2 +\tilde{\delta}_1\tilde{\delta}_2 u_3 - \tilde{\delta}_1 u_2 \tilde{\delta}_3,\\
    &\leq (\sigma - \mut_1 )\tilde{\delta}_1^2 + (1 - \mut_2 )\tilde{\delta}_2^2 + b \tilde{\delta}_3^2 + \frac{K^2}{2}\tilde{\delta}_1^2 + \frac{K^2}{2}\tilde{\delta}_2^2 + \frac{K^2}{4\eta}\tilde{\delta}_1^2 + \eta\tilde{\delta}_3^2,\\
    &\leq (\sigma - \mut_1 + \frac{K^2}{2} + \frac{K^2}{4\eta})\tilde{\delta}_1^2 + (1 - \mut_2 + \frac{K^2}{2})\tilde{\delta}_2^2 +(b+\eta) \tilde{\delta}_3^2,
\end{align*}
Now, we assume that $\mut_1$ and $\mut_2$ satisfy the following lower bounds for some fixed $\eta > 0 \in \mathbb{R}$ to be determined later:
\begin{align}\label{eq:Gronwall-Lorenz}
\begin{cases}
    \mut_1 \geq \sigma + \frac{K^2}{2} + \frac{K^2}{4\eta},\\
    \mut_2 \geq 1 + \frac{K^2}{2}.
\end{cases}
\end{align}
Thanks to Gr\"onwall's inequality \cref{eq:Gronwall-Lorenz} ensures that $\tilde{\delta}$ satisfies the following estimate:
\begin{align}\label{ineq:backward growth}
    \norm{\tilde{\delta}(0)}_{\ell^2}^2\leq \norm{\tilde{\delta}(T)}_{\ell^2
    }^2e^{2(b+\eta)T}.
\end{align}

Now, noting that if we initialize the backward nudging algorithm \cref{eq:Lorenz-XY-Nudging-B} such that $\tilde{\delta}(T) = \delta(T)$, then by virtue of \cref{ineq:backward growth,ineq:forward decay} we obtain:
\begin{align}
    \norm{\tilde{\delta}(0)}_{\ell^2}^2 \leq \norm{\delta(0)}_{\ell^2}^2e^{2\left(\veps + \eta\right) T}.
\end{align}
Denoting $\tilde{\delta}^n = \tilde{v}^n - u$ for $n=1, 2,...$, we can iterate the above argument to obtain the estimate:
\begin{align}\label{ineq:backward growth iterated}
    \norm{\tilde{\delta}^n(0)}_{\ell^2}^2 \leq \norm{\delta(0)}_{\ell^2}^2 e^{2\left(\veps + \eta\right)nT}.
\end{align}

For positive $\eta$ and $\varepsilon$, the right-hand side of \cref{ineq:backward growth iterated} diverges, suggesting that the error may not decay regardless of the choice of $\tilde{\mu}_1$ and $\tilde{\mu}_2$. As discussed, the analytical difficulty at the energy level arises from the nonlinear terms, particularly within the third components of \cref{eq:Lorenz-XY-Nudging-F,eq:Lorenz,eq:Lorenz-XY-Nudging-B}.

To analytically recover the initial condition of the reference solution, it is insufficient to simply absorb the nonlinear terms into the nudged quantities, as the viscous term in the third component remains. Successful recovery would require the nonlinear terms to have a favorable sign to counteract the energy growth contributed by the destabilizing viscous term during the backward-in-time iterations. This is problematic in cases where the nonlinear terms vanish entirely, rendering recovery of the true initial state impossible.

Therefore, consider, for example, a reference solution to \cref{eq:Lorenz} of the form $(0,0,ae^{-bt} - (\rho + \sigma))$ for some $a \in \mathbb{R}$. It is easily verified that for solutions of this form the nonlinear coupling between components drops out completely. When only the first two components are observed, and if the nudged solution \cref{eq:Lorenz-XY-Nudging-F} is initialized with $\tilde{v}_0^0 = (0,0,z_0)$ for some $z_0\in \mathbb{R}$, both the nudging and nonlinear terms vanish in the forward and backward equations. Consequently, for all $n\in \mathbb{N}$:
\begin{align}
    \norm{\tilde{\delta}^n(0)}_{\ell^2}^2 = \norm{\delta(0)}_{\ell^2}^2 = \abs{z_0 - a + (\rho + \sigma)}^2.
\end{align}
This demonstrates that the BFN algorithm fails to recover the initial condition in this regime.

Another argument that supports this conclusion is via another related algorithm in data assimilation, the synchronization algorithm (also referred to as direct insertion, see, e.g. \cite{Olson_Titi_2003}):
\begin{align}\label{eq:Lorenz-Synch}
    \begin{cases}
    w_1(t) = u_1(t),\\
    w_2(t) = u_2(t), \\
    \dot{w_3} = -bw_3 + w_1w_2 - b(\rho+\sigma),\\
    (w_1(0),w_2(0),w_3(0)) = (u_1(0),u_2(0),w_0),
    \end{cases}
\end{align}
where $u(t)$ is the exact reference solution to \cref{eq:Lorenz}.
Formally, \cref{eq:Lorenz-Synch} is the infinite-gain limit of both of the forward and backward nudging systems \cref{eq:Lorenz-XY-Nudging-F,eq:Lorenz-XY-Nudging-B}.
By this we mean that for any fixed time, $t \in [0,T]$, the solutions, $v^n(t;\mu)$, $ \tv^n(t;\mu)$, to \cref{eq:Lorenz-XY-Nudging-F,eq:Lorenz-XY-Nudging-B}, respectively, are expected to converge to the solution $w(t)$ of the synchronization algorithm, given by \cref{eq:Lorenz-Synch} as $\mu,\mut \to \infty$, provided the initial and terminal conditions are chosen to coincide for each algorithm. 
This can be justified rigorously by applying the argument present in \cite{Carlson_Farhat_Martinez_Victor_2026_ISYNC} to the Lorenz 1963 system, provided both the synchronization and nudging equations, given by \cref{eq:Lorenz-Synch} and \cref{eq:Lorenz-XY-Nudging-F,eq:Lorenz-XY-Nudging-B}, respectively, are initialized with the same initial data. We note that the backward nudging algorithm \cref{eq:Lorenz-XY-Nudging-B} converges to the time-reversed synchronization dynamics, provided that the terminal condition is chosen to coincide with the forward nudging solution at time $T$.
We use \cref{eq:Lorenz-Synch} as an idealized point of comparison for BFN and AOT algorithms in the noise-free setting, where large nudging gains are typically used in both analysis and applications.

The barrier to recovery of the true initial state can be made even more clear by considering the synchronization algorithm, \cref{eq:Lorenz-Synch}. It is immediately clear that the synchronization algorithm does not recover the true initial state of the reference solution, as both the ``forward-in-time'' and ``backward-in-time'' evolution equations are the exact same and given by \cref{eq:Lorenz-Synch}. This means that the effect of the backwards nudging algorithm simply undoes the forward nudging algorithm, in the case when the observations are directly inserted into the nudged solution at each point in time.
As a consequence of this, the difficulties that arise in the BFN algorithm cannot be overcome by inflating the nudging parameters, and must be addressed in a different manner.

\subsubsection{Numerical Results - The Lorenz System}\label{sect:Lorenz:computations}
We now demonstrate and confirm our theoretical findings numerically through simulations of the classical 1963 Lorenz system. The system is integrated using a standard explicit fourth-order Runge-Kutta scheme (RK4) for the Lorenz terms, while the nudging terms are implemented explicitly. Notably, the nudging term cannot be implemented numerically using multi-stage methods such as RK4, as this would require temporal interpolation of the observations.

We fix the time step at $\Delta t = 10^{-5}$ over the time interval $[0,1]$. Both the integration scheme and the value of $\Delta t$ were chosen to ensure stability in the backward step equations \cref{eq:Lorenz-XY-Nudging-B}. While the dissipative term provides stability during forward integration step, it induces exponential growth in the backward equations, which can amplify small discretization errors.

To evaluate the effectiveness of data assimilation, we employ a standard ``identical twin'' experimental design. In this framework, we first generate a reference solution, which is treated as the true state to be recovered using observations. In all simulations detailed in this section, we set $\mu_1 = \mu_2 = \tilde{\mu}_1 = \tilde{\mu}_2$, denoted simply by $\mu$. The error is plotted over ``iteration time" to illustrate its evolution across both forward and backward iterations. Specifically, the interval $[n, n+1)$ represents the error for the $n$-th iteration, where even and odd $n$ correspond to forward and backward passes of the BFN algorithm, respectively.

In our implementation of \cref{eq:Lorenz-XY-Nudging-B}, we observed a minor loss of accuracy when running the equations backward in time due to the dissipative nature of \cref{eq:Lorenz}. For the simulated durations, the discretization error in the backward time-stepping was approximately $10^{-12}$. While this exceeds the machine epsilon for double precision ($\epsilon \sim 10^{-16}$), it remains well within the requirements for this study. Since the first two components of the solution $u$ of \cref{eq:Lorenz} serves as observational data, we record the entire time series of observations prior to running the BFN algorithm rather than simulating \cref{eq:Lorenz,eq:Lorenz-XY-Nudging-B,eq:Lorenz-XY-Nudging-F} simultaneously.

We now turn to the specific numerical tests. We consider a reference solution to \cref{eq:Lorenz} of the form:
\begin{align}\label{eq:Lorenz bad solution}
u(t) = \begin{pmatrix}
0\\
0\\
ae^{-bt} - (\rho+\sigma)
\end{pmatrix},
\end{align}
where $a \in \mathbb{R}$. For simplicity, we set $a = \rho+\sigma$, yielding $u(0) = \vec{0}$. The nudged solution is initialized with:
\begin{align}
\tilde{v}_0^0 = \begin{pmatrix}
0\\0\\\varphi
\end{pmatrix},
\end{align}
where $\varphi \in \mathbb{R}$.

As shown in \cref{fig:Lorenz-invariant}, the error between the reference solution and the BFN-generated solution (of \cref{eq:Lorenz-XY-Nudging-F,eq:Lorenz-XY-Nudging-B}) fails to vanish, i.e., to recover the initial data $u(0)$, for various choices of $\varphi$. For these trials, we set $\mu = 1000$. Because the nudging term vanishes in this specific case, the error is independent of the values of $\mu$ and $\tilde{\mu}$. The initial error decays exponentially during the forward iteration but returns to its original magnitude during the backward step.

In contrast to the specific case in \cref{eq:Lorenz bad solution}, the BFN algorithm successfully recovers the initial condition when the first two components of the initial data are non-zero. The exponential convergence rate appears highly sensitive to the parameters $b$, $\rho$, and $\sigma$. These results are presented in \cref{fig:Lorenz-beta,fig:Lorenz-rho,fig:Lorenz-sigma}. For these trials, we fix $\mu = 100$, $u_0 = (20,30,40)^T$, and $\tilde{v}_0^0 = (30,40,50)^T$. As seen in \cref{fig:Lorenz-sigma}, the number of iterations required to reach a fixed precision level increases with $\sigma$, whereas it decreases as $\rho$ or $b$ increases, as shown in \cref{fig:Lorenz-rho,fig:Lorenz-beta}.

\begin{figure}[t]
\centering

\begin{subfigure}[t]{0.49\textwidth}
\centering
\includegraphics[width=.9\linewidth]{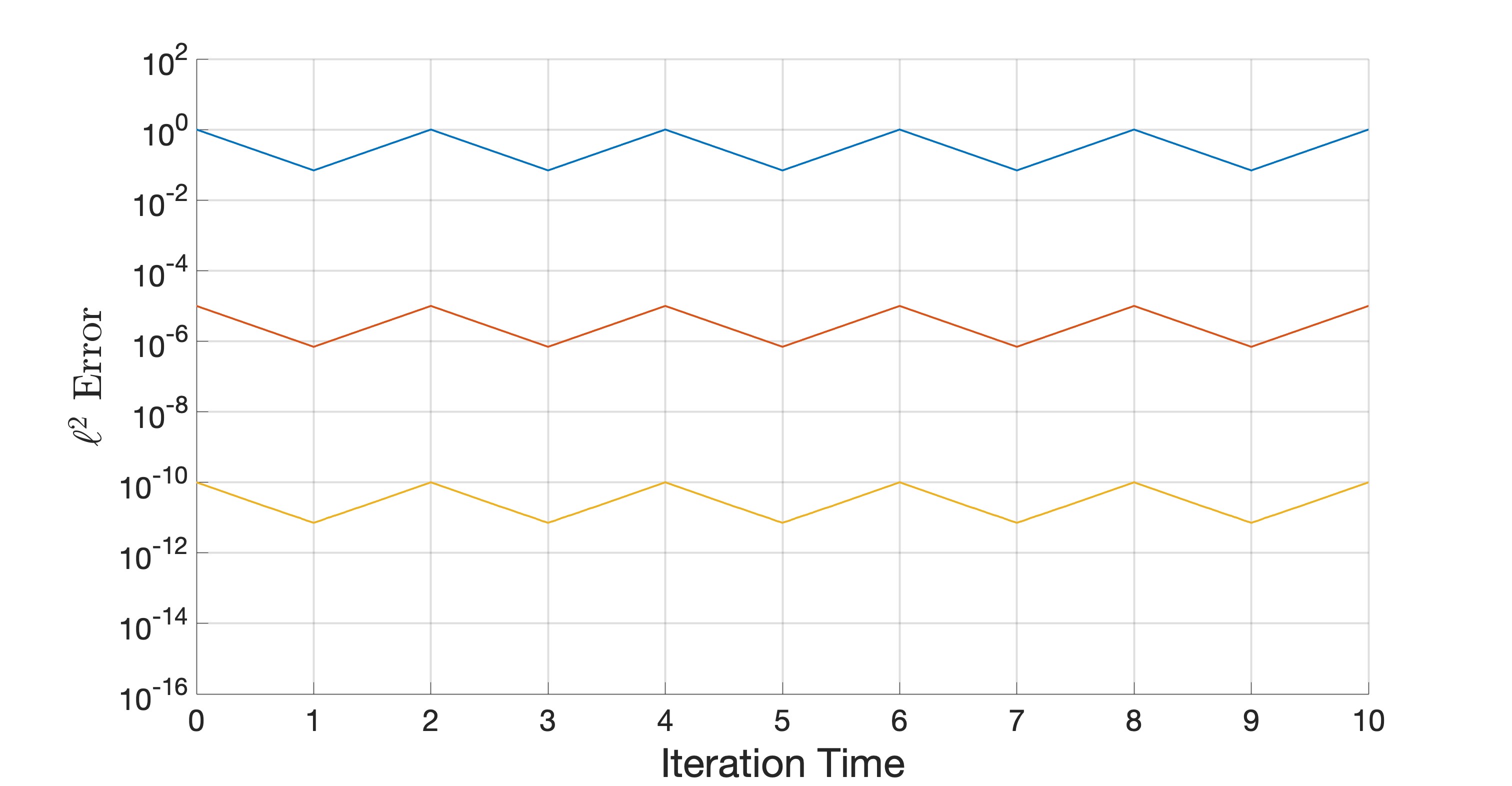}
\caption{\small Euclidean error for BFN with unobserved $u_3(0)$. Curves correspond to $\varphi=10^0$ (blue), $10^{-5}$ (red), and $10^{-10}$ (yellow).}
\label{fig:Lorenz-invariant}
\end{subfigure}
\hfill
\begin{subfigure}[t]{0.49\textwidth}
\centering
\includegraphics[width=.9\linewidth]{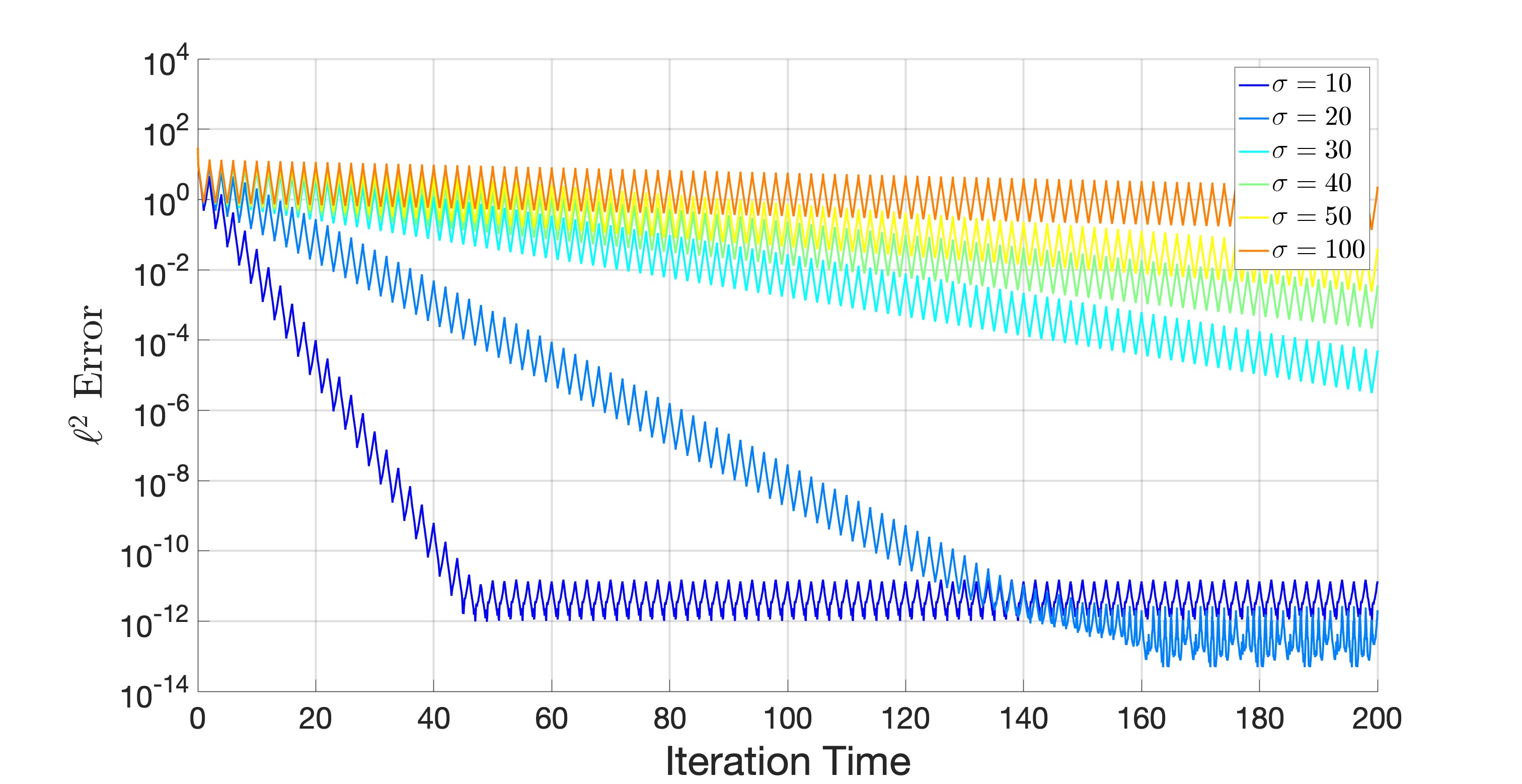}
\caption{\small Impact of $\sigma$ on error.}
\label{fig:Lorenz-sigma}
\end{subfigure}

\vspace{0.75em}

\begin{subfigure}[t]{0.49\textwidth}
\centering
\includegraphics[width=.9\linewidth]{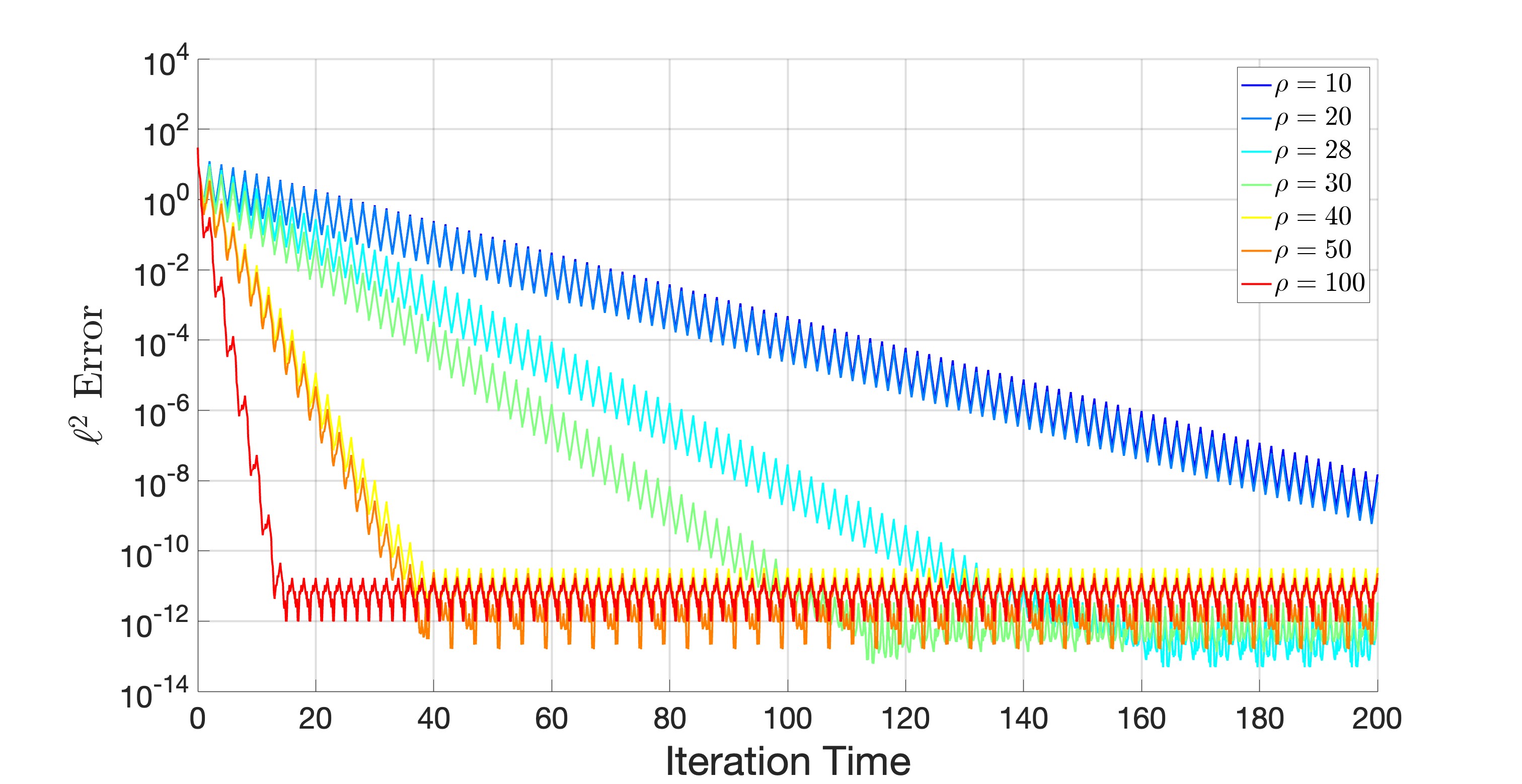}
\caption{\small Impact of $\rho$ on error.}
\label{fig:Lorenz-rho}
\end{subfigure}
\hfill
\begin{subfigure}[t]{0.49\textwidth}
\centering
\includegraphics[width=.9\linewidth]{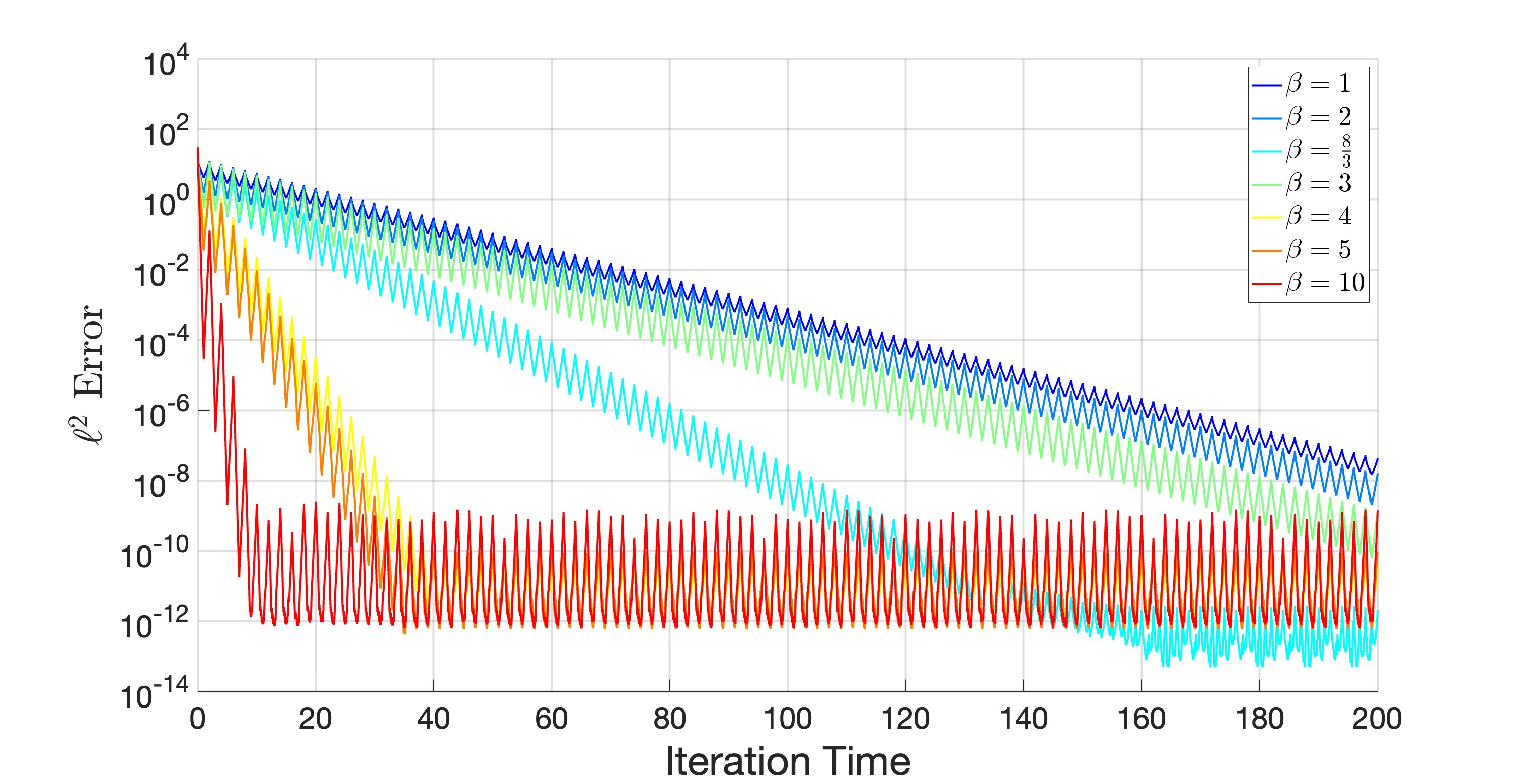}
\caption{\small Impact of $b$ on error.}
\label{fig:Lorenz-beta}
\end{subfigure}

\caption{\small BFN error behavior for the Lorenz 1963 system. }
\label{fig:Lorenz-panels}
\end{figure}

\subsection{BFN with Full Components Observations on Subintervals of Time}
We recall that \cite{Peng_Wu_Shiue_2023} investigated the convergence of the BFN algorithm for the Lorenz 1963 system. The authors demonstrated that for continuous, full-state observations, the algorithm in \cref{eq:Lorenz-XYZ-Nudging-full-F,eq:Lorenz-XYZ-Nudging-full-B} converges provided the nudging parameter is sufficiently large, depending on the bounds of the true solution. We note that in this full-observation setting, data assimilation is redundant to reconstruct the initial condition, as the complete state is already known for all $t\in [0,T]$ and in particular initially. Nevertheless, we follow the work of \cite{Peng_Wu_Shiue_2023} to verify our assertions and evaluate the BFN algorithm. In the discrete-time case, convergence was proven under the condition that the assimilation time step is sufficiently small; notably, a smaller time step required a larger nudging parameter, which may increase the overall error.

We now consider the case where components of the full true solution $u(t)$ of \cref{eq:Lorenz} is observed on different sub-intervals of the time interval $[0,T]$, an experimental setup similar to that in \cite{Peng_Wu_Shiue_2023}. We include this case to provide a comprehensive critique of the BFN algorithm. This regime is of particular interest because observational data in practical applications is often incomplete, featuring missing components or temporal gaps. Specifically, we study the forward and backward nudged systems:
\begin{align}
(F) &
\begin{cases}\label{eq:Lorenz-XYZ-Nudging-F}
\dot{v}_1^n = -\sigma v_1^n + \sigma v_2^n + \chi_{1,\gamma_{obs}}(t)\mu_1(u_1 - v_1^n),\\
\dot{v}_2^n = -\sigma v_1^n - v_2^n - v_1^n v_3^n + \chi_{2,\gamma_{obs}}(t)\mu_2(u_2 - v_2^n), \\
\dot{v}_3^n = -b v_3^n + v_1^n v_2^n - b(\rho+\sigma) + \chi_{3,\gamma_{obs}}(t)\mu_3(u_3 - v_3^n),\\
\left(v_1^n(0),v_2^n(0),v_3^n(0)\right) = \tilde{v}_0^{n-1},
\end{cases}\\
(B) &
\begin{cases}\label{eq:Lorenz-XYZ-Nudging-B}
\dot{\tilde{v}}_1^n = -\sigma \tilde{v}_1^n + \sigma \tilde{v}_2^n - \chi_{1,\gamma_{obs}}(t)\tilde{\mu}_1(u_1 - \tilde{v}_1^n),\\
\dot{\tilde{v}}_2^n = -\sigma \tilde{v}_1^n - \tilde{v}_2^n - \tilde{v}_1^n \tilde{v}_3^n - \chi_{2,\gamma_{obs}}(t)\tilde{\mu}_2(u_2 - \tilde{v}_2^n), \\
\dot{\tilde{v}}_3^n = -b \tilde{v}_3^n + \tilde{v}_1^n \tilde{v}_2^n - b(\rho+\sigma) - \chi_{3,\gamma_{obs}}(t)\tilde{\mu}_3(u_3 - \tilde{v}_3^n),\\
\left(\tilde{v}_1^n(T),\tilde{v}_2^n(T),\tilde{v}_3^n(T)\right) = v_T^{n-1},
\end{cases}
\end{align}
where $\chi_{i,\gamma_{obs}}$ is the indicator function of the time subinterval $\left[\tfrac{i-1}{3}, \tfrac{(i-1)+\gamma_{obs}}{3}\right)$ for $i = 1,2,3$. Here, $\gamma_{obs}\in[0,1]$ represents the fraction of each sub-interval over which observations are available. Specifically, the unit interval $[0,1]$ is divided into thirds: the first, second, and third components of $u(t)$ are assimilated in the windows $[0,\tfrac{1}{3})$, $[\tfrac{1}{3},\tfrac{2}{3})$, and $[\tfrac{2}{3},1)$, respectively. Assimilation is restricted to a portion of each sub-interval with a relative length of $\gamma_{obs}$. For example, if $\gamma_{obs}=\tfrac{1}{2}$, the components are assimilated on $[0,\tfrac{1}{6})$, $[\tfrac{2}{6},\tfrac{3}{6})$, and $[\tfrac{4}{6},\tfrac{5}{6})$. While the nudging coefficients $\mu_i$ and $\tilde{\mu}_i$ can be set independently, all simulations presented here use $\mu_i = \tilde{\mu}_i = \mu$ for $i = 1,2,3$ for a constant $\mu>0$.

The Euclidean error for $\gamma_{obs} = 1$ is shown in \cref{fig:Lorenz-discrete}. As illustrated, the initial condition is recovered for sufficiently large $\mu$. However, the recovery process requires more iterations or fails entirely for certain values of $\gamma_{obs}$, as shown in \cref{fig:Lorenz-discrete-gaps}. Specifically, for $\gamma_{obs} \in [0, \tfrac{1}{2}]$, the BFN algorithm \cref{eq:Lorenz-XY-Nudging-F,eq:Lorenz-XY-Nudging-B}fails to recover the initial condition or provide significant improvement over the initial guess $\tilde{v}_0^0$.

\begin{figure}[t]
\centering

\begin{subfigure}[t]{0.49\textwidth}
\centering
\includegraphics[width=.9\linewidth]{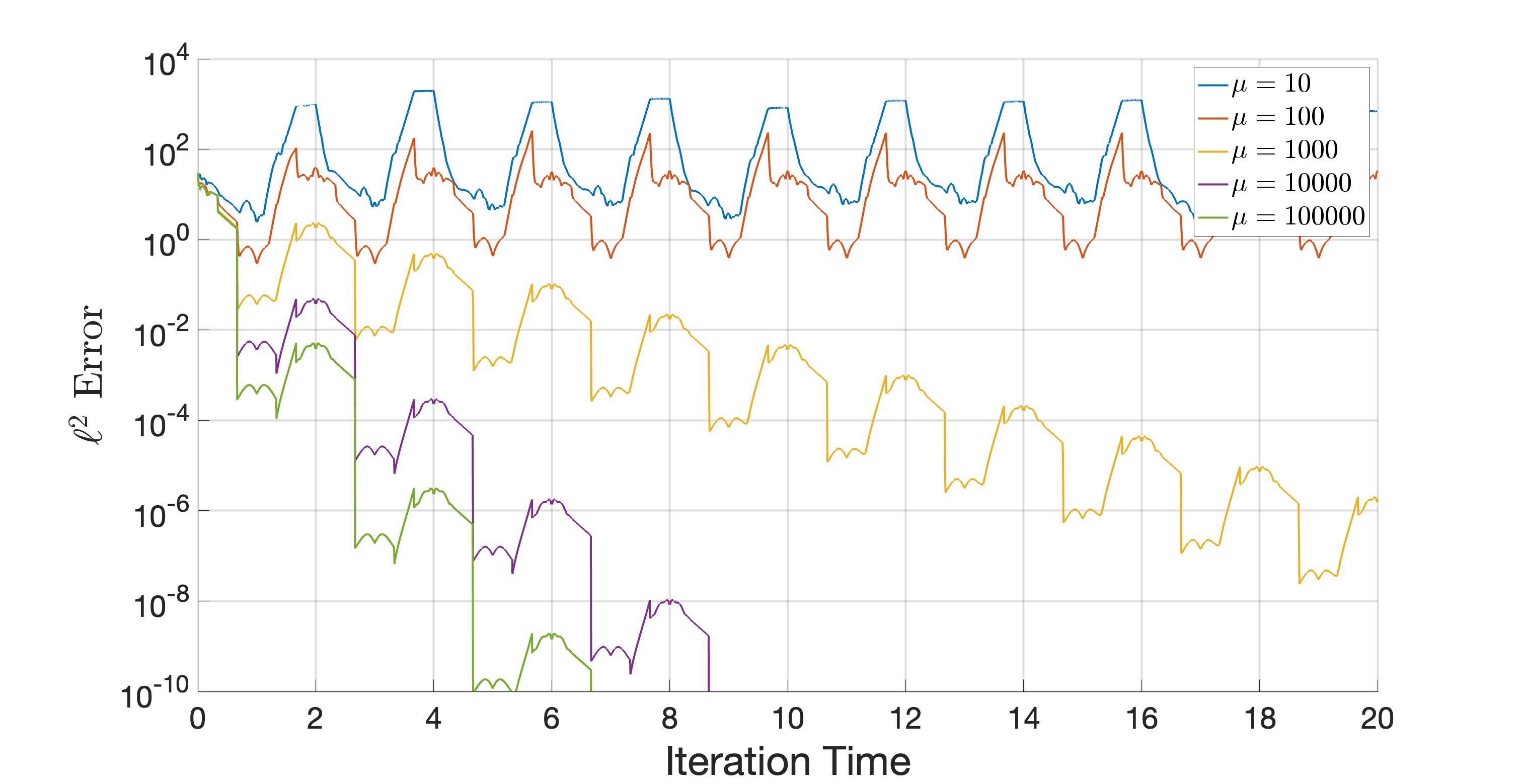}
\caption{\small Error for varying nudging strengths $\mu$ under full observation ($\gamma_{obs}=1$).}
\label{fig:Lorenz-discrete}
\end{subfigure}
\hfill
\begin{subfigure}[t]{0.49\textwidth}
\centering
\includegraphics[width=.9\linewidth]{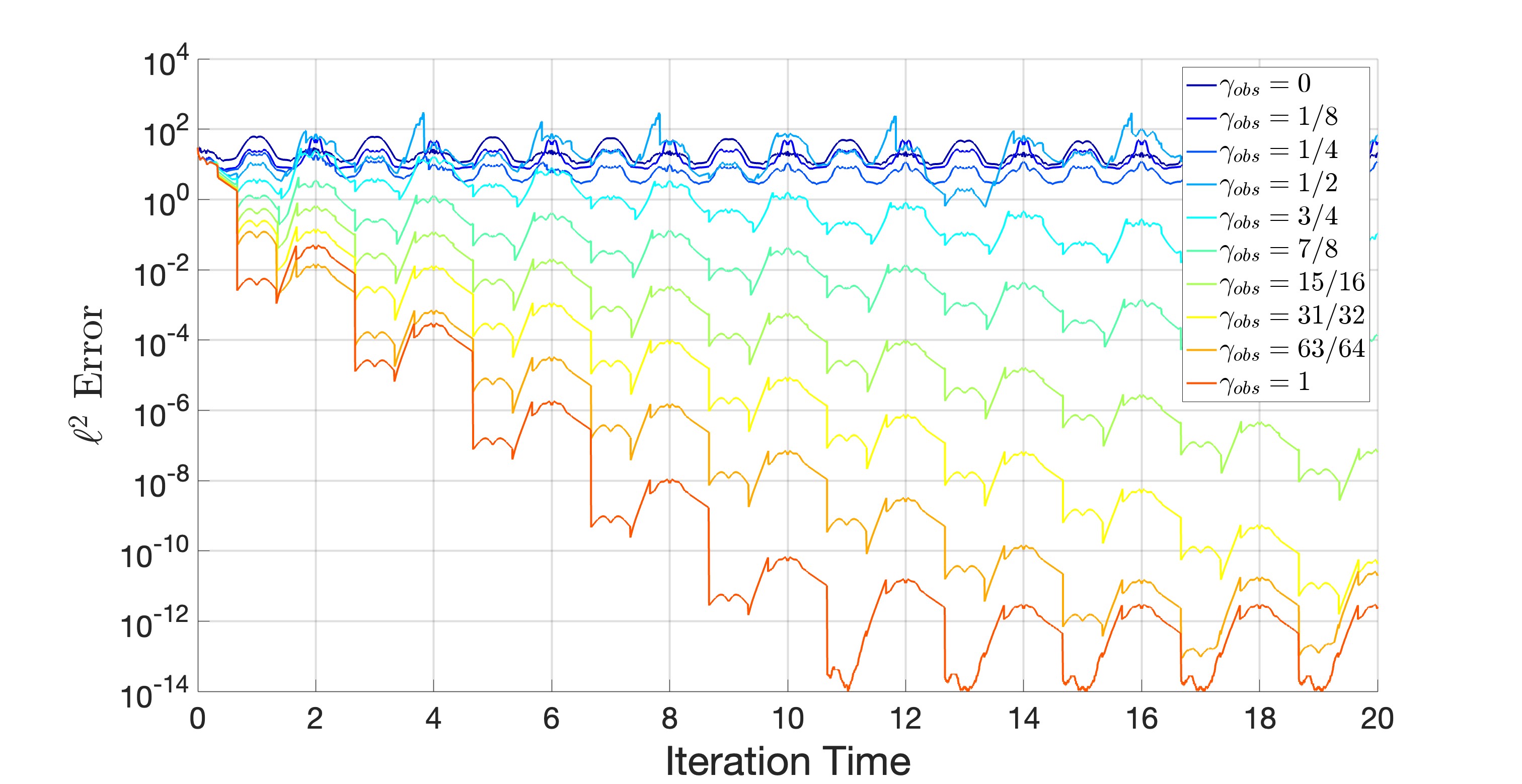}
\caption{\small Error for varying observation fractions $\gamma_{obs}$ with fixed $\mu=10{,}000$.}
\label{fig:Lorenz-discrete-gaps}
\end{subfigure}
\caption{\small BFN error for the Lorenz 1963 system.}
\label{fig:Lorenz-discrete-panels}
\end{figure}


\section{1D PDEs Paradigm}\label{sect:1D}
In this section, we outline specific counterexamples that demonstrate the inefficacy of the BFN algorithm when applied to partial observations of solutions to 1D PDEs. We first consider the 1D heat equation to illustrate the failure of the BFN algorithm in linear dissipative systems. We also examine the viscous transport and viscous Burgers equations, as well as their inviscid counterparts, which were the subject of previous analytical work in \cite{Auroux_Nodet_2012}.

The BFN algorithm and its diffusive variant, dBFN, have been studied in other contexts, such as shallow-water equations and various ocean models. We will consider the dBFN variant in \cref{sect:NSE} regarding the recovery of solutions to the 2D incompressible Navier-Stokes equations. However, the present discussion in this section focuses on the 1D models as they are most relevant to our core results.

Specifically, we show that the BFN algorithm fails to recover the initial condition for each of these equations when observations are restricted to a finite number of the first Fourier modes. We examine the linear and nonlinear cases separately in \cref{sect:linear,sect:heat} and \cref{sect:burgers}. Before presenting these results, we outline the notation and preliminaries used throughout the section.
\subsection{Preliminaries for 1D PDEs}\label{sect:1D:preliminaries}
Consider the generic evolution system \cref{eq:generic system}, where $F$ is a differential operator that may be nonlinear or nonlocal. We assume the system is well-posed on the time interval $[0,T]$ and in the spatial domain $\mathbb{R}$ subject to $L$-periodic boundary conditions ($L = 2\pi$) with initial condition $u_0$.

In this section, we utilize the Sobolev space $H^m_{\text{per}}$ for non-negative integers $m$, defined as the closure of $L$-periodic trigonometric polynomials under the norm:
\begin{equation}
\norm{f}_{H^{m}_{\text{per}}} := \left(\sum_{k = 0}^m \int_0^L \abs{\frac{\partial^k f}{\partial x^k} }^2 dx\right)^{\frac{1}{2}}.
\end{equation}
Note that $H^0_{\text{per}}$ corresponds to $L^2_{\text{per}}$, the space of $L$-periodic square-integrable functions. We assume solutions to \cref{eq:generic system} lie in $L^2_{\text{per}}$ for $t \in [0,T]$ and specifically examine the following operators for $F$:
\begin{align}
\text{Heat:}\quad & F(u) = \nu u_{xx} + f,\quad \nu>0, \quad f \in L^2_{\text{per}},\\
\text{Viscous Transport:}\quad & F(u) = \nu u_{xx} - a u_x,\quad a>0,\quad \nu\geq 0,\\
\text{Viscous Burgers:}\quad & F(u) = \nu u_{xx} - u u_x,\quad \nu\geq 0.
\end{align}
While the viscous cases are globally well-posed for all $T>0$, the inviscid case ($\nu = 0$) for the Burgers equation requires the \textit{a priori} assumption that the solution $u$ remains shock-free on $[0,T]$. Although the unforced heat equation is a special case of the transport equation ($a=0, \nu>0$), we treat it separately as an illustrative example of BFN failure for linear dissipative PDEs.

For the interpolation operator $I_h$ in \cref{eq:BFN Forward,eq:BFN Backward}, we set $I_h(g) := P_M g$, where $P_M$ is the projection onto the first $M$ Fourier modes of $g\in L^2_{\text{per}}$. Let $Q_M = I - P_M$ denote the orthogonal projection onto the remaining infinitely many Fourier modes. The BFN algorithm iteratively solves the following forward (F) and backward (B) nudging equations:
\begin{align}
\text{(F)}\quad &
\begin{cases}
v_t^n = F(v^n) + \mu P_M(u - v^n),\\
v^n(0) = \tilde{v}^{n-1}(0),
\end{cases}\label{eq:BFN_forward}\\[1ex]
\text{(B)}\quad &
\begin{cases}
\tilde{v}_t^n = F(\tilde{v}^n) - \mu P_M(u - \tilde{v}^n),\\
\tilde{v}^n(T) = v^n(T),
\end{cases}\label{eq:BFN_backward}
\end{align}
for $n = 1,2, \dots$, with the algorithm initialized by choosing $v^1(0) = v_0$. 

This iterative procedure aims to reconstruct $u(0)$ from partial (the first $M$) Fourier observations. A fundamental challenge of the BFN algorithm is the potential ill-posedness of the backward equation \cref{eq:BFN_backward}. For the purposes of this analysis, we assume the existence and uniqueness of solutions to \cref{eq:BFN_forward,eq:BFN_backward}. Finally, for $g \in L^2_{\text{per}}$, we denote the Fourier coefficients as:
\begin{align}
\hat{g}(k,t):= \int_{0}^L g(x,t) e^{-\frac{2\pi}{L}ik x}dx,
\end{align}
and denote the indicator function for the set $\{k\in \mathbb{Z}:\abs{k}\leq M\}$ by $\chi_M$.

\subsection{Heat Equation}\label{sect:heat}
We now apply the BFN algorithm to the heat equation \cref{eq:heat}. The reference system and the corresponding BFN equations are:
\begin{align}
&\begin{cases}\label{eq:heat}
    u_t = \nu u_{xx} + f,\\
    u(0) = u_0,
\end{cases}\\
\text{(F)}\quad &
\begin{cases}
    v_t^n = \nu v^n_{xx} + \mu P_M(u - v^n) + f,\\
    v^n(0) = \tilde{v}^{n-1}(0),
\end{cases}\label{eq:heat:F}\\
\text{(B)}\quad &
\begin{cases}
    \tilde{v}_t^n = \nu \tilde{v}^n_{xx} - \tilde{\mu} P_M(u - \tilde{v}^n) + f,\\
    \tilde{v}^n(T) = v^n(T),
\end{cases}\label{eq:heat:B}
\end{align}
for $n = 1,2, \dots$.

\begin{Prop}\label{prop:heat}
Let $u$ be the solution to \cref{eq:heat} with initial data $u_0$ given by a trigonometric polynomial of degree $K \in \mathbb{N}$, and let $M \in \mathbb{N}$. Then the sequence of solutions $v^n, \tilde{v}^n$ to \cref{eq:heat:F,eq:heat:B}, initialized with a trigonometric polynomial $v_0$ of degree $K$, satisfies:
\begin{align}
    \lim_{n\to\infty} \norm{v^n(0) - u_0}_{L^2_{\text{per}}} = \norm{Q_M (v_0 - u_0)}_{L^2_{\text{per}}}.
\end{align}
\end{Prop}

\begin{proof}
The result follows from the linearity of the heat equation, where Fourier modes evolve independently. Define the errors $w^n = u - v^n$ and $\tilde{w}^n = u - \tilde{v}^n$. In terms of Fourier coefficients, the error evolution for the first iteration is given by:
\begin{align}
\hat{w}^1(k,t) &= e^{\left(-\nu k^2 - \mu\chi_M(k)\right)t }\hat{w}^1(k,0),\\
\hat{\tilde{w}}^1(k,t) &= e^{\left(\nu k^2 - \tilde{\mu}\chi_M(k)\right)(T-t)}\hat{\tilde{w}}^1(k,T),
\end{align}
for $0 \le \abs{k} \le K$. Consequently, the error at the end of the first full iteration satisfies:
\begin{align}
    \hat{\tilde{w}}^1(k,0) = \begin{cases}
                    e^{-(\mu + \tilde{\mu})T}\hat{w}^1(k,0) &  0 \leq \abs{k}\leq M,\\
                    \hat{w}^1(k,0) & \abs{k} > M.
                    \end{cases}
\end{align}
Iterating this process $n$ times yields:
\begin{align}
    \hat{w}^n(k,0) = \begin{cases}
                    e^{-(\mu + \tilde{\mu})(n-1)T}\hat{w}^1(k,0) & 0 \leq \abs{k}\leq M,\\
                    \hat{w}^1(k,0) & \abs{k} > M.
                    \end{cases}
\end{align}
As $n \to \infty$, the error in the observed modes ($0\leq |k| \le M$) vanishes, while the error in the unobserved modes ($|k| > M$) remains constant, proving the proposition.
\end{proof}

\cref{prop:heat} demonstrates that the BFN algorithm, \cref{eq:heat:F,eq:heat:B}, fails to recover the initial condition if $u_0$ contains unobserved modes (i.e., $Q_M u_0 \neq 0$). While BFN succeeds if $P_M u_0 = u_0$, i.e., when $M\ge K$, the algorithm in this case is redundant as the initial condition is already known. Furthermore, if the algorithm, \cref{eq:heat:F,eq:heat:B}, is initialized with $v_0 = P_M u_0$, the error in the unobserved high-frequency Fourier modes remains invariant across all iterations. This failure is a direct consequence of the linearity of the heat equation: high-frequency modes evolve independently of observed low-frequency modes, rendering them fundamentally unrecoverable by any data assimilation method relying solely on $P_M u$.

\subsection{Linear Transport Equation}\label{sect:linear}
We now consider the linear transport equation \cref{eq:transport} with constant velocity $a > 0$ and viscosity $\nu \geq 0$ on a periodic 1D domain. The reference system and the corresponding BFN equations are:
\begin{align}
&\begin{cases}\label{eq:transport}
    u_t = \nu u_{xx} - a u_x,\\
    u(0) = u_0,
\end{cases}\\
\text{(F)}\quad &
\begin{cases}
    v_t^n = \nu v^n_{xx} - av^n_x  + \mu P_M(u - v^n),\\
    v^n(0) = \tilde{v}^{n-1}(0),
\end{cases}\label{eq:transport:F}\\
\text{(B)}\quad &
\begin{cases}
    \tilde{v}_t^n = \nu \tilde{v}^n_{xx} - a\tilde{v}^n_x - \tilde{\mu} P_M(u - \tilde{v}^n),\\
    \tilde{v}^n(T) = v^n(T),
\end{cases}\label{eq:transport:B}
\end{align}
where $n = 1,2,\dots$. Similar to the heat equation, the linear transport equation allows the Fourier modes to evolve independently, leading to an analogous application of the BFN algorithm.

\begin{Prop}\label{prop:transport}
Let $u$ be the solution to \cref{eq:transport} with initial data $u_0$ given by a trigonometric polynomial of degree $K \in \mathbb{N}$, and let $M \in \mathbb{N}$ such that $M \leq K$. The sequence of solutions $v^n, \tilde{v}^n$ to \cref{eq:transport:F,eq:transport:B}, initialized with a trigonometric polynomial $v_0$ of degree $K$, satisfies:
\begin{align}
    \lim_{n\to\infty} \norm{v^n(0) - u_0}_{L^2_{\text{per}}} = \norm{Q_M (v_0 - u_0)}_{L^2_{\text{per}}}.
\end{align}
\end{Prop}

\begin{proof}
Following the argument in the proof of \cref{prop:heat}, the Fourier coefficients of the error satisfy:
\begin{align}
\hat{w}(k,t) &= e^{\left(-\nu k^2 - \mu\chi_M(k) - i a k\right)t }\hat{w}(k,0),\\
\hat{\tilde{w}}(k,t) &= e^{\left(\nu k^2 - \tilde{\mu}\chi_M(k) + i a k\right)(T-t)}\hat{\tilde{w}}(k,T),
\end{align}
for $0\le \abs{k} \le K$. Iterating this result shows that the error in unobserved modes ($|k| > M$) remains unchanged across iterations.
\end{proof}

We conclude that the BFN algorithm, \cref{eq:transport:F,eq:transport:B}, fails to recover the initial condition for the linear transport equation if the initial data is not fully observed ($P_M u_0 \neq u_0$). Despite continuous-in-time noise-free observations, the algorithm neither recovers the full initial condition nor improves upon the initial observation of unobserved modes. 

This failure results from the lack of nonlinear interaction between Fourier modes, which prevents the algorithm from possibly correcting high-frequency components through observed low-frequency data. While the BFN algorithm, \cref{eq:transport:F,eq:transport:B}, recovers the initial condition when $P_M u_0 = u_0$, the BFN algorithm, \cref{eq:transport:F,eq:transport:B}, is redundant in this case as the initial state is already known. 

Notably, \cref{prop:transport} is consistent with the results in \cite{Auroux_Nodet_2012}. While that work demonstrates BFN effectiveness when characteristic curves are fully observed, our results show that the BFN algorithm, \cref{eq:transport:B,eq:transport:F} is unsuitable for recovering initial states from sparse-in-space observations (partial Fourier modes) of the linear transport equation.

\subsection{Viscous Burgers Equation}\label{sect:burgers}
We now consider the viscous Burgers equation, a nonlinear PDE of significant interest in data assimilation. The reference system and the corresponding BFN equations are:
\begin{align}
&\begin{cases}\label{eq:burgers}
u_t = \nu u_{xx} - uu_x,\\
u(0) = u_0,
\end{cases}\\
\text{(F)}\quad &
\begin{cases}
v_t^n = \nu v^n_{xx} - v^nv^n_x  + \mu P_M(u - v^n),\\
v^n(0) = \tilde{v}^{n-1}(0),
\end{cases}\label{eq:burgers:F}\\
\text{(B)}\quad &
\begin{cases}
\tilde{v}_t^n = \nu \tilde{v}^n_{xx} - \tilde{v}^n\tilde{v}^n_x - \mu P_M(u - \tilde{v}^n),\\
\tilde{v}^n(T) = v^n(T),
\end{cases}\label{eq:burgers:B}
\end{align}
for $n = 1,2, \dots$.

It was proven in \cite{Auroux_Nodet_2012} that the BFN, \cref{eq:burgers:B,eq:burgers:F}, algorithm successfully recovers the initial condition for the inviscid Burgers equation under full-state observations ($I_h(u) := u$). However, in practice, complete observations eliminate the need for data assimilation since the initial state is already known. Furthermore, the inclusion of viscosity renders the backward equation \cref{eq:burgers:B} ill-posed. While numerical implementations in \cite{Auroux_Nodet_2012} address this by modifying \cref{eq:burgers:B} through reversing the sign of viscosity in the backward step, such results still rely on the identity operator for observations, i.e., when $I_h(u) = u$. We focus instead on the application of BFN with partial observations restricted to low, or first $M$, Fourier modes.

In the previous sections, we showed that BFN algorithm fails for linear systems if the initial condition is not fully observed. Here, we demonstrate a family of initial conditions for which the BFN algorithm \cref{eq:burgers:F,eq:burgers:B}, fails for the Burgers equation. We utilize the following preliminary results, previously introduced in \cite{Titi_Victor_2025}:

\begin{Prop}\label{prop:1}
    Let $k\in \mathbb{Z}\setminus \{0\}$ and $\phi \in L^2_{\text{per}}$ with zero spatial mean.
    \begin{enumerate}[label=(\roman*)]
        \item $\phi(x + L/k) = \phi(x)$ for a.e. $x\in \mathbb{R}$ if and only if
        $\phi(x) = \sum_{n\in k\mathbb{Z}} \hat{\phi}_n e^{i \frac{2\pi}{L} n x}.$
        \item If $M > 0$ and $|k| \geq M$, then $\phi(x + L/k) = \phi(x)$ for a.e. $x\in \mathbb{R}$ implies $P_M \phi = 0$.
    \end{enumerate}
\end{Prop}

\begin{Prop}\label{prop:periodic invariance}
If $u_0 \in H^2_{\text{per}}$ is $L/k$-periodic and mean-free for an integer $k \neq 0$, then the corresponding solutions to \cref{eq:burgers} remain $L/k$-periodic for all $t \in [0,T]$.
\end{Prop}

\begin{Cor}\label{cor:periodic_nonuniqueness}
Two distinct solutions $u_1, u_2$ initialized with $L/k_1$ and $L/k_2$ periodic data, respectively, where $\min(|k_1|, |k_2|) > M$, satisfy $P_M u_1(t) = P_M u_2(t) = 0$ for all $t$, despite the fact that $u_1 \neq u_2$.
\end{Cor}

This result identifies a fundamental limitation of the BFN algorithm \cref{eq:burgers:F,eq:burgers:B}:

\begin{Prop}\label{prop:bfn_failure}
For any $M \in \mathbb{Z}_{\geq 0}$, there exist solutions to \cref{eq:burgers} that cannot be recovered by \cref{eq:burgers:F,eq:burgers:B} from arbitrary initial conditions.
\end{Prop}

\begin{proof}
Consider two distinct $L/k$-periodic solutions $u_1$ and $u_2$ of \cref{eq:burgers} with $|k| > M$. If $u_1$ is the reference solution and the algorithm, \cref{eq:burgers:F,eq:burgers:B}, is initialized such that $v^n = \tilde{v}^n = u_2$, the nudging terms $P_M(u_1 - u_2)$ vanish identically by \cref{prop:1}. Thus:
\begin{align}
\limsup_{n\to \infty} \norm{\tilde{v}^n(0) - u_1(0)}_{L^2_{\text{per}}} = \norm{u_2(0) - u_1(0)}_{L^2_{\text{per}}} > 0.
\end{align}
\end{proof}

The BFN algorithm, \cref{eq:burgers:F,eq:burgers:B}, fails to recover the initial condition for any $L/k$-periodic state with $|k| > M$. Such initial conditions are easily constructed using trigonometric polynomials involving only modes $|n| > M$. As demonstrated in \cref{sect:numerical}, this failure extends to more general cases where the initial condition contains unobserved Fourier modes. These results do not contradict \cite{Auroux_Nodet_2012}, as the spectral interpolant used here differs from the interpolant used there. Ultimately, these findings indicate that the initial condition cannot be uniquely determined from observing only a finite number of Fourier modes by any data assimilation method.

\subsection{Numerical Validation Tests - The Burgers equation case}\label{sect:numerical}

In this section, we detail the numerical methods employed to simulate the Burgers and linear transport equations, alongside their back-and-forth nudging (BFN) counterparts. For brevity, results for the heat equation are omitted, as they qualitatively mirror the findings of the linear transport simulations. All simulations were performed in MATLAB (R2024b) using custom-developed code.

The systems are defined on a periodic spatial domain $\Omega = [0, 2]$. We utilize pseudo-spectral methods, specifically an integrating factor technique, to compute the linear viscous and transport terms exactly. Time-stepping is handled via an explicit fourth-order Runge-Kutta (RK4) scheme, while the nonlinear and nudging terms are treated explicitly. For a comprehensive review of integrating factor schemes, we refer the reader to \cite{Kassam_Trefethen_2005, Trefethen_2000_spML}. The specific algorithmic implementation is summarized in \cref{alg:IFRK4} below.

A primary challenge in this regime is the ill-posedness of the backward-in-time equations for the Burgers and transport models in the presence of non-zero viscosity. While this does not affect the reference solution---which is recorded during a well-posed forward simulation---it necessitates regularization for the backward BFN pass. To mitigate the exponential growth of high wavenumbers induced by the anti-diffusive term, we regularize the equation by truncating the diffusion operator, retaining dissipative effects only on the lowest 50 Fourier modes.

To accurately resolve the nonlinear terms and prevent aliasing, we apply the $2/3$ dealiasing rule, truncating the upper third of the Fourier spectrum. In the spectral plots provided, a red vertical line indicates this dealiasing cutoff, while a blue vertical line denotes the observational window cutoff, representing the Fourier modes effectively observed and utilized in the nudging term.

\begin{algorithm}
    \caption{Fourth-Order Runge-Kutta with Integrating Factor for Nudging Systems}
    \label{alg:IFRK4}
\begin{algorithmic}[1] 
\Require
Observational operator $H := P_M$ (projection onto the lowest $M$ Fourier modes); \\
Nudging parameter $\mu$;\\
Nonlinear function $\mathcal{N}(\hat{u},t) := \frac{ik}{2} \mathcal{F}\left(\left(\mathcal{F}^{-1}(P_{\frac{2N}{3}}\hat{u})\right)^2\right)$, where $P_{\frac{2N}{3}}$ is the dealiasing projection;\\
Linear operator $L := -\nu k^2 + aik$ for the integrating factor ($\nu \geq 0, a \geq 0$);\\
$\hat{u}_n$ and $\hat{v}_n$, the Fourier-transformed solutions at the current time step.
\bigskip

\State \textbf{Reference state evolution:} $\hat{u}_{n+1} := M(\hat{u}_n)$
\begin{align*}
k_1 &= \mathcal{N}(\hat{u}_n,t_n)\\
k_2 &= \mathcal{N}(e^{-L\Delta t}\left(\hat{u}_n + \frac{\Delta t}{2}k_1\right),t_n + \frac{\Delta t}{2})\\
k_3 &= \mathcal{N}(e^{-L\Delta t}\left(\hat{u}_n + \frac{\Delta t}{2}k_2\right),t_n + \frac{\Delta t}{2})\\
k_4 &= \mathcal{N}(e^{-L\Delta t} \hat{u}_n + e^{-L\frac{\Delta t}{2}}\Delta t k_3, t_n + \Delta t)\\
\hat{u}_{n+1} &= e^{-L\Delta t}\hat{u}_{n} + \frac{\Delta t}{6}\left( e^{-L\Delta t}k_1 + 2e^{-L\frac{\Delta t}{2}}\left(k_2 + k_3\right) + k_4\right)
\end{align*}

\State \textbf{Nudged state evolution (Forward):}
\begin{align*}\hat{v}_{n+1} = M(\hat{v}_n) + e^{-L\Delta t}\mu \Delta t P_M\left( \hat{u}_n - \hat{v}_n\right)\end{align*}
\end{algorithmic}
\end{algorithm}

The nudging term is implemented numerically in \cref{alg:IFRK4} as an explicit term. This formulation of the feedback-control term induces a CFL-like stability constraint, $\mu \lesssim \frac{2}{\Delta t}$, which is necessary for the numerical scheme to remain stable. While this constraint could be circumvented by utilizing an implicit formulation of the nudging term, such an approach is not considered in the present work.

Notably, the nudging term cannot be directly implemented using multi-stage methods, such as RK4, because doing so would necessitate temporal interpolation of the observational data. For the simulations featured in this study, a nudging parameter of $\mu = 100$ was used, as it was found to be sufficient for our analysis. For an in-depth examination of how the nudging parameter influences the recovery of the reference solution, we refer the reader to \cite{Carlson_Farhat_Martinez_Victor_2026_ISYNC}.

\subsection{Computational Results}
In this section, we present numerical results for the viscous linear transport ($a=1$) and Burgers equations, alongside their inviscid counterparts. For these simulations, observational data consist of the first 16 Fourier modes ($M=16$). We report results for the first five iterations of the BFN algorithm, noting that extended simulations yielded qualitatively identical behavior. Error is plotted against ``iteration time'', illustrating the evolution across successive forward and backward passes. Specifically, the intervals $[n, n+0.5)$ and $[n+0.5, n+1)$ correspond to the $n$-th forward and backward BFN passes, respectively.

\subsubsection{BFN with Zero Observations} 
We first evaluate the algorithm's performance when observations are identically zero. We initialize \cref{eq:generic system} with 
\begin{align}
    u_0(x) = 0.05\cos(30\pi x), \quad \text{and} \quad \tilde{v}_0(x) \equiv 0.
\end{align}
These profiles are selected to numerically demonstrate the recovery failure predicted by \cref{prop:bfn_failure}. Since the observed modes of $u$ remain zero for all $t$, no information is available to reconstruct the initial condition. Consequently, the nudged solution $v$ remains identically zero throughout the simulation.

The results for this trial are shown in \cref{fig:error zero ob_LT,fig:spec zero ob_LT,fig:soln zero ob_LT,fig:error zero ob_B,fig:spec zero ob_B,fig:soln zero ob_B}. In the viscous case (\cref{fig:error zero ob:VLT}), the $L^2_{\text{per}}$ error appears to converge toward machine precision; however, this is strictly a consequence of physical dissipation. The failure to recover the initial state is evidenced by the fact that the error returns to its initial magnitude at the beginning of each backward BFN iteration.

\subsubsection{BFN with Full-State Observations}
We verify the efficacy of the BFN algorithm in the regime where the initial profile is fully resolved within the observational window. We set 
\begin{align}
    u_0(x) = \cos(\pi x), \quad \text{and} \quad \tilde{v}_0(x) \equiv 0.
\end{align}
The simulation time ($T = 0.0001$) is sufficiently short to ensure that the dynamics of the Burgers equation \cref{eq:burgers} do not excite wavenumbers beyond $M = 16$. Thus, the reference solution $u(x,t)$ is fully observed at every time step.

The results are presented in \cref{fig:error full ob_LT,fig:spec full ob_LT,fig:soln full ob_LT,fig:error full ob_B,fig:spec full ob_B,fig:soln full ob_B}. As expected, the BFN algorithm successfully recovers the initial condition, consistent with the analytical proof in previous literature. We observe an increase in error during the backward iterations of the Burgers equation (\cref{fig:error full ob:IB,fig:error full ob:VB}) due to the inherent numerical instability of reversing nonlinear flows. We consider the solutions converged once single precision ($10^{-8}$) is achieved. While precision could be further improved by refining the time step or employing higher-order discretizations, single precision is sufficient for the scope of this study.

\subsubsection{BFN with Low-Mode Observations}
Finally, we demonstrate that the BFN algorithm, \cref{eq:burgers:F,eq:burgers:B}, fails to recover high-frequency components for more general initial data. We set 
\begin{align}
    u_0(x) = \cos(\pi x) + 0.05\cos(30\pi x), \quad \text{and} \quad \tilde{v}_0(x) \equiv 0.
\end{align}
As shown in \cref{fig:spec partial ob_LT,fig:spec partial ob_B}, the reference solution exhibits significant activity in both the observed low modes and unobserved high modes.

The results in \cref{fig:error partial ob_LT,fig:spec partial ob_LT,fig:soln partial ob_LT,fig:error partial ob_B,fig:spec partial ob_B,fig:soln partial ob_B} reveal that the error converges to a non-zero floor. In the inviscid cases (\cref{fig:error partial ob:IB,fig:error partial ob:ILT}), this error remains constant, while in the viscous cases (\cref{fig:error partial ob:VB,fig:error partial ob:VLT}), it oscillates in alignment with the BFN iterations. During the forward pass, the error decays exponentially, which is consistent with established forward-nudging results. However, the spectral plots (\cref{fig:spec partial ob_LT,fig:spec partial ob_B}) explicitly show that while low modes are recovered, no frequencies beyond $M=16$ develop in the nudged solution. This confirms that the algorithm's failure is an intrinsic property under low-mode observations rather than a result of the specific cases used in \cref{prop:bfn_failure}.

\begin{figure}
    \centering
    \begin{subfigure}[b]{.49\textwidth}
        \centering
        \includegraphics[width=\textwidth]{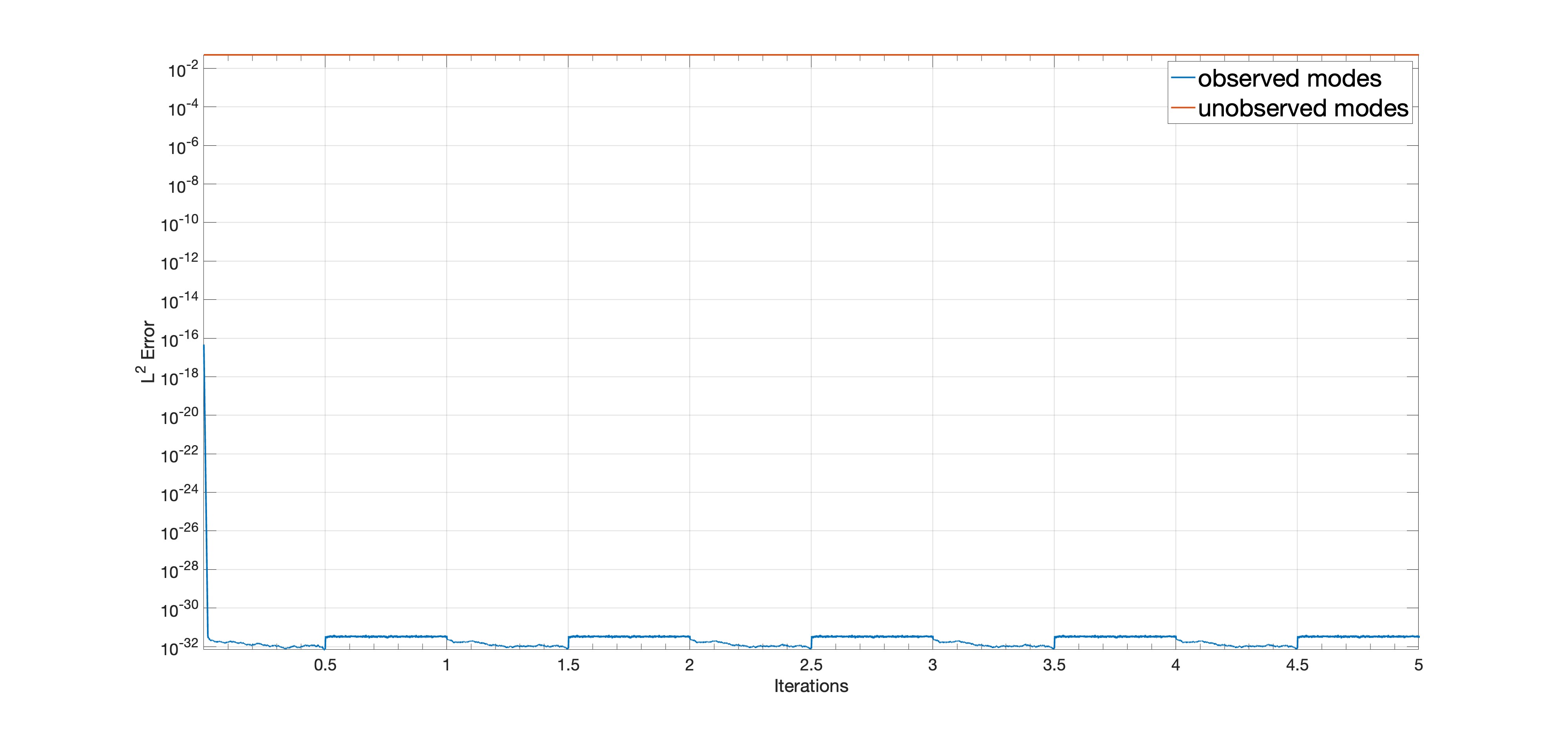}
        \caption{\small Inviscid Linear Transport equation}
        \label{fig:error zero ob:ILT}
    \end{subfigure}
    \hfill
    \begin{subfigure}[b]{.49\textwidth}
        \centering
        \includegraphics[width=\textwidth]{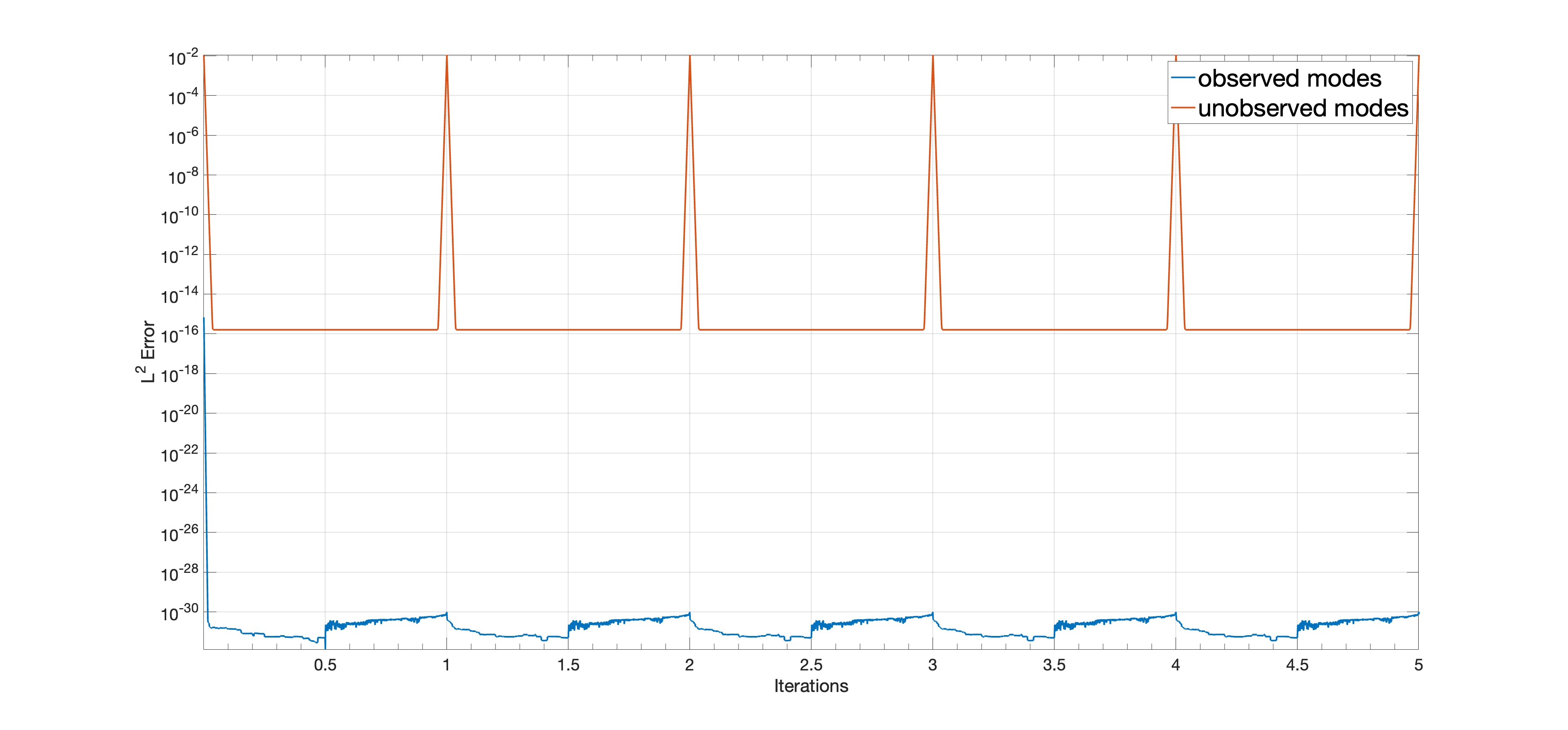}
        \caption{\small Viscous Linear Transport equation}
        \label{fig:error zero ob:VLT}
    \end{subfigure}
    \caption{\small $L^2_{\text{per}}$ error with zero observations.}
    \label{fig:error zero ob_LT}
\end{figure}

\begin{figure}
    \centering
    \begin{subfigure}[b]{.49\textwidth}
        \centering
        \includegraphics[width=\textwidth]{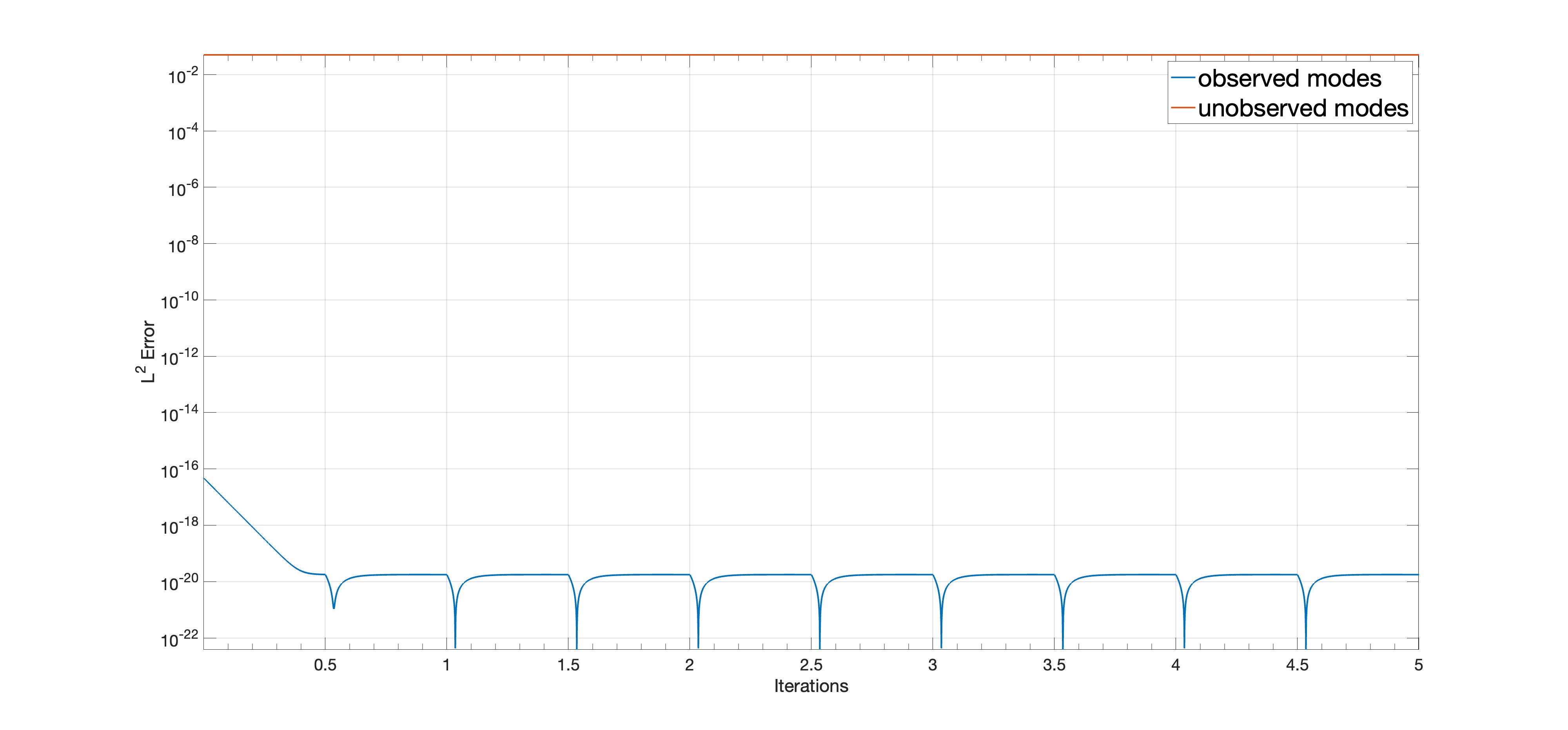}
        \caption{\small Inviscid Burgers equation}
        \label{fig:error zero ob:IB}
    \end{subfigure}
    \hfill
    \begin{subfigure}[b]{.49\textwidth}
        \centering
        \includegraphics[width=\textwidth]{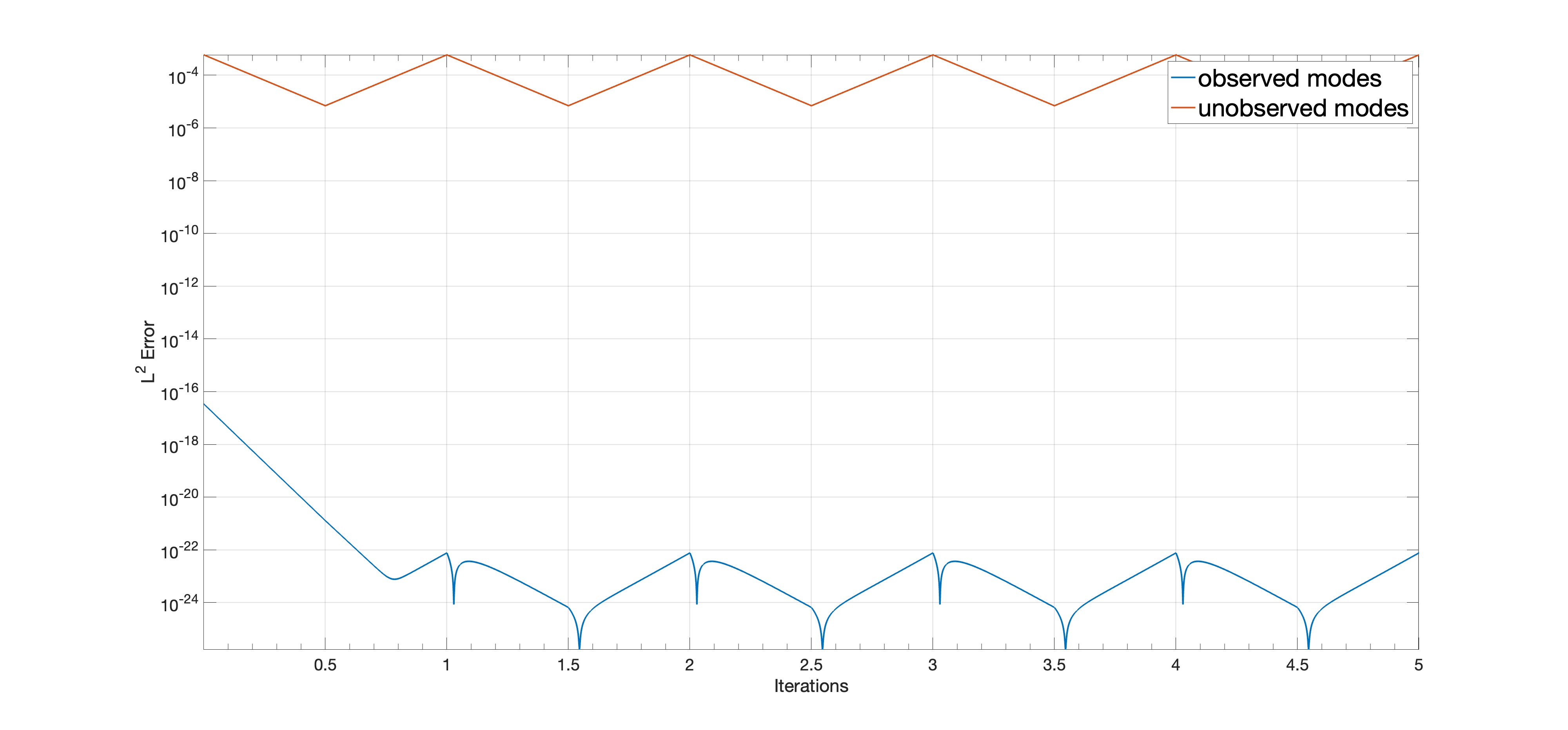}
        \caption{\small Viscous Burgers equation}
        \label{fig:error zero ob:VB}
    \end{subfigure}
    \caption{\small $L^2_{\text{per}}$ error with zero observations.}
    \label{fig:error zero ob_B}
\end{figure}

\begin{figure}
    \centering
    \begin{subfigure}[b]{.49\textwidth}
        \centering
        \includegraphics[width=\textwidth]{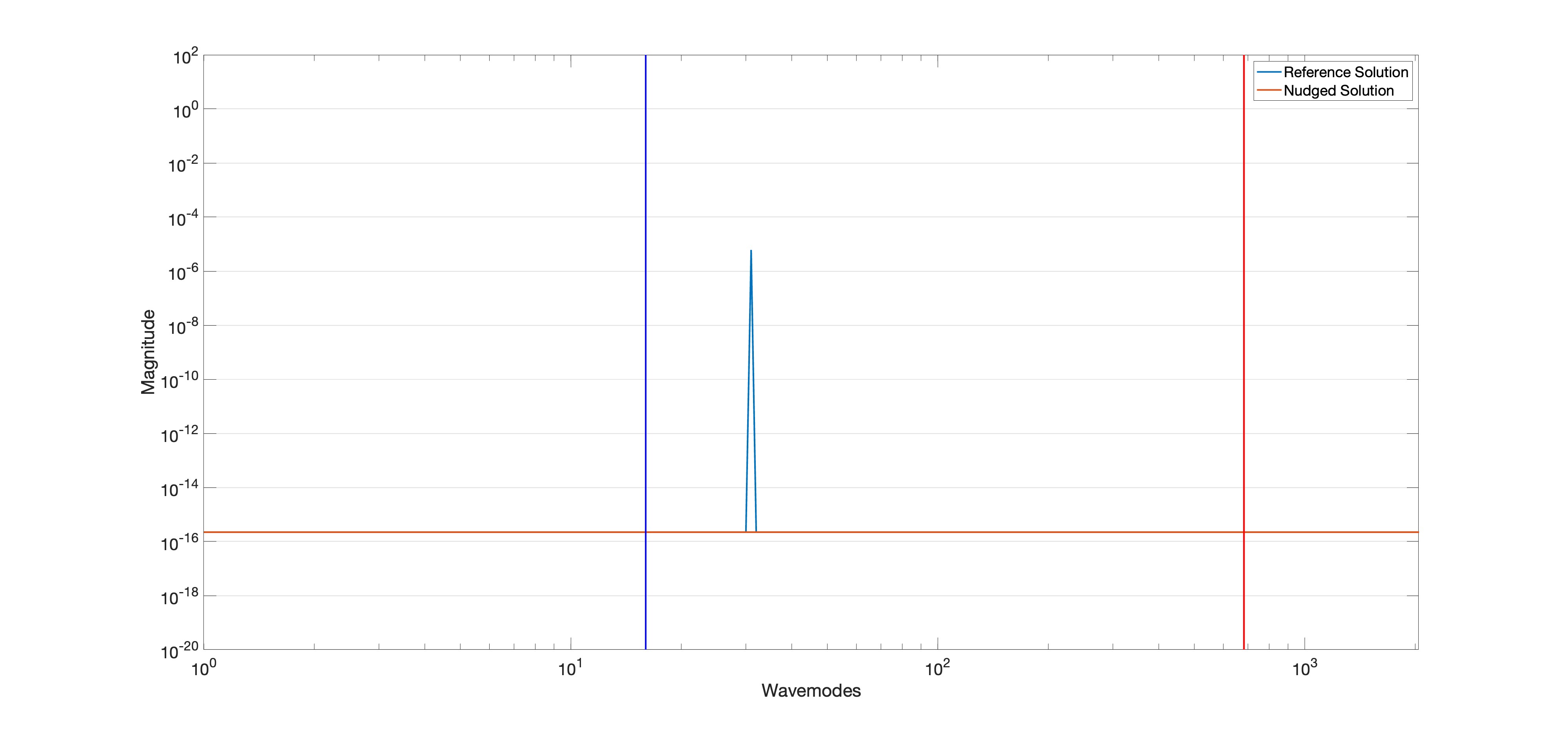}
        \caption{\small Inviscid Linear Transport equation}
        \label{fig:spec zero ob:ILT}
    \end{subfigure}
    \hfill
    \begin{subfigure}[b]{.49\textwidth}
        \centering
        \includegraphics[width=\textwidth]{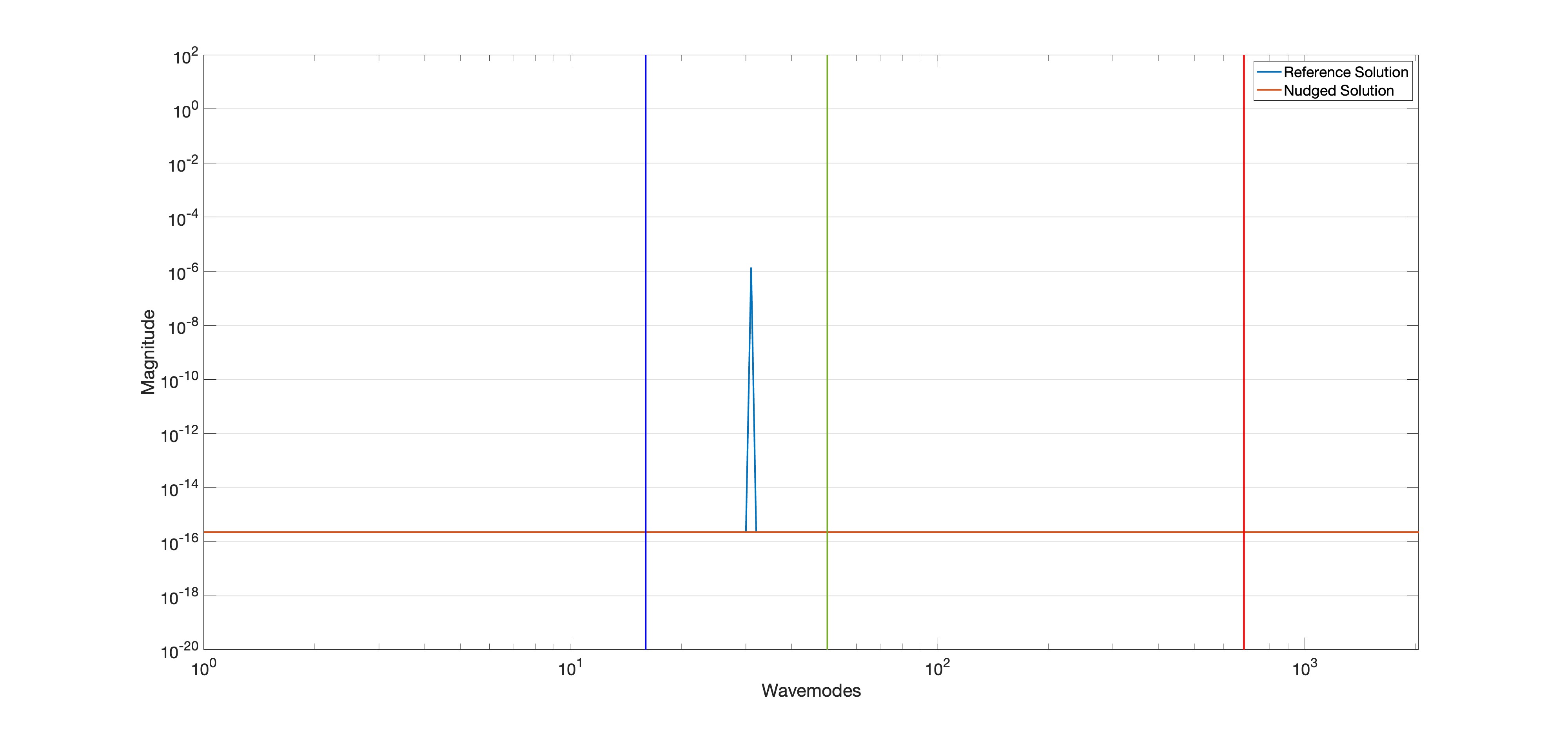}
        \caption{\small Viscous Linear Transport equation}
        \label{fig:spec zero ob:VLT}
    \end{subfigure}
    \caption{\small $L^2_{\text{per}}$ energy spectrum of the recovered initial condition using zero observations.}
    \label{fig:spec zero ob_LT}
\end{figure}

\begin{figure}
    \centering
    \begin{subfigure}[b]{.49\textwidth}
        \centering
        \includegraphics[width=\textwidth]{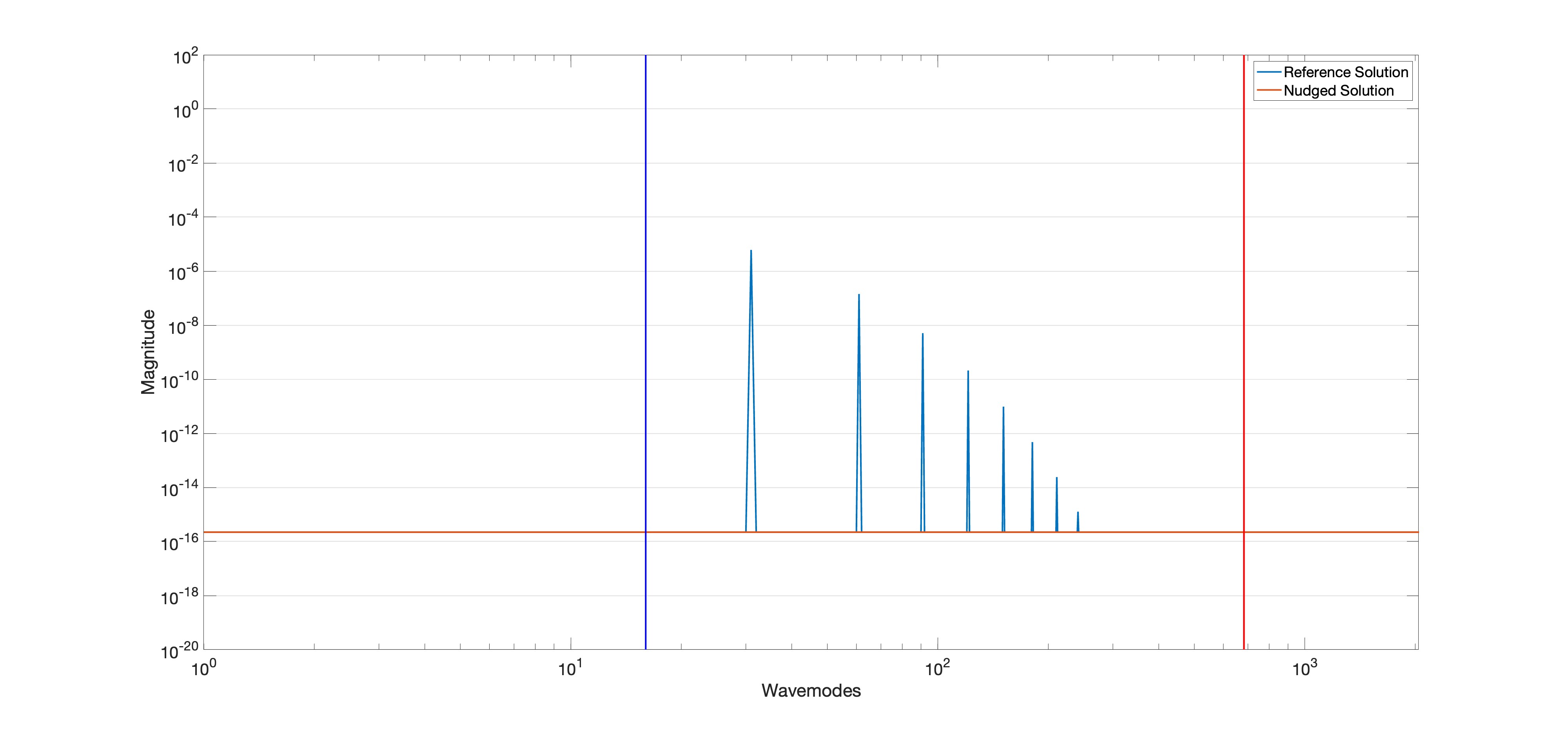}
        \caption{\small Inviscid Burgers equation}
        \label{fig:spec zero ob:IB}
    \end{subfigure}
    \hfill
    \begin{subfigure}[b]{.49\textwidth}
        \centering
        \includegraphics[width=\textwidth]{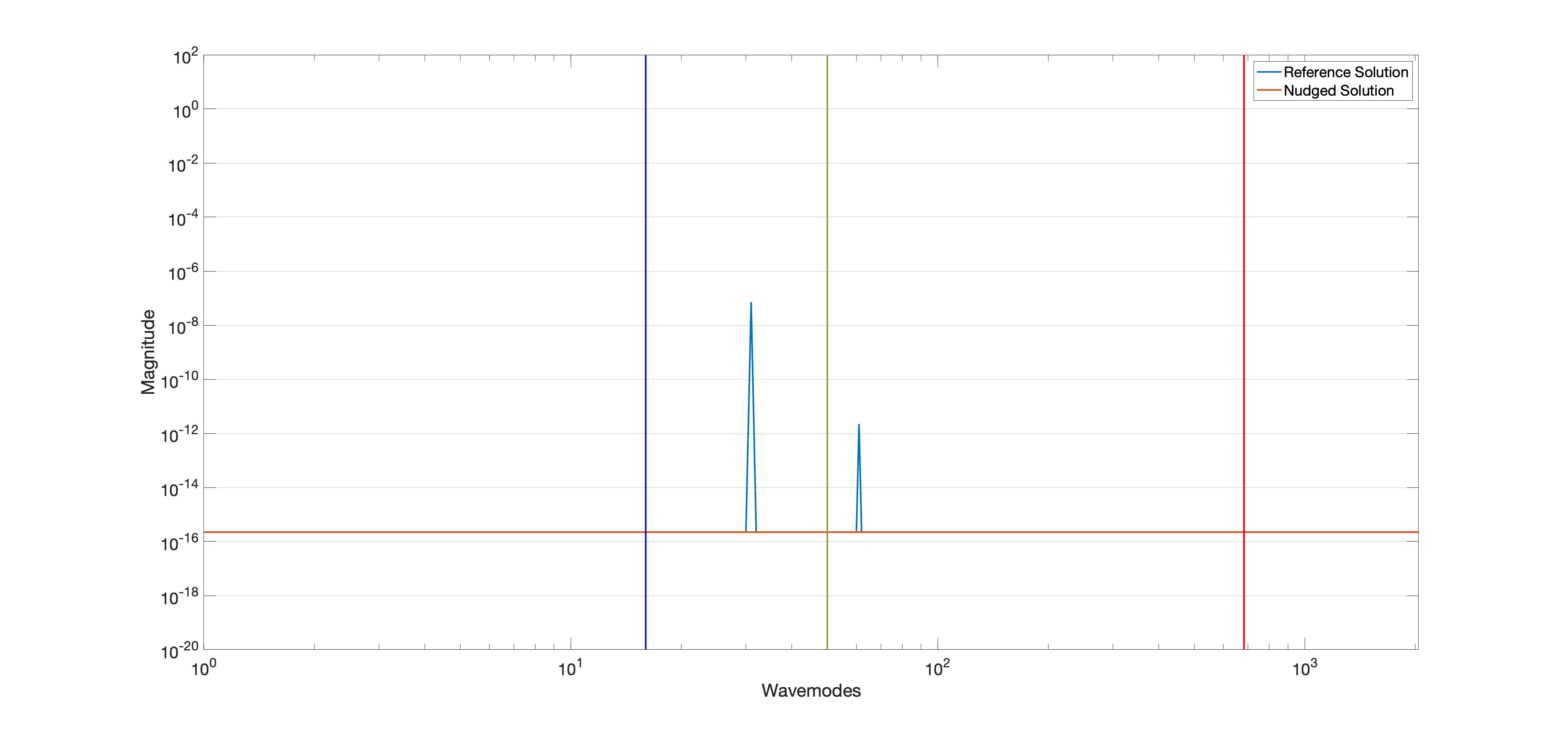}
        \caption{\small Viscous Burgers equation}
        \label{fig:spec zero ob:VB}
    \end{subfigure}
    \caption{\small $L^2_{\text{per}}$ energy spectrum of the recovered initial condition using zero observations.}
    \label{fig:spec zero ob_B}
\end{figure}

\begin{figure}
    \centering
    \begin{subfigure}[b]{.49\textwidth}
        \centering
        \includegraphics[width=\textwidth]{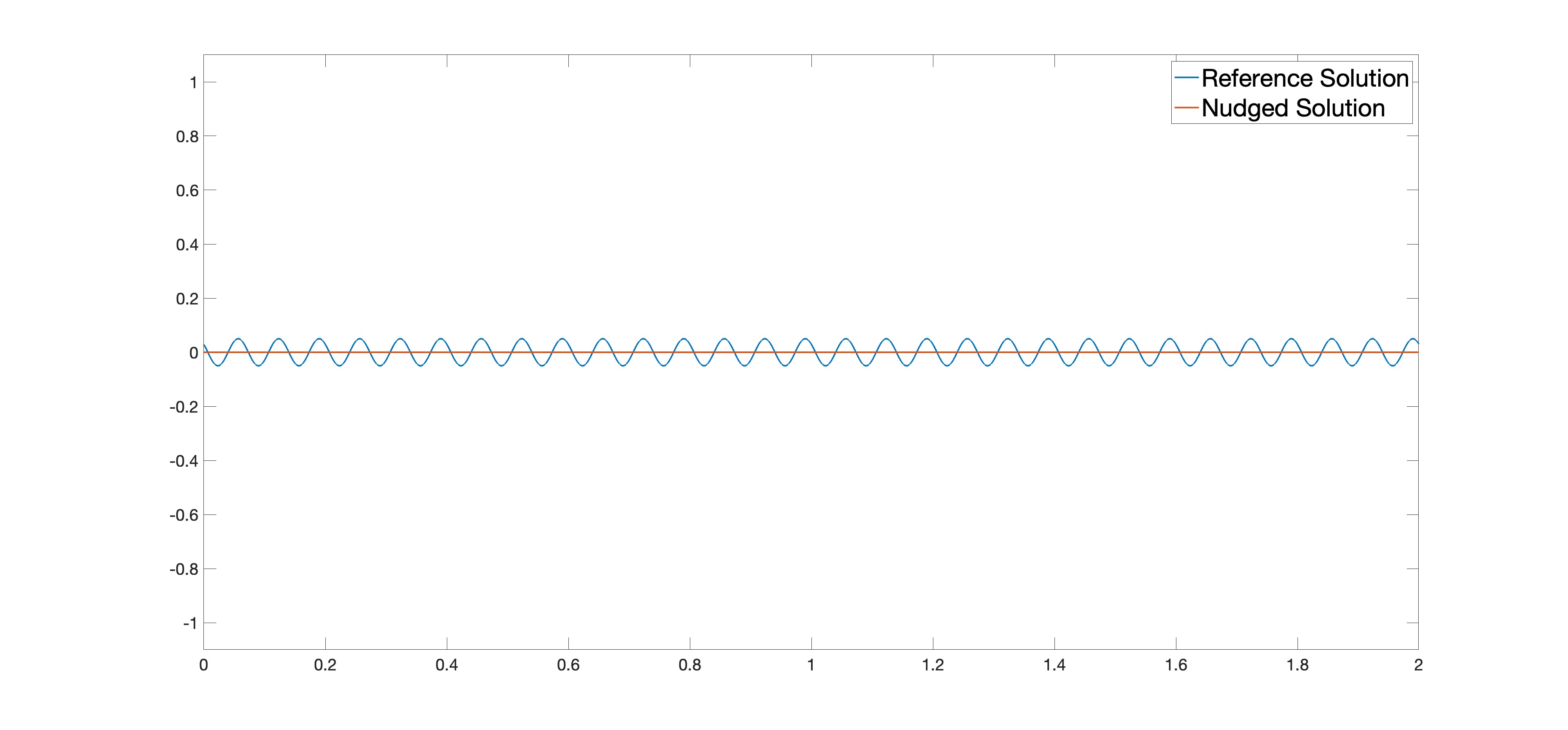}
        \caption{\small Inviscid Linear Transport equation}
        \label{fig:soln zero ob:ILT}
    \end{subfigure}
    \hfill
    \begin{subfigure}[b]{.49\textwidth}
        \centering
        \includegraphics[width=\textwidth]{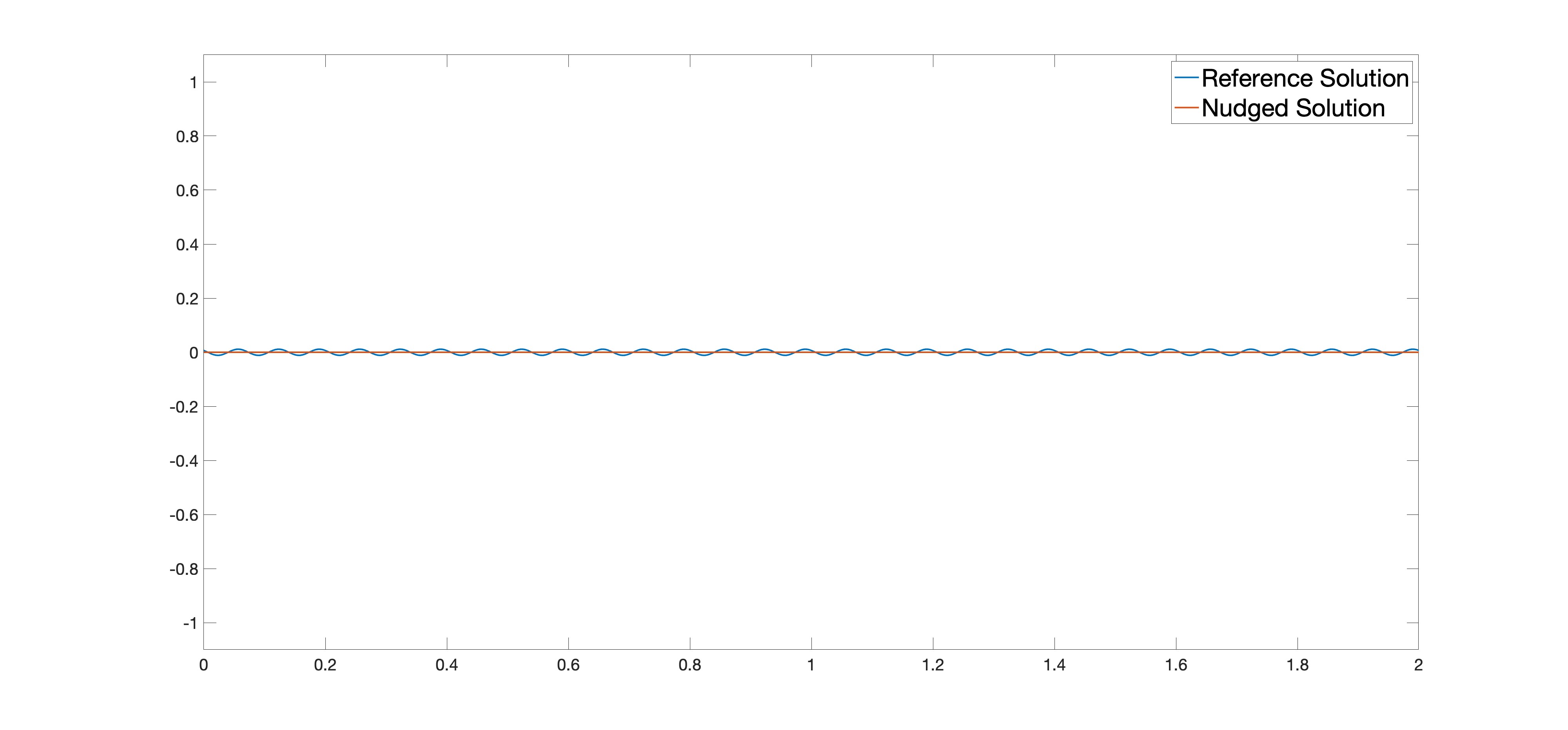}
        \caption{\small Viscous Linear Transport equation}
        \label{fig:soln zero ob:VLT}
    \end{subfigure}
    \caption{\small Recovered initial condition with zero observations.}
    \label{fig:soln zero ob_LT}
\end{figure}

\begin{figure}
    \centering
    \begin{subfigure}[b]{.49\textwidth}
        \centering
        \includegraphics[width=\textwidth]{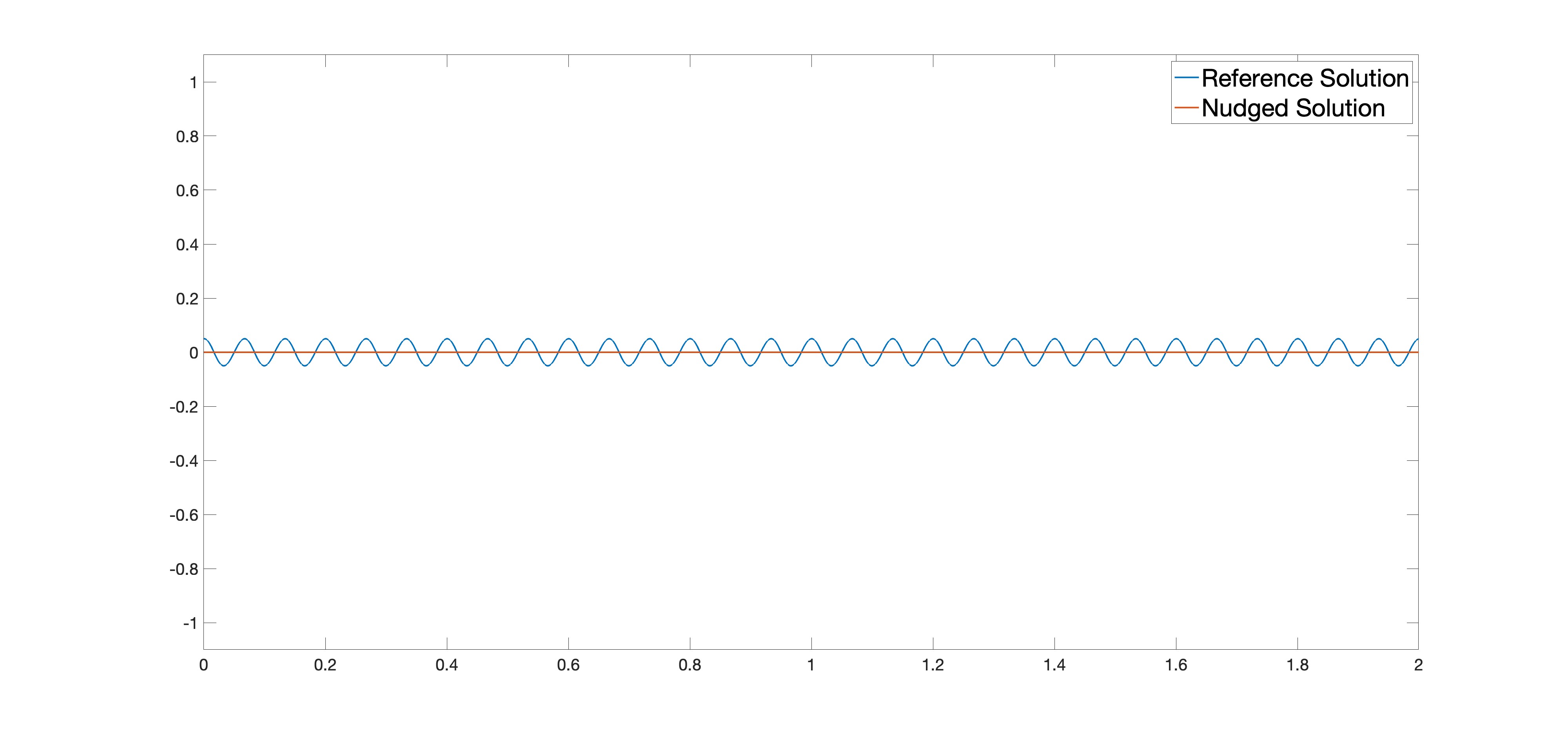}
        \caption{\small Inviscid Burgers equation}
        \label{fig:soln zero ob:IB}
    \end{subfigure}
    \hfill
    \begin{subfigure}[b]{.49\textwidth}
        \centering
        \includegraphics[width=\textwidth]{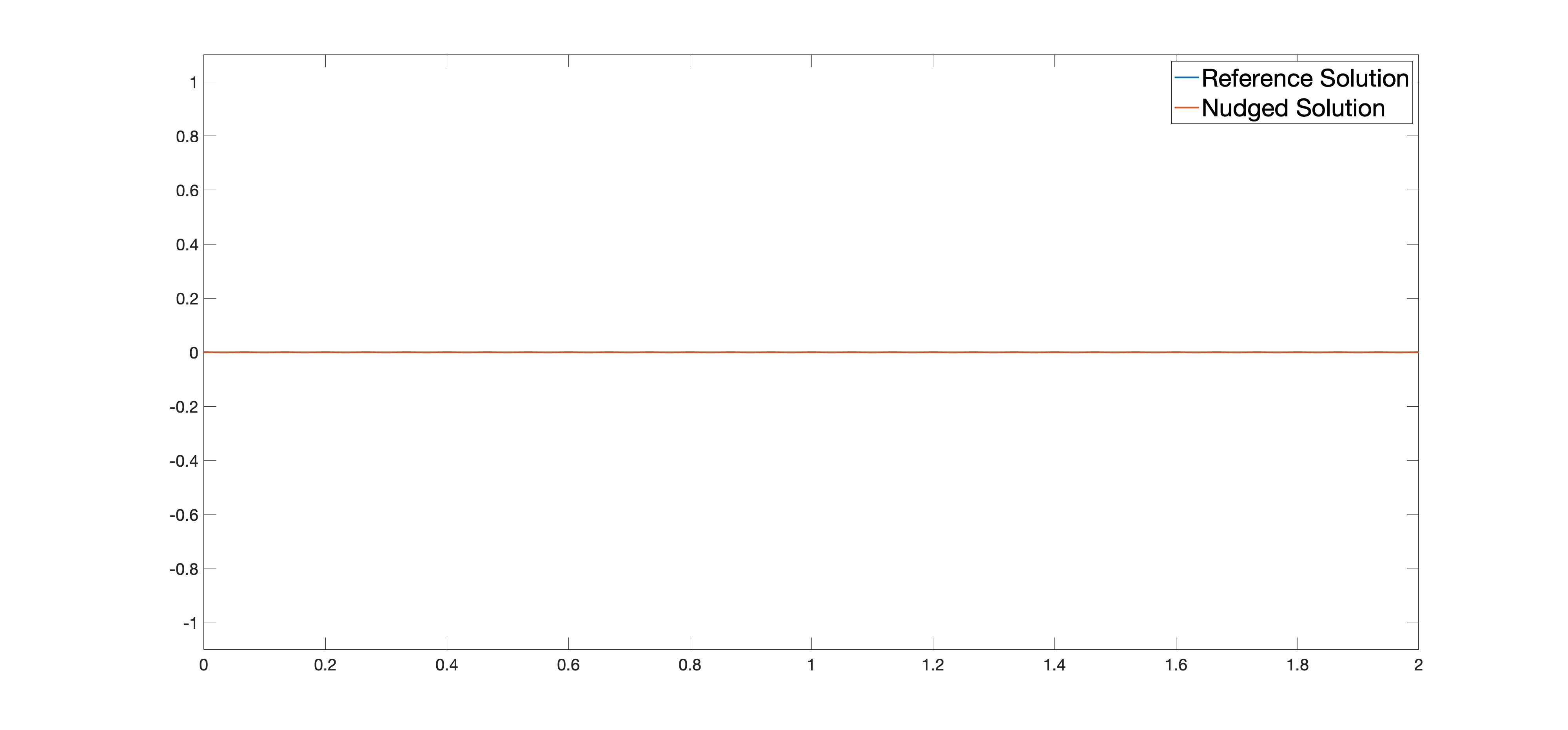}
        \caption{\small Viscous Burgers equation}
        \label{fig:soln zero ob:VB}
    \end{subfigure}
    \caption{\small Recovered initial condition with zero observations.}
    \label{fig:soln zero ob_B}
\end{figure}

\begin{figure}
    \centering
    \begin{subfigure}[b]{.49\textwidth}
        \centering
        \includegraphics[width=\textwidth]{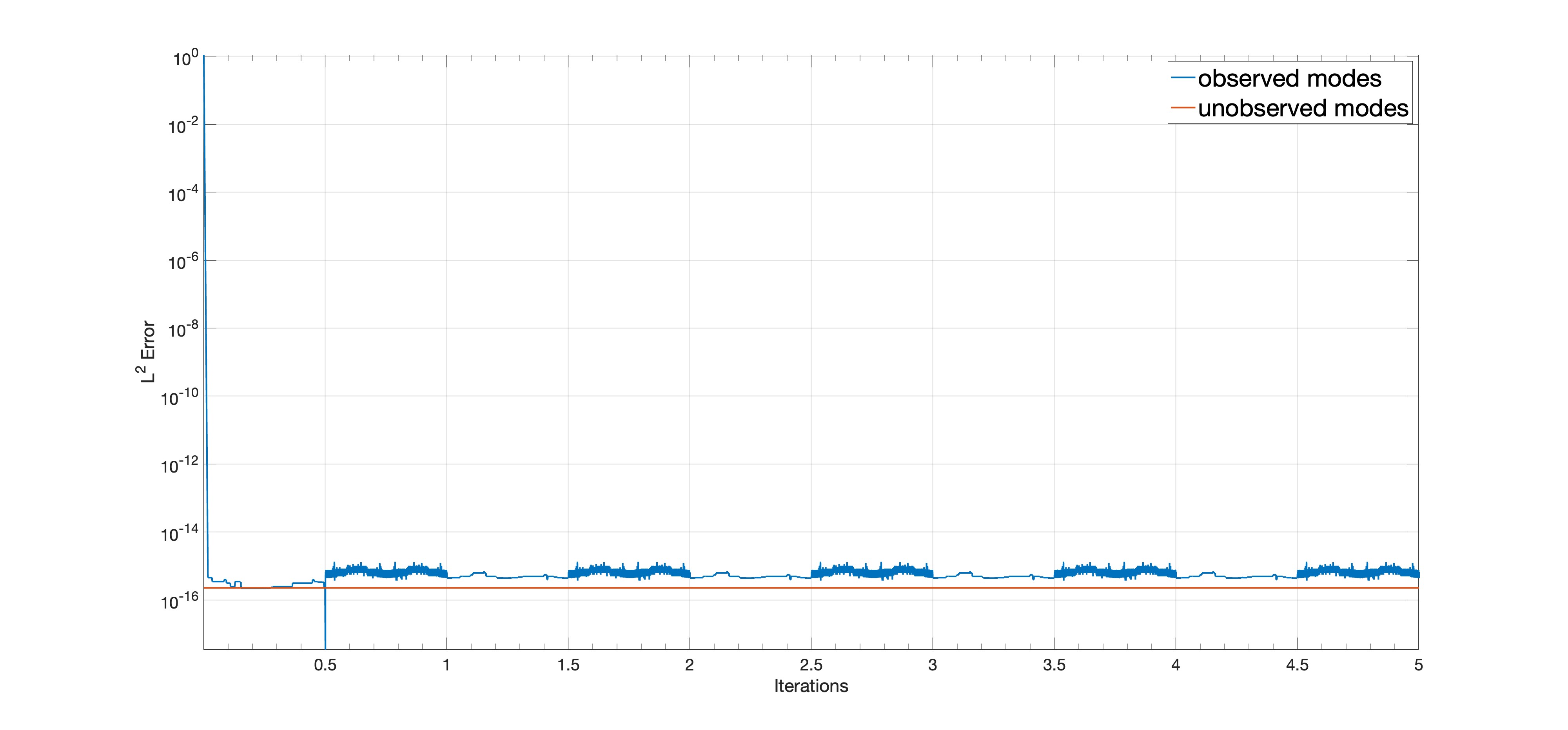}
        \caption{\small Inviscid Linear Transport equation}
        \label{fig:error full ob:ILT}
    \end{subfigure}
    \hfill
    \begin{subfigure}[b]{.49\textwidth}
        \centering
        \includegraphics[width=\textwidth]{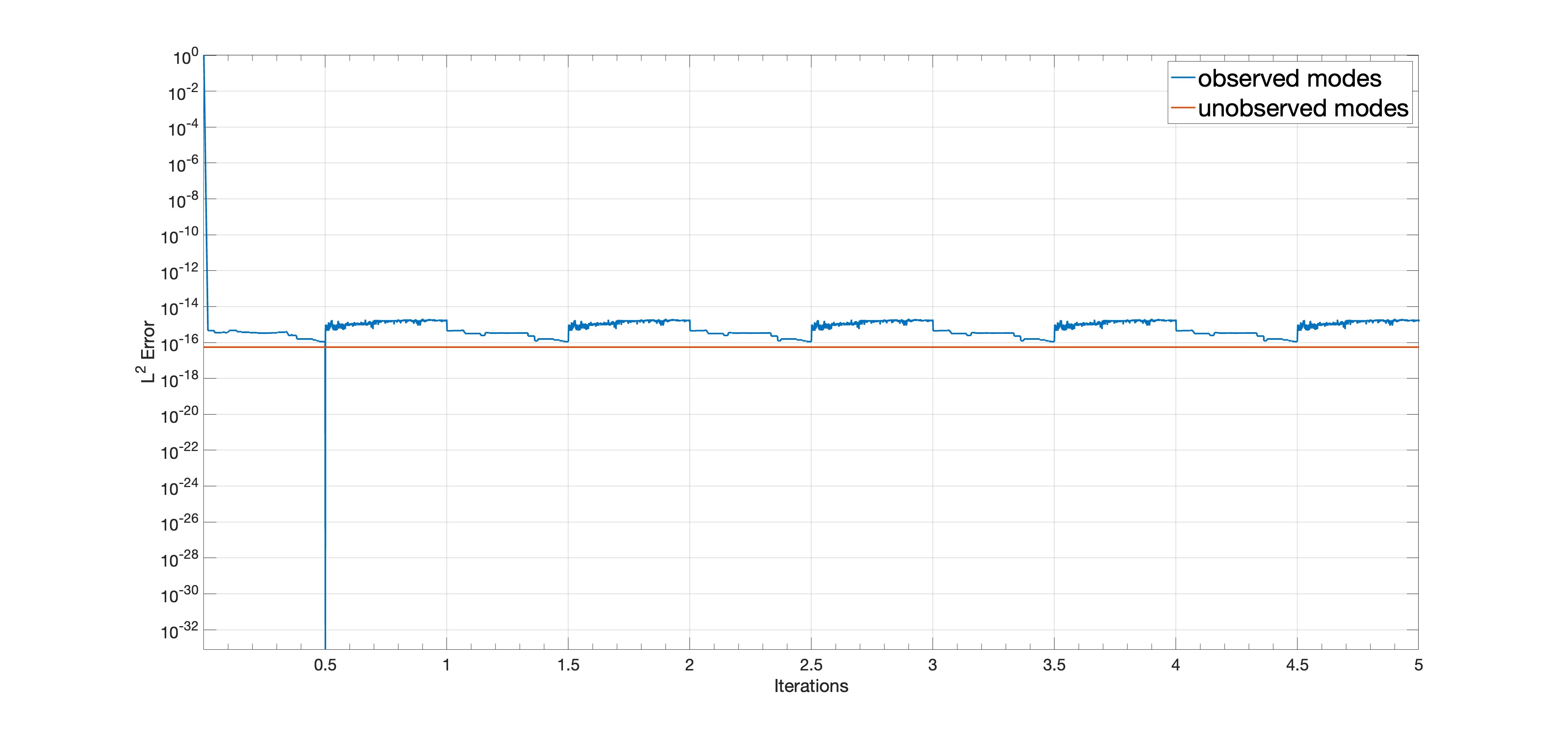}
        \caption{\small Viscous Linear Transport equation}
        \label{fig:error full ob:VLT}
    \end{subfigure}
    \caption{\small $L^2_{\text{per}}$ error using full observations.}
    \label{fig:error full ob_LT}
\end{figure}

\begin{figure}
    \centering
    \begin{subfigure}[b]{.49\textwidth}
        \centering
        \includegraphics[width=\textwidth]{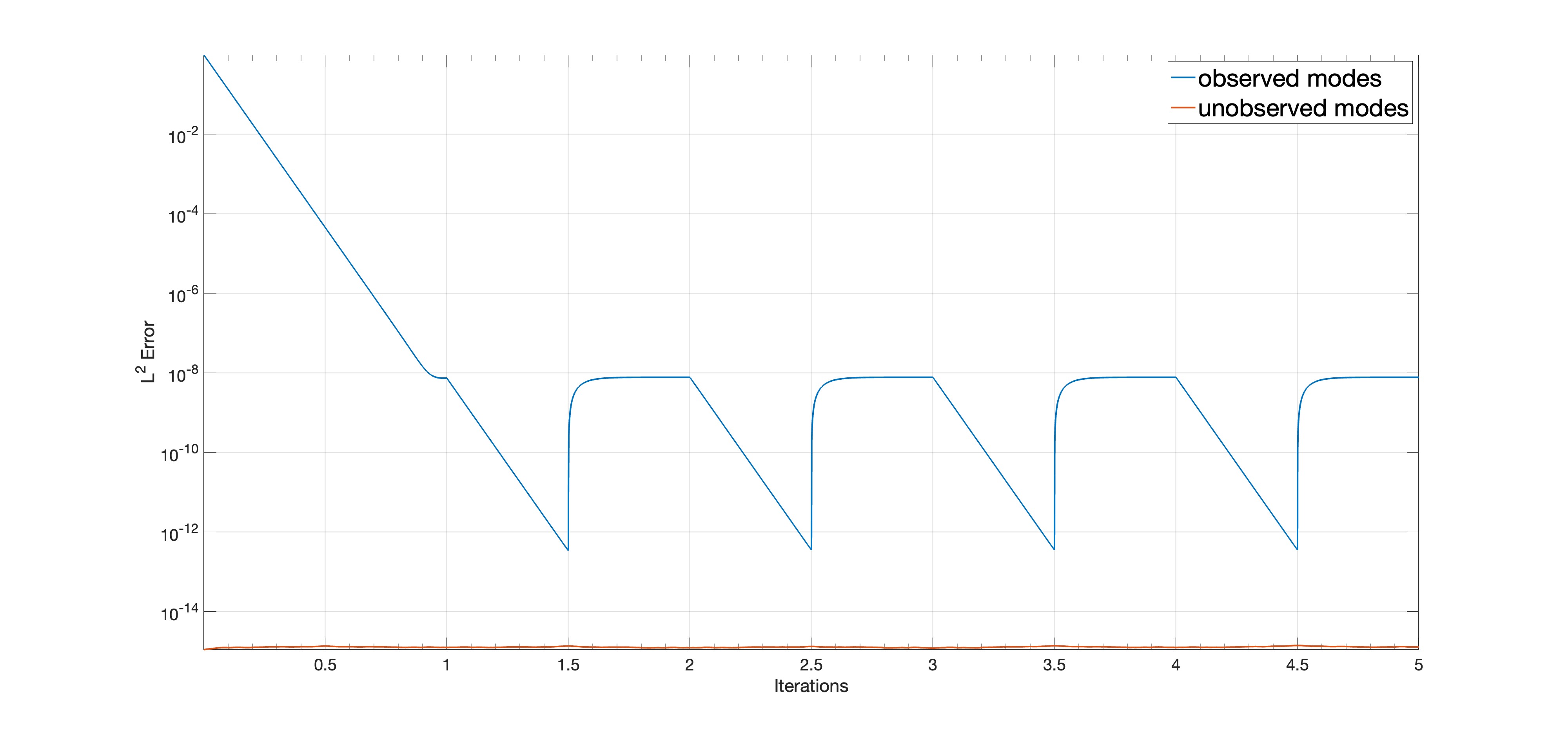}
        \caption{\small Inviscid Burgers equation}
        \label{fig:error full ob:IB}
    \end{subfigure}
    \hfill
    \begin{subfigure}[b]{.49\textwidth}
        \centering
        \includegraphics[width=\textwidth]{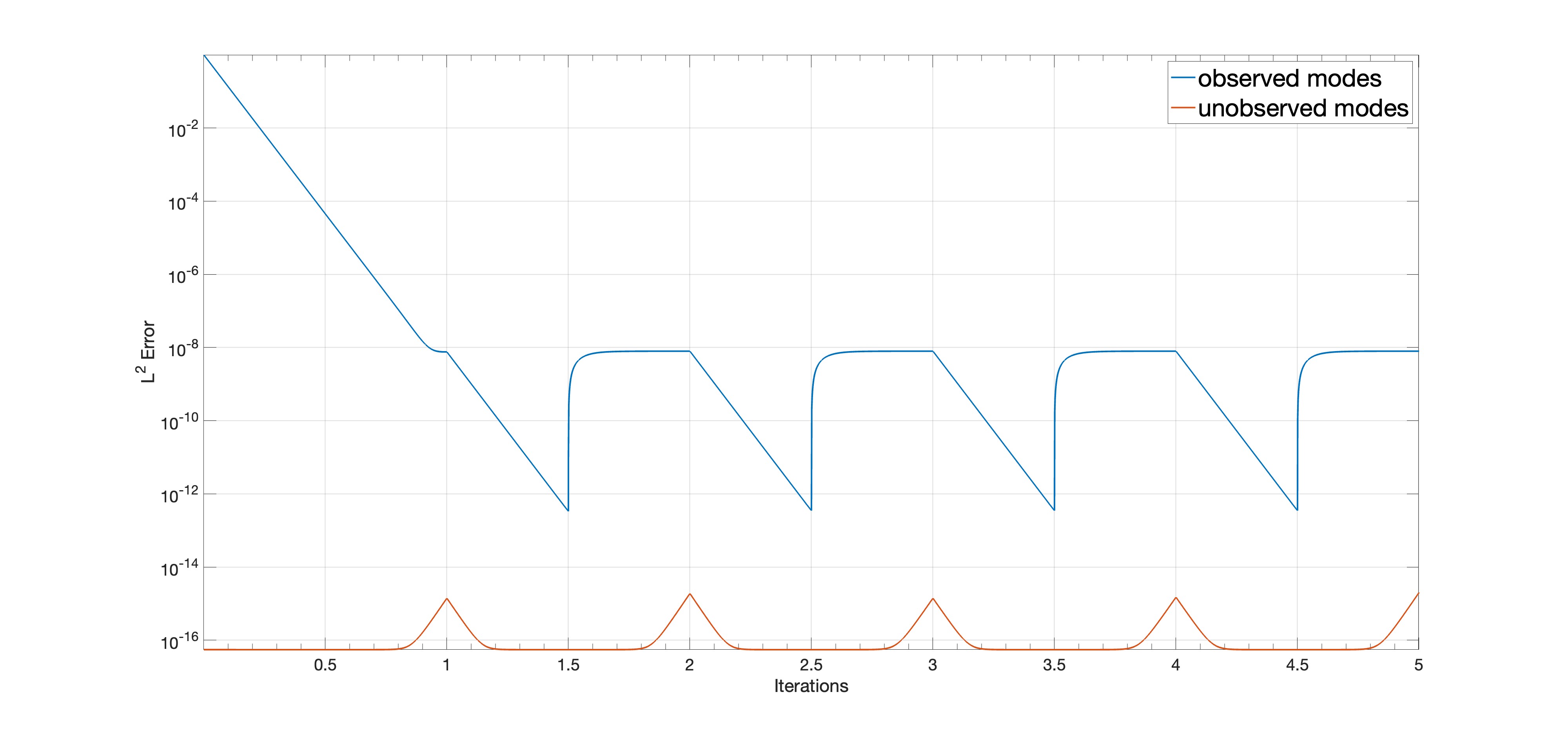}
        \caption{\small Viscous Burgers equation}
        \label{fig:error full ob:VB}
    \end{subfigure}
    \caption{\small $L^2_{\text{per}}$ error using full observations.}
    \label{fig:error full ob_B}
\end{figure}

\begin{figure}
    \centering
    \begin{subfigure}[b]{.49\textwidth}
        \centering
        \includegraphics[width=\textwidth]{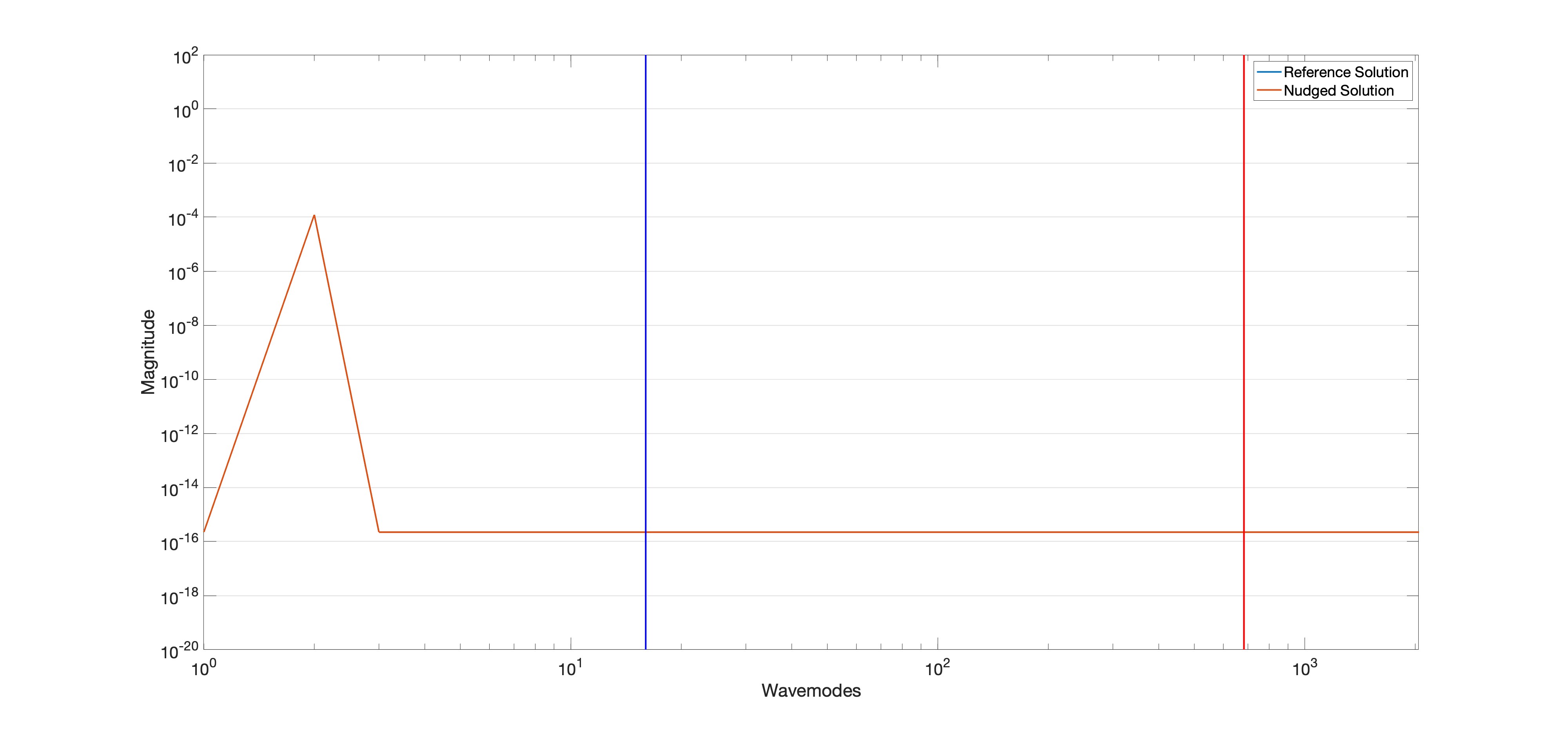}
        \caption{\small Inviscid Linear Transport equation}
        \label{fig:spec full ob:ILT}
    \end{subfigure}
    \hfill
    \begin{subfigure}[b]{.49\textwidth}
        \centering
        \includegraphics[width=\textwidth]{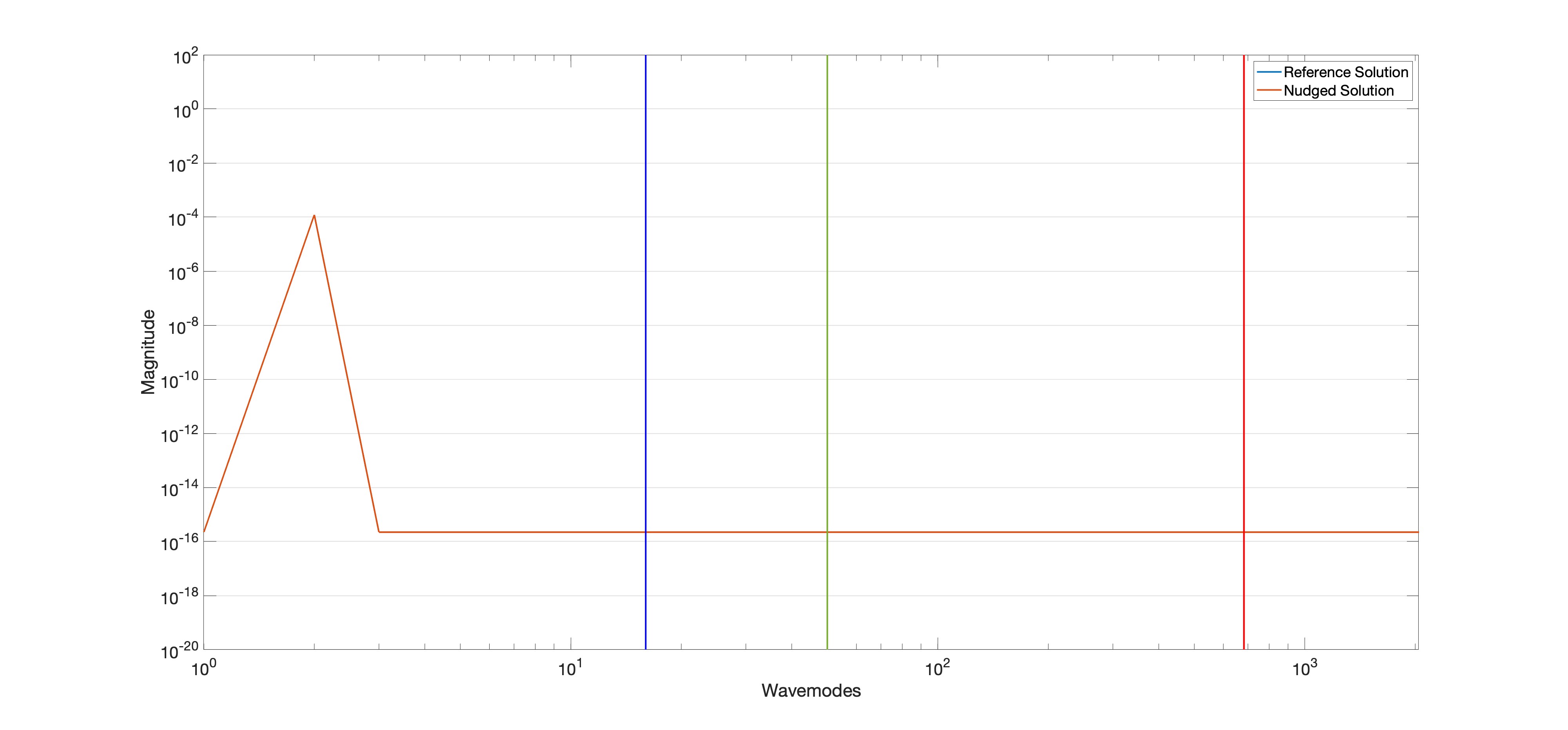}
        \caption{\small Viscous Linear Transport equation}
        \label{fig:spec full ob:VLT}
    \end{subfigure}
    \caption{\small $L^2_{\text{per}}$ energy spectrum of the recovered initial condition with full observations.}
    \label{fig:spec full ob_LT}
\end{figure}

\begin{figure}
    \centering
    \begin{subfigure}[b]{.49\textwidth}
        \centering
        \includegraphics[width=\textwidth]{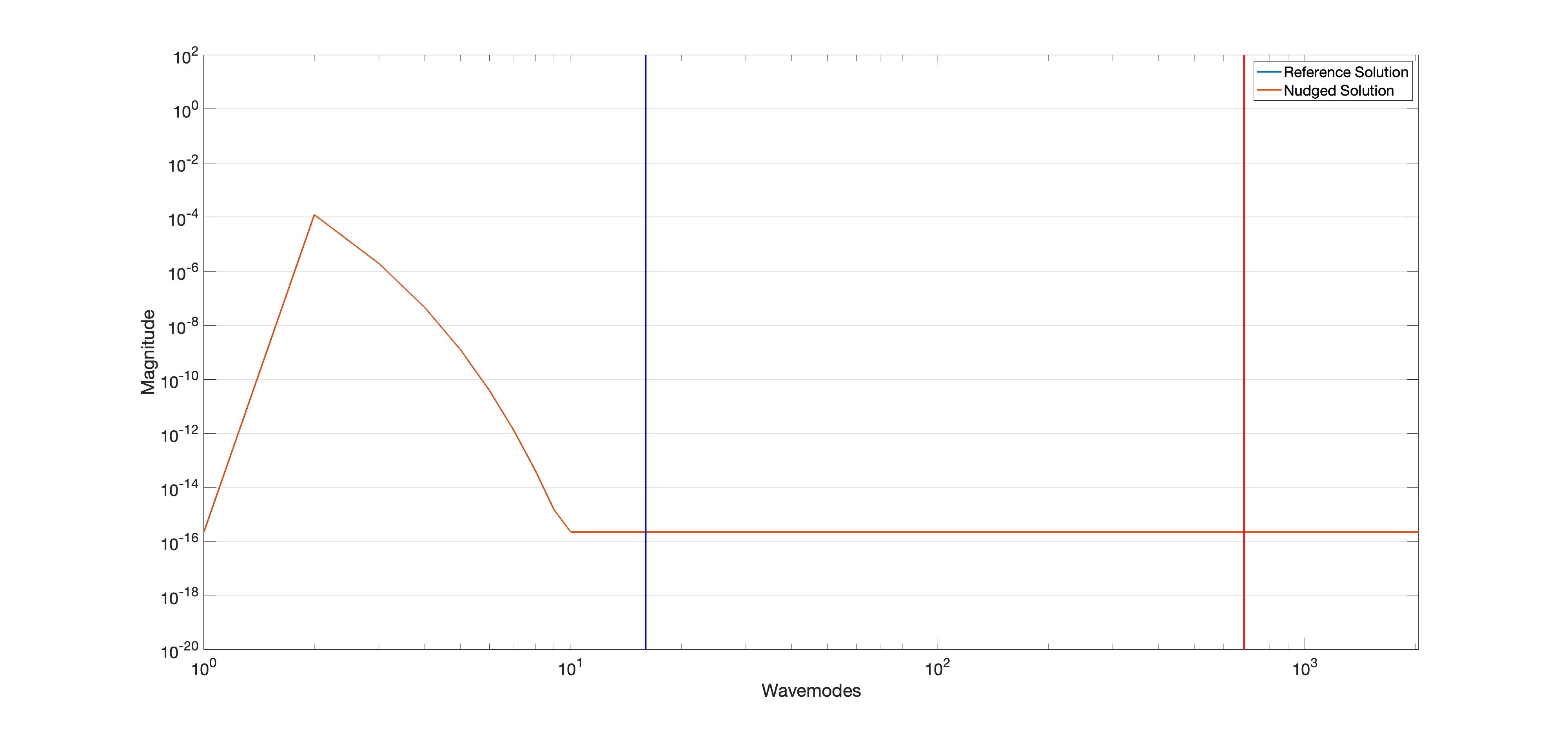}
        \caption{\small Inviscid Burgers equation}
        \label{fig:spec full ob:IB}
    \end{subfigure}
    \hfill
    \begin{subfigure}[b]{.49\textwidth}
        \centering
        \includegraphics[width=\textwidth]{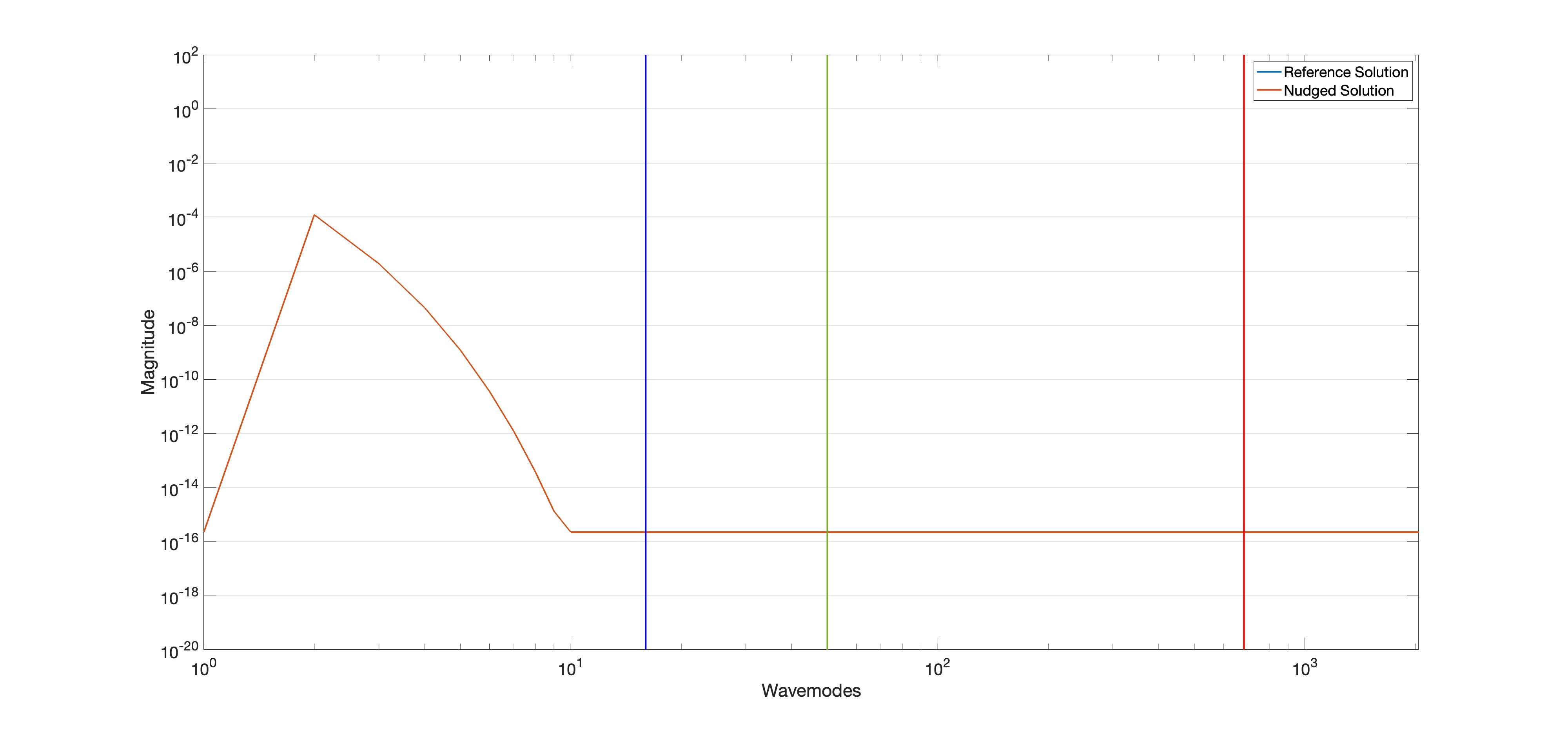}
        \caption{\small Viscous Burgers equation}
        \label{fig:spec full ob:VB}
    \end{subfigure}
    \caption{\small $L^2_{\text{per}}$ energy spectrum of the recovered initial condition using full observations.}
    \label{fig:spec full ob_B}
\end{figure}

\begin{figure}
    \centering
    \begin{subfigure}[b]{.49\textwidth}
        \centering
        \includegraphics[width=\textwidth]{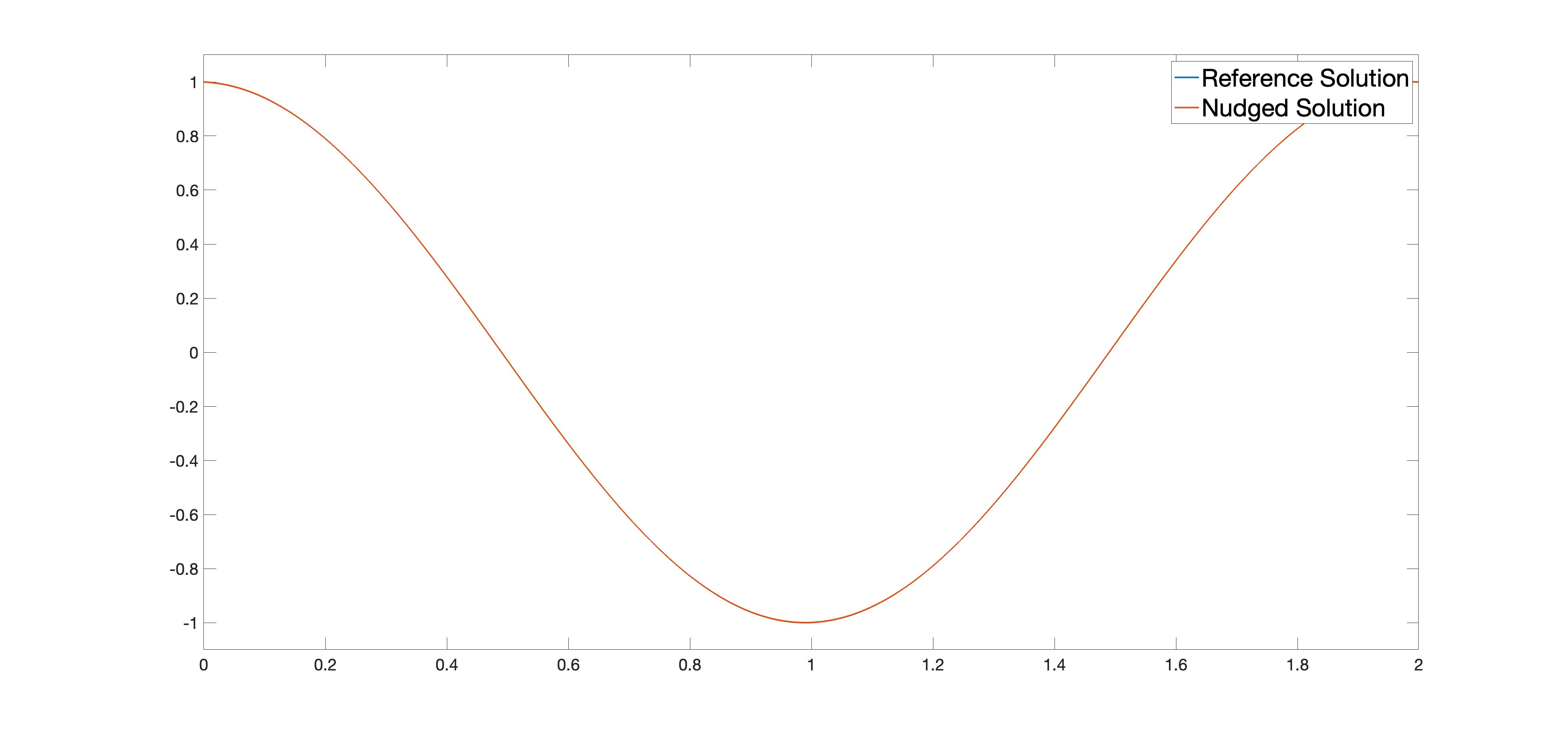}
        \caption{\small Inviscid Linear Transport equation}
        \label{fig:soln full ob:ILT}
    \end{subfigure}
    \hfill
    \begin{subfigure}[b]{.49\textwidth}
        \centering
        \includegraphics[width=\textwidth]{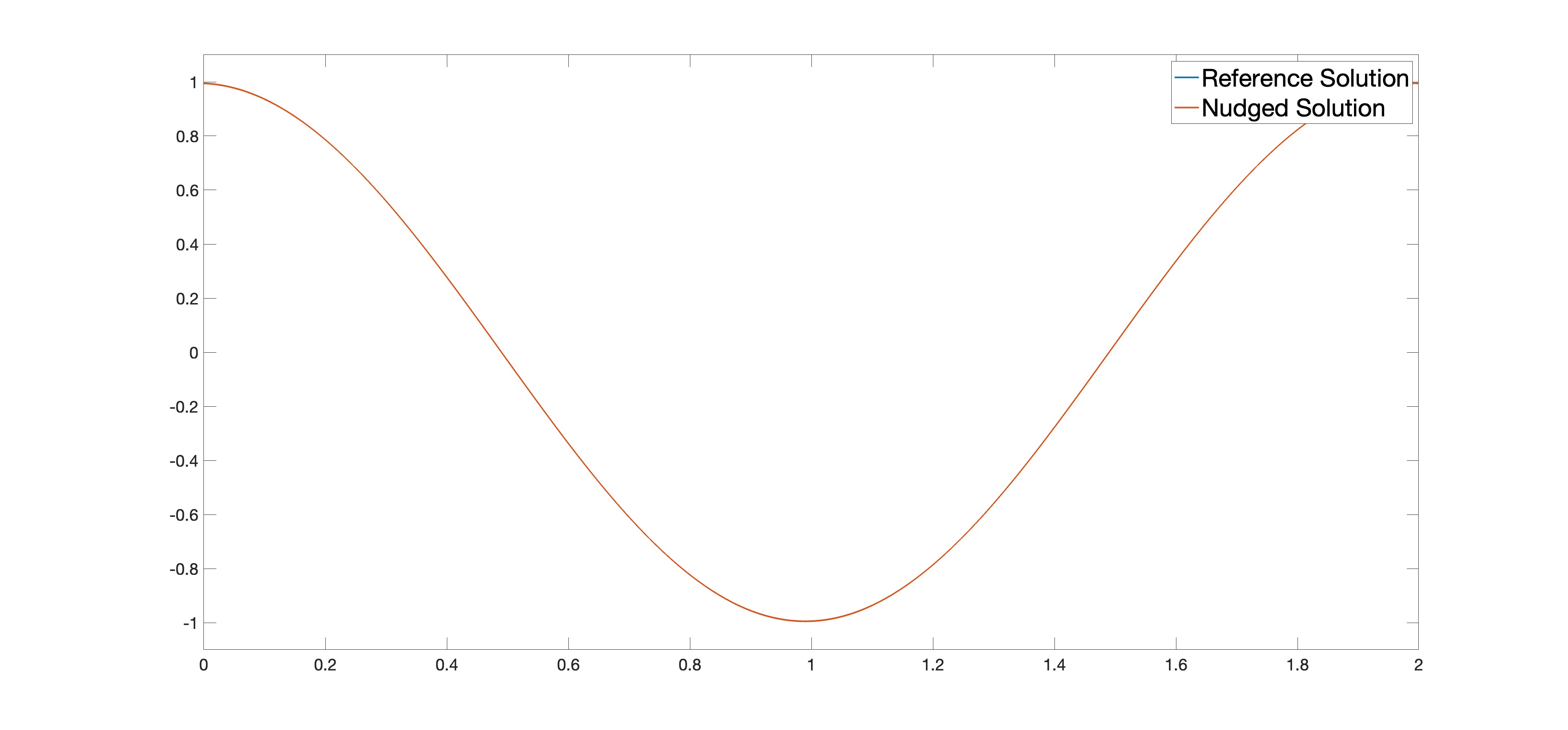}
        \caption{\small Viscous Linear Transport equation}
        \label{fig:soln full ob:VLT}
    \end{subfigure}
    \caption{\small Recovered initial condition using full observations.}
    \label{fig:soln full ob_LT}
\end{figure}

\begin{figure}
    \centering
    \begin{subfigure}[b]{.49\textwidth}
        \centering
        \includegraphics[width=\textwidth]{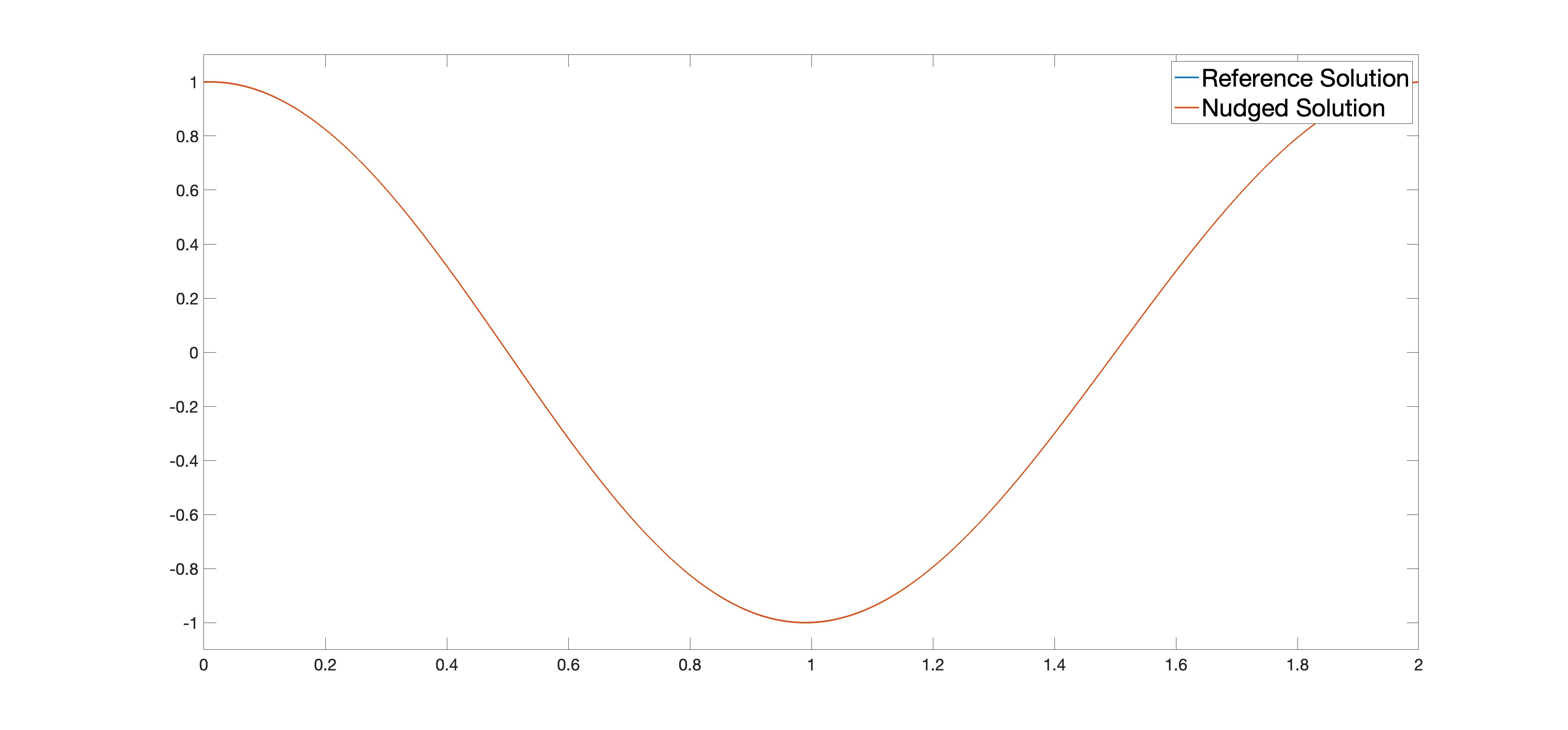}
        \caption{\small Inviscid Burgers equation}
        \label{fig:soln full ob:IB}
    \end{subfigure}
    \hfill
    \begin{subfigure}[b]{.49\textwidth}
        \centering
        \includegraphics[width=\textwidth]{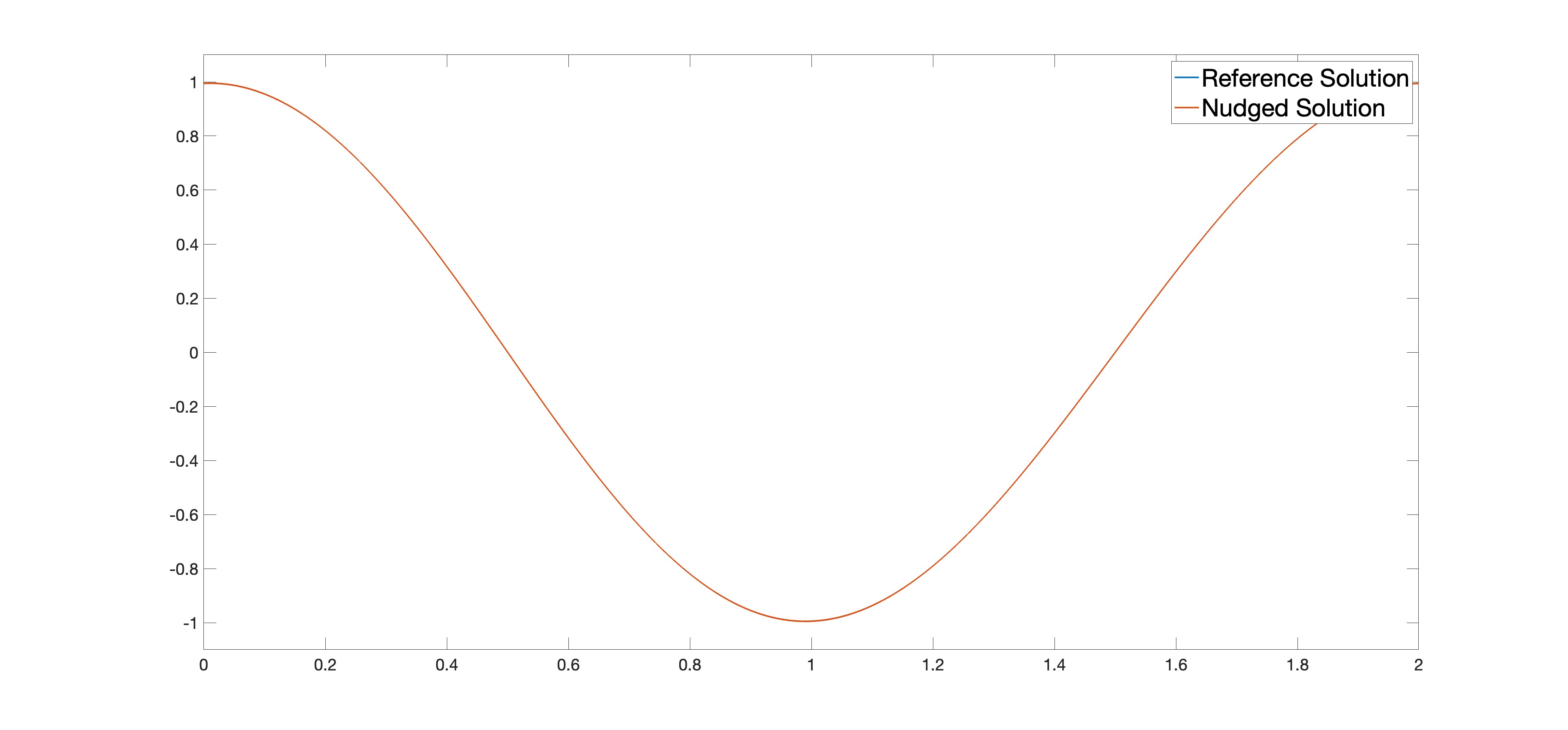}
        \caption{\small Viscous Burgers equation}
        \label{fig:soln full ob:VB}
    \end{subfigure}
    \caption{\small Recovered initial condition using full observations.}
    \label{fig:soln full ob_B}
\end{figure}

\begin{figure}
    \centering
    \begin{subfigure}[b]{.49\textwidth}
        \centering
        \includegraphics[width=\textwidth]{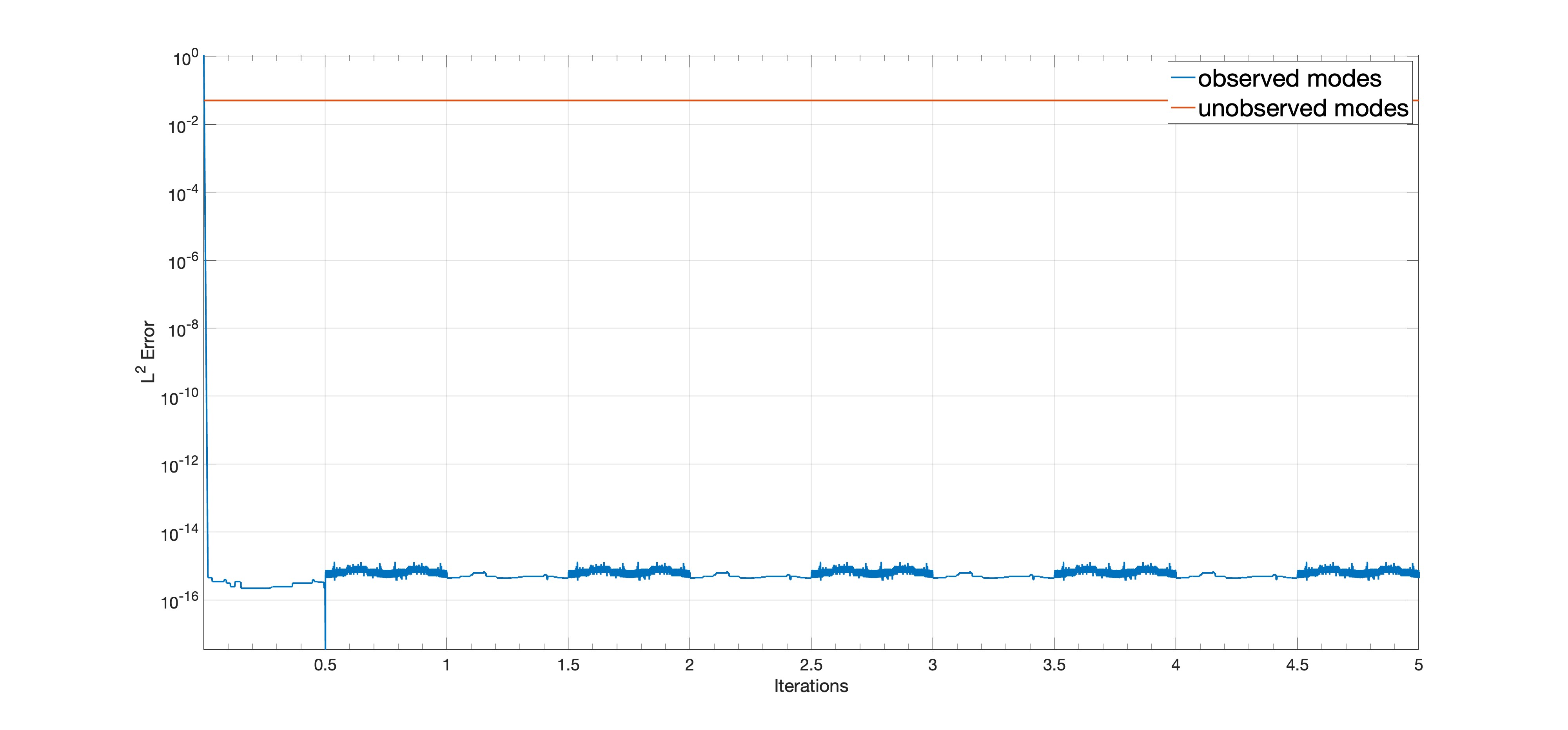}
        \caption{\small Inviscid Linear Transport equation}
        \label{fig:error partial ob:ILT}
    \end{subfigure}
    \hfill
    \begin{subfigure}[b]{.49\textwidth}
        \centering
        \includegraphics[width=\textwidth]{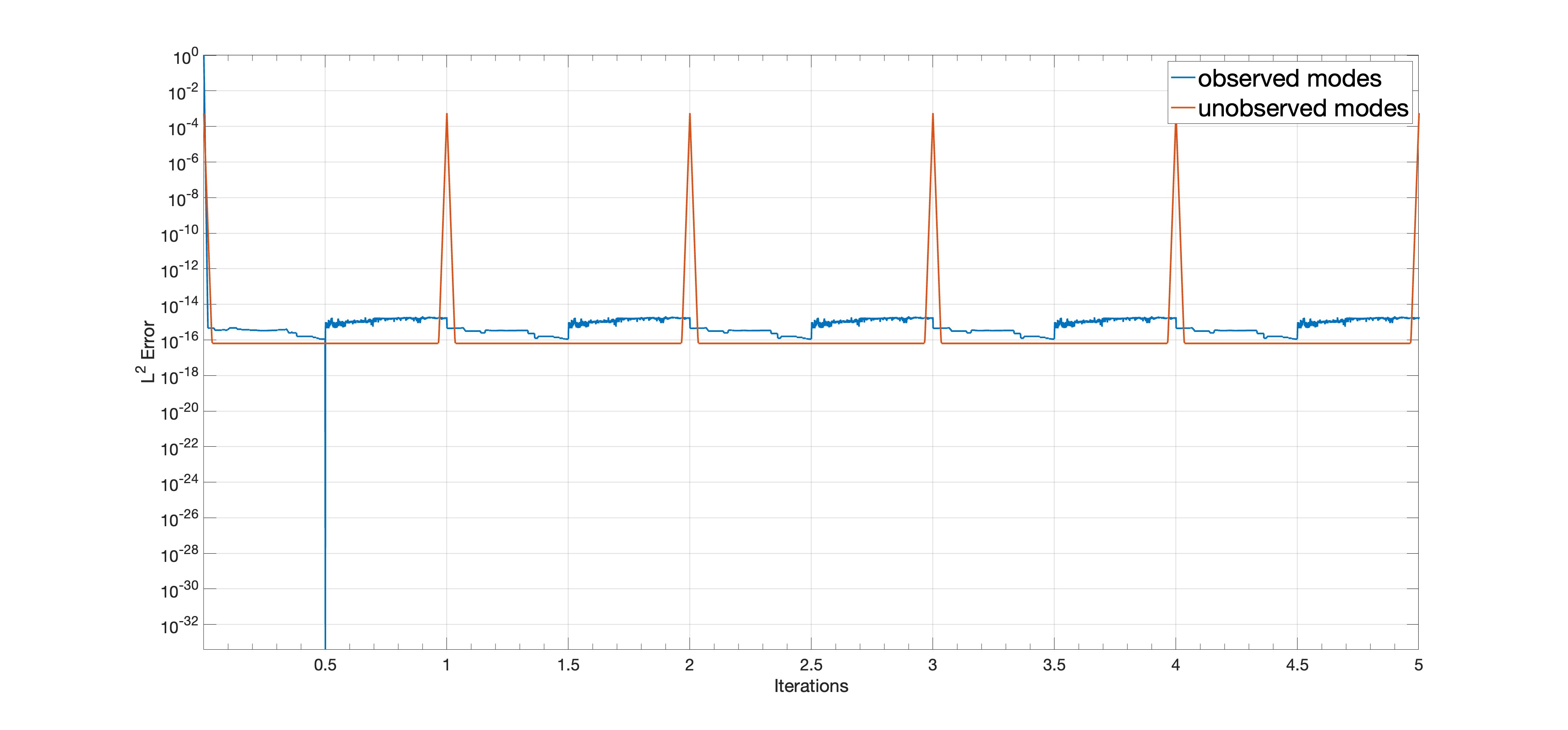}
        \caption{\small Viscous Linear Transport equation}
        \label{fig:error partial ob:VLT}
    \end{subfigure}
    \caption{\small $L^2_{\text{per}}$ error with low-mode observations ($|k| \leq M$).}
    \label{fig:error partial ob_LT}
\end{figure}

\begin{figure}
    \centering
    \begin{subfigure}[b]{.49\textwidth}
        \centering
        \includegraphics[width=\textwidth]{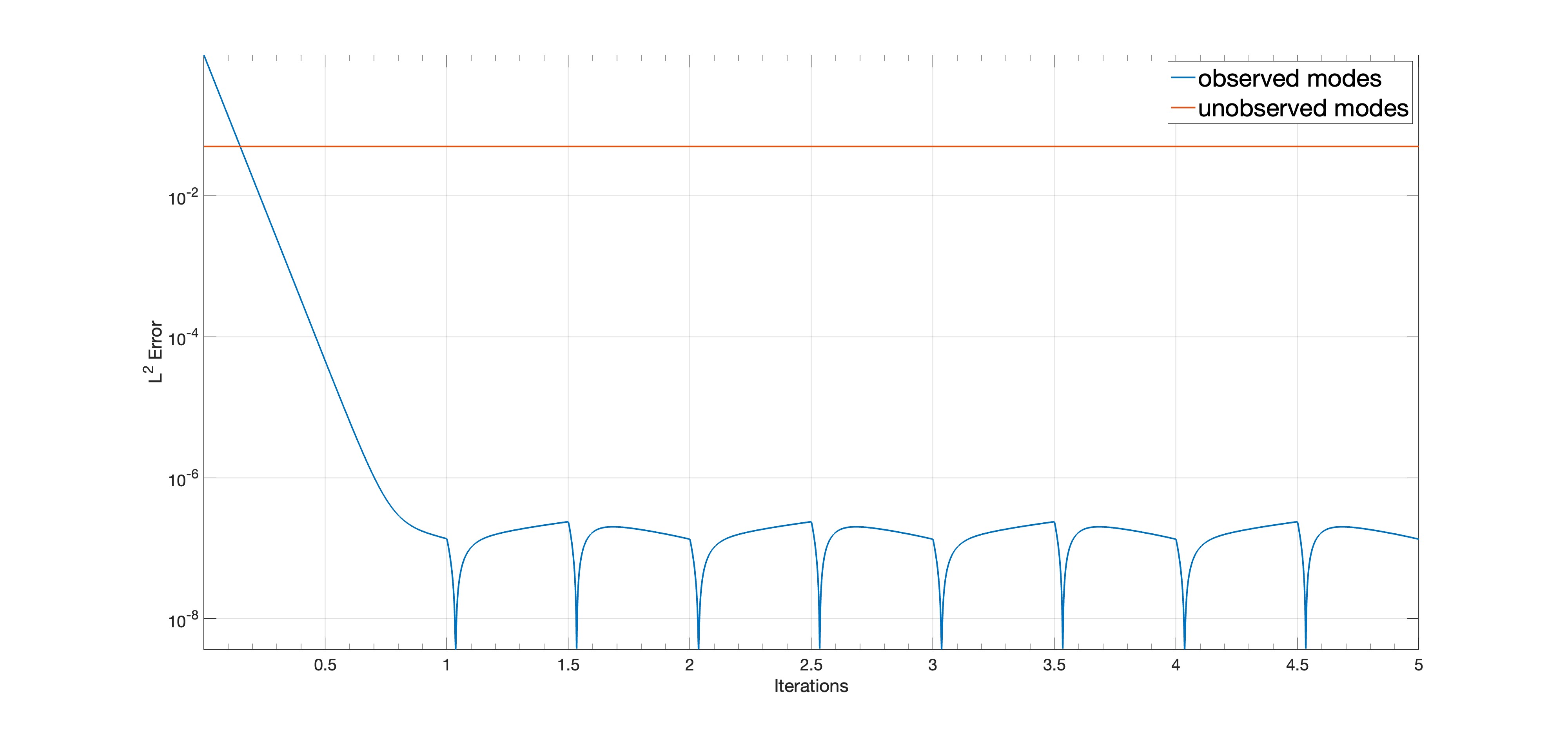}
        \caption{\small Inviscid Burgers equation}
        \label{fig:error partial ob:IB}
    \end{subfigure}
    \hfill
    \begin{subfigure}[b]{.49\textwidth}
        \centering
        \includegraphics[width=\textwidth]{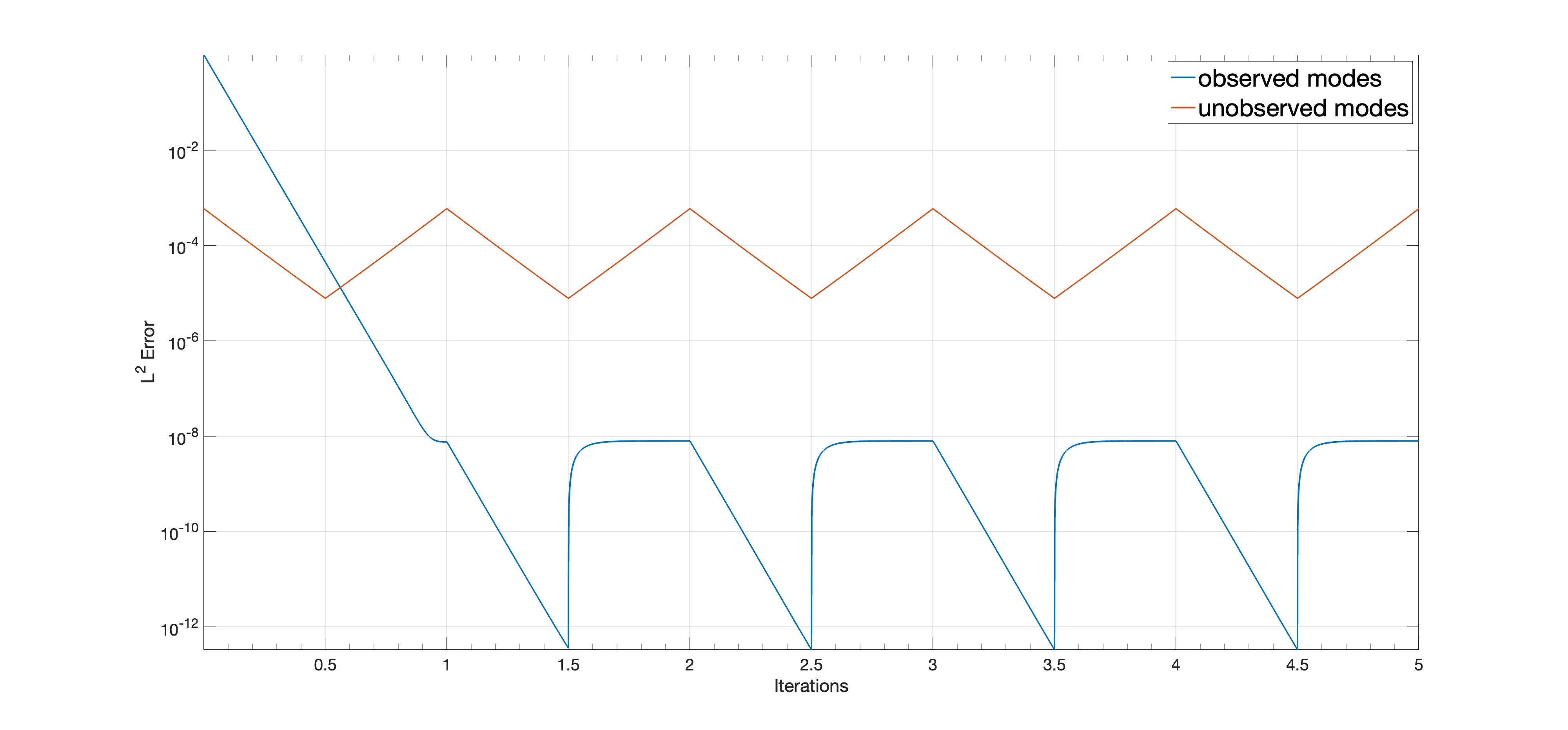}
        \caption{\small Viscous Burgers equation}
        \label{fig:error partial ob:VB}
    \end{subfigure}
    \caption{\small $L^2_{\text{per}}$ error with low-mode observations ($|k| \leq M$).}
    \label{fig:error partial ob_B}
\end{figure}

\begin{figure}
    \centering
    \begin{subfigure}[b]{.49\textwidth}
        \centering
        \includegraphics[width=\textwidth]{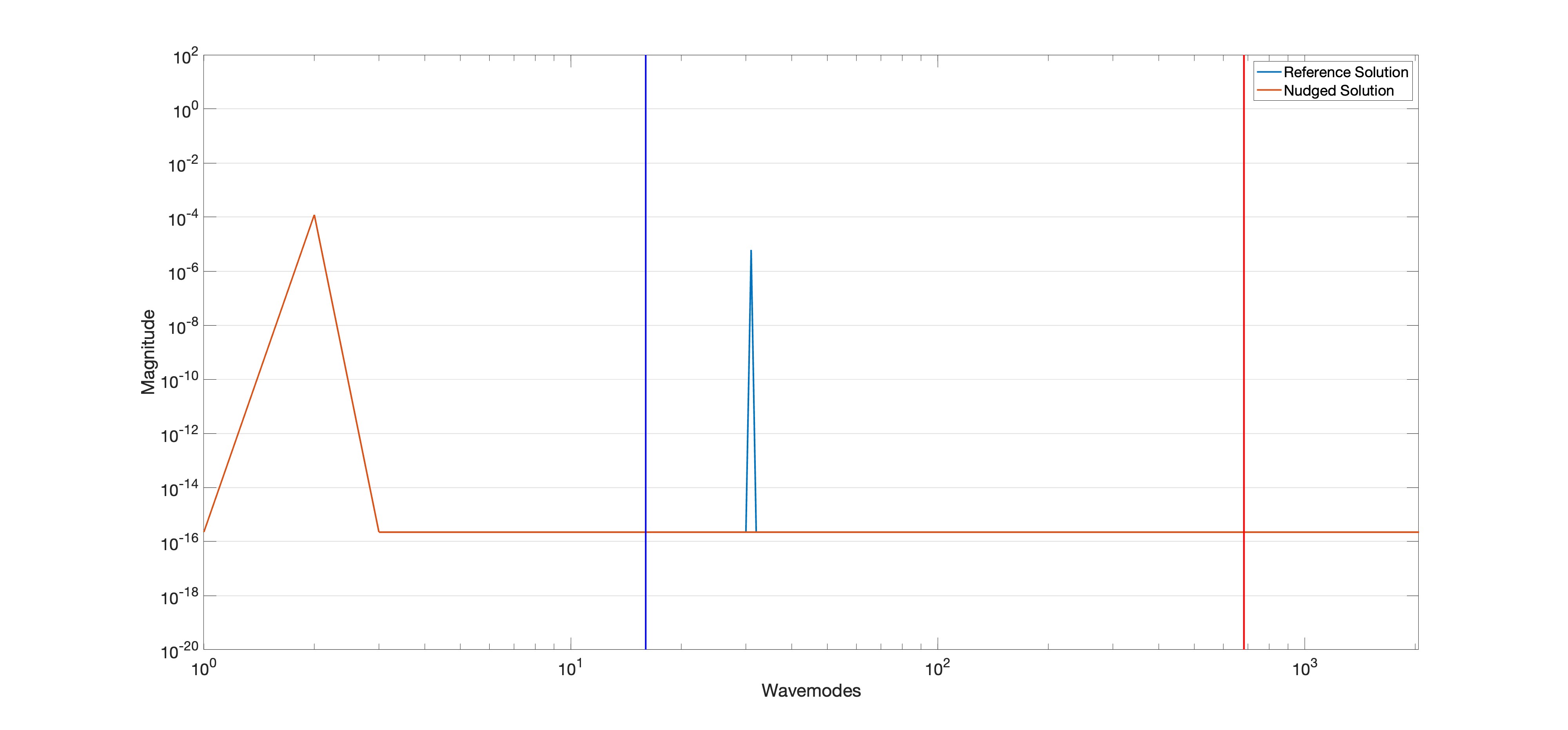}
        \caption{\small Inviscid Linear Transport equation}
        \label{fig:spec partial ob:ILT}
    \end{subfigure}
    \hfill
    \begin{subfigure}[b]{.49\textwidth}
        \centering
        \includegraphics[width=\textwidth]{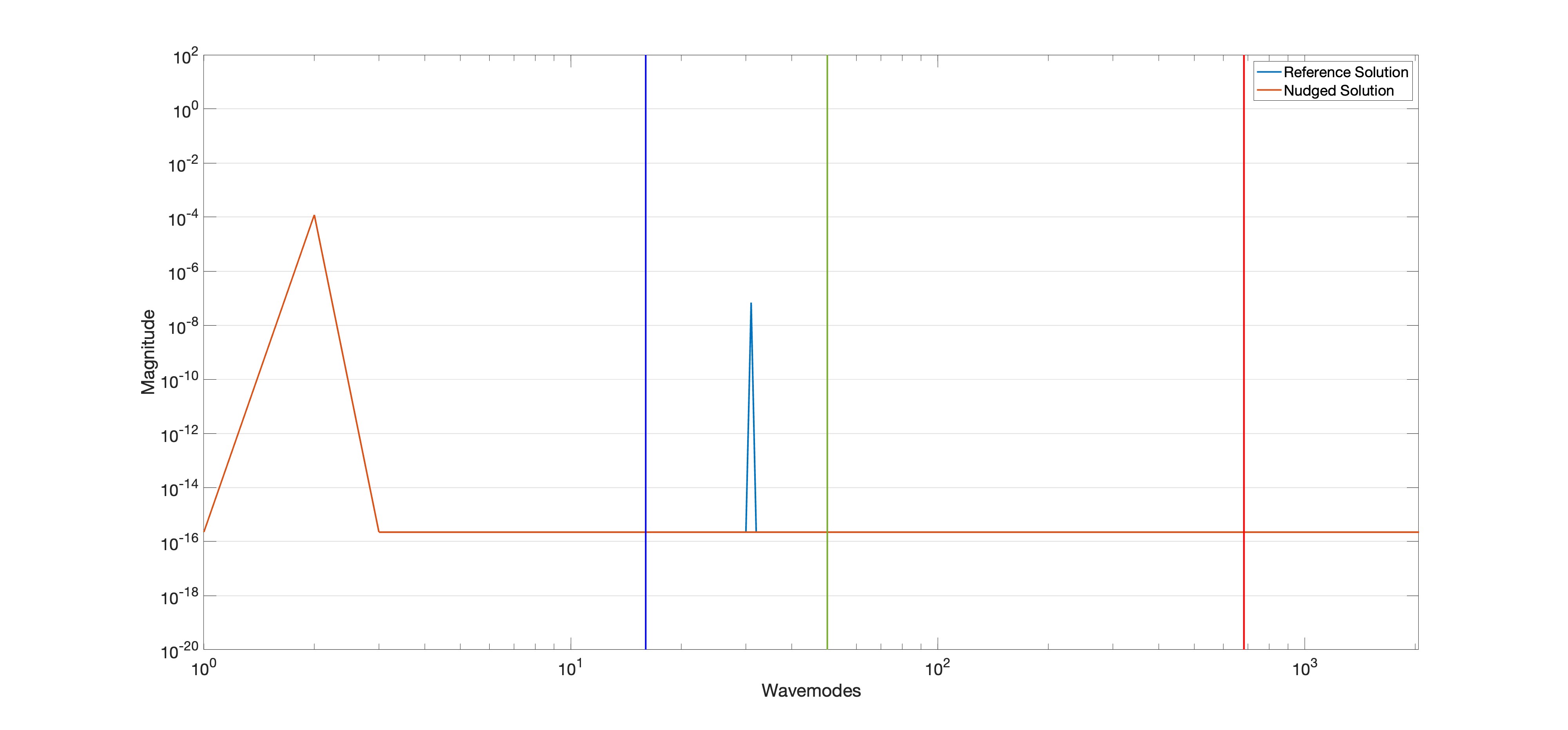}
        \caption{\small Viscous Linear Transport equation}
        \label{fig:spec partial ob:VLT}
    \end{subfigure}
    \caption{\small $L^2_{\text{per}}$ energy spectrum of the recovered initial condition using low-mode observations.}
    \label{fig:spec partial ob_LT}
\end{figure}

\begin{figure}
    \centering
    \begin{subfigure}[b]{.49\textwidth}
        \centering
        \includegraphics[width=\textwidth]{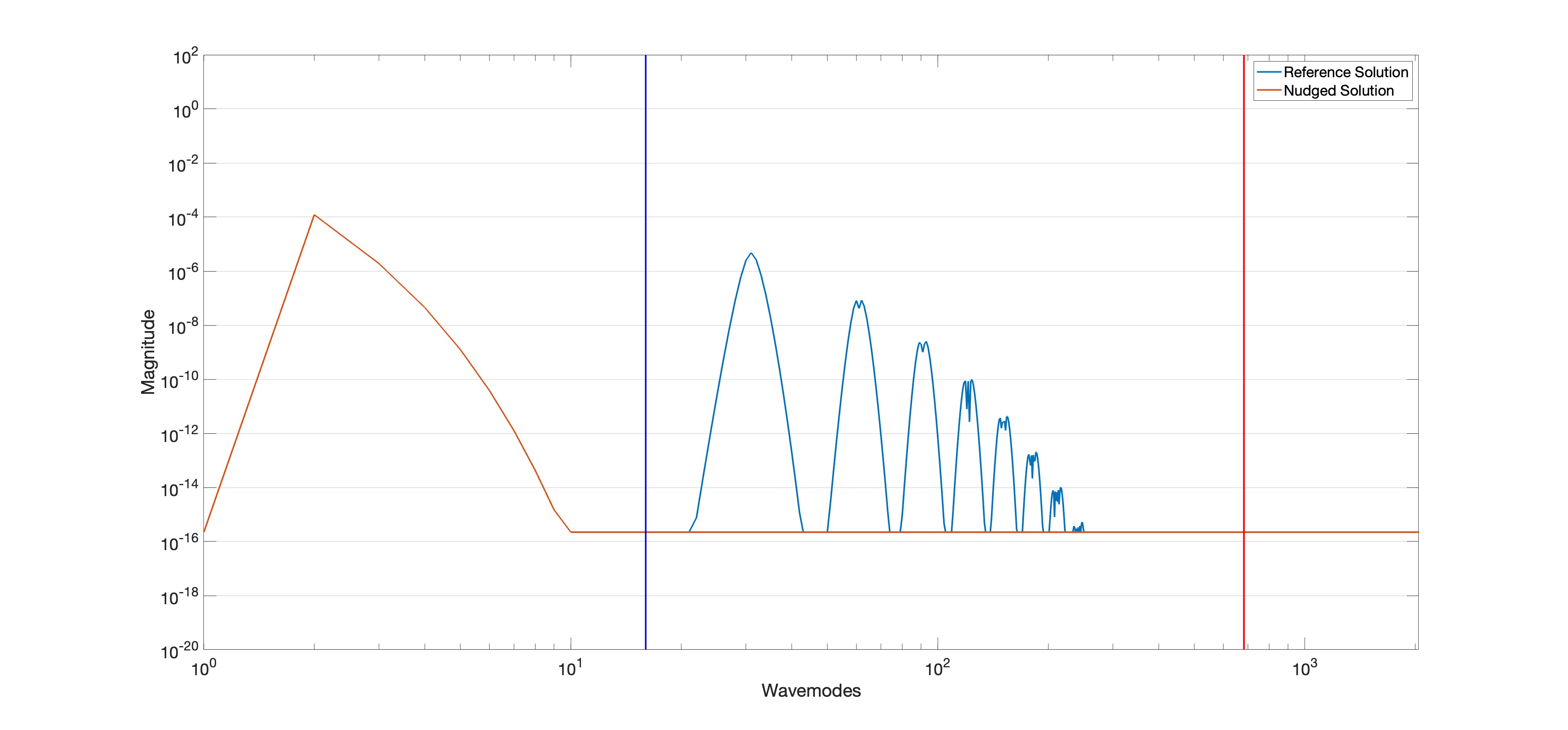}
        \caption{\small Inviscid Burgers equation}
        \label{fig:spec partial ob:IB}
    \end{subfigure}
    \hfill
    \begin{subfigure}[b]{.49\textwidth}
        \centering
        \includegraphics[width=\textwidth]{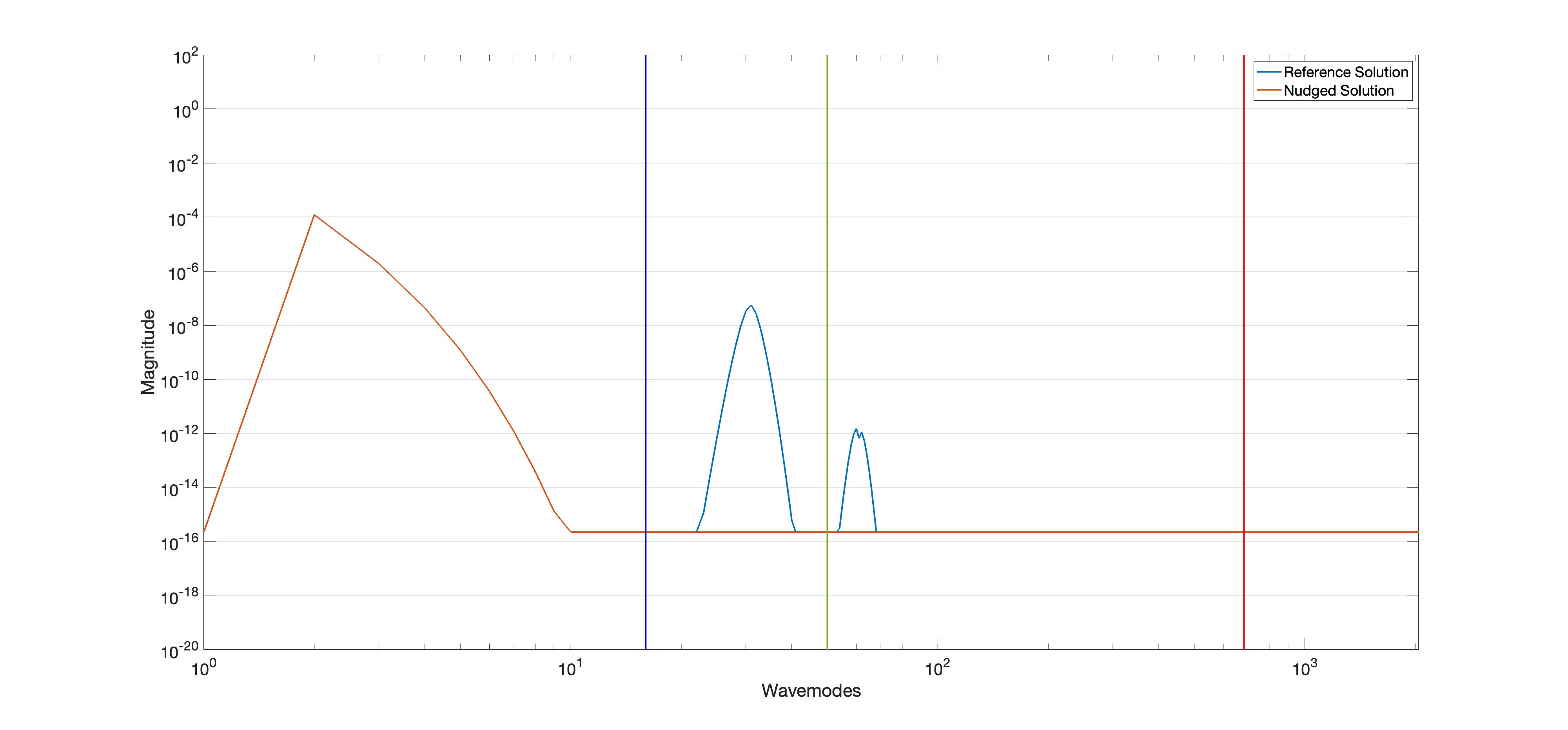}
        \caption{\small Viscous Burgers equation}
        \label{fig:spec partial ob:VB}
    \end{subfigure}
    \caption{\small $L^2_{\text{per}}$ energy spectrum of the recovered initial condition using low-mode observations.}
    \label{fig:spec partial ob_B}
\end{figure}

\begin{figure}
    \centering
    \begin{subfigure}[b]{.49\textwidth}
        \centering
        \includegraphics[width=\textwidth]{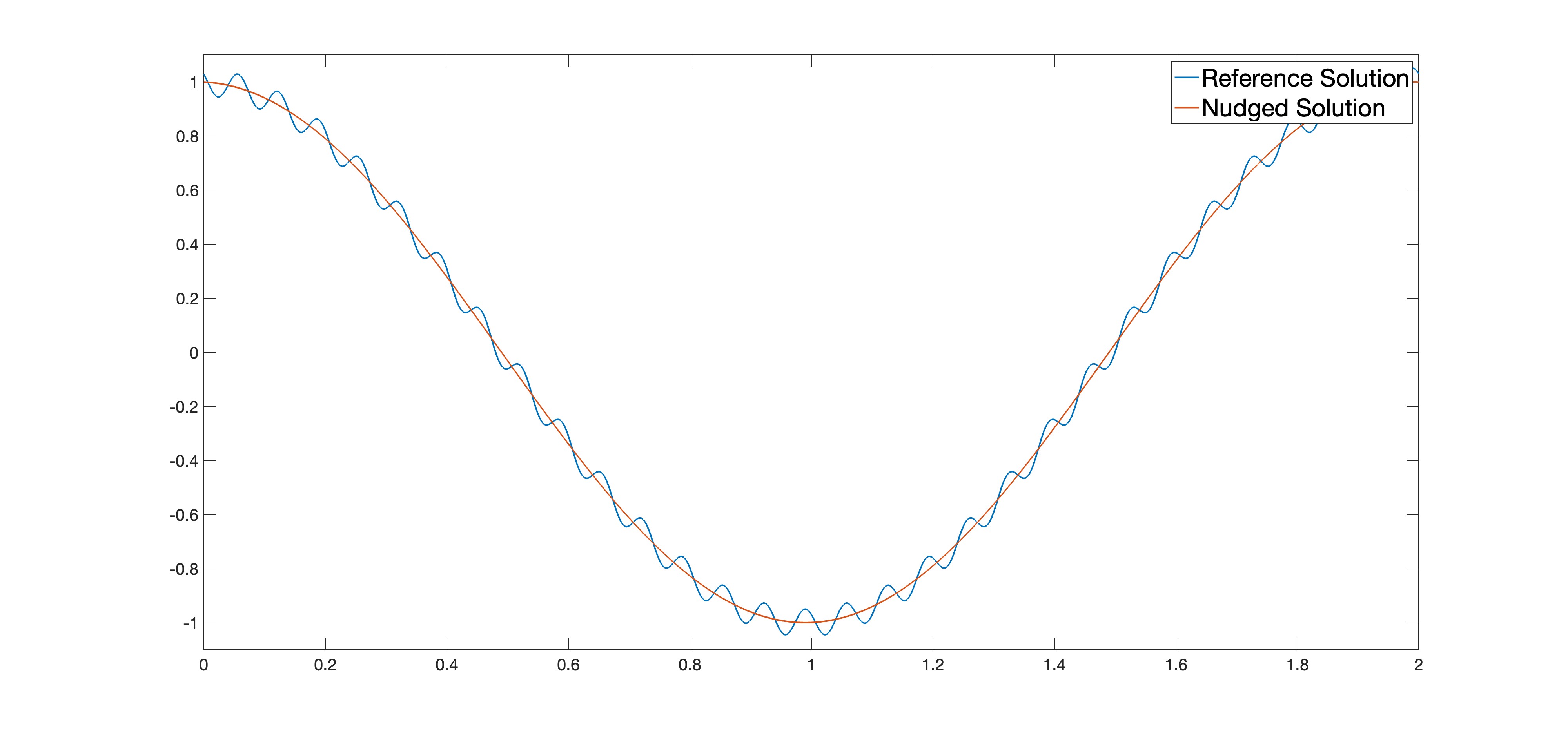}
        \caption{\small Inviscid Linear Transport equation}
        \label{fig:soln partial ob:ILT}
    \end{subfigure}
    \hfill
    \begin{subfigure}[b]{.49\textwidth}
        \centering
        \includegraphics[width=\textwidth]{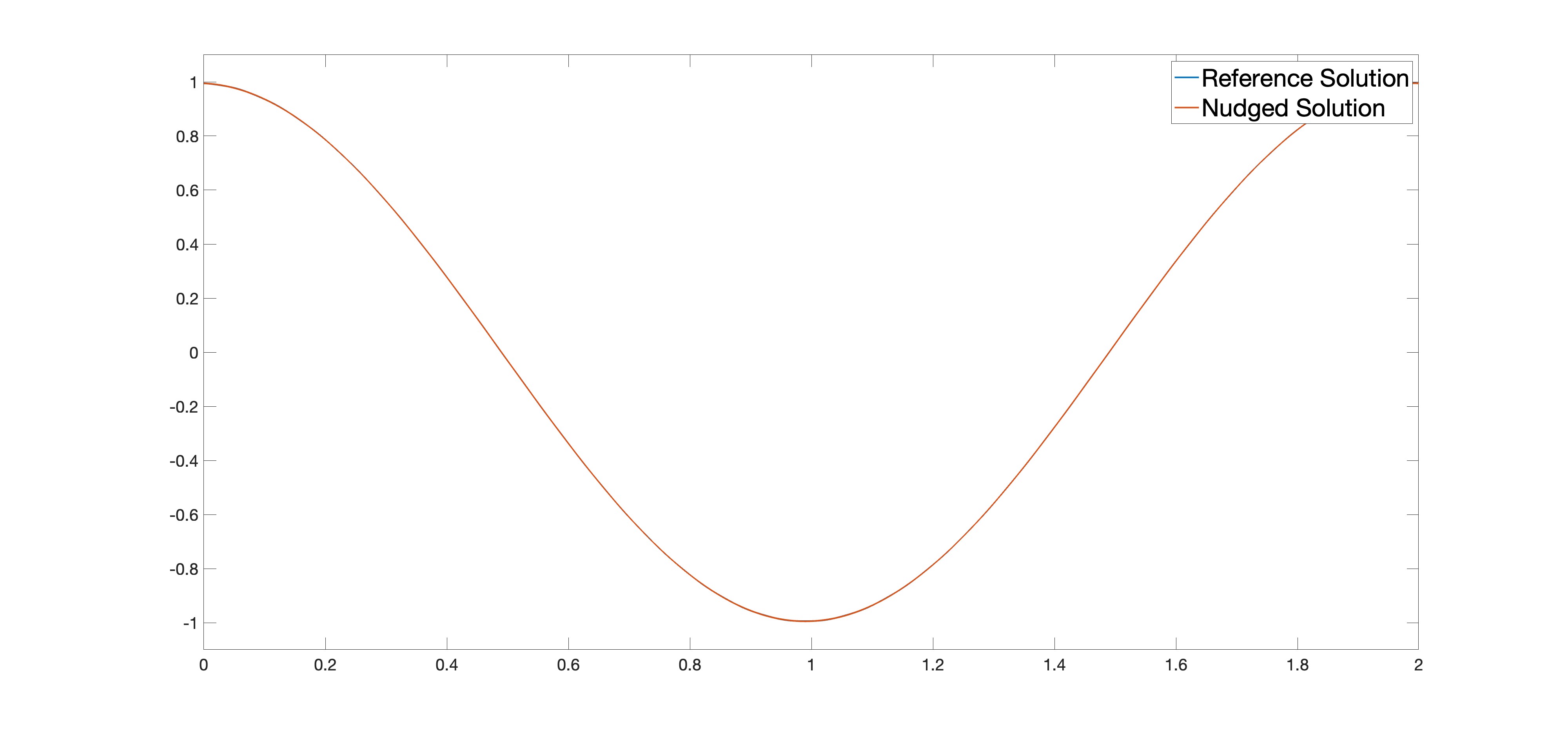}
        \caption{\small Viscous Linear Transport equation}
        \label{fig:soln partial ob:VLT}
    \end{subfigure}
    \caption{\small Recovered initial condition using low-mode observations.}
    \label{fig:soln partial ob_LT}
\end{figure}

\begin{figure}
    \centering
    \begin{subfigure}[b]{.49\textwidth}
        \centering
        \includegraphics[width=\textwidth]{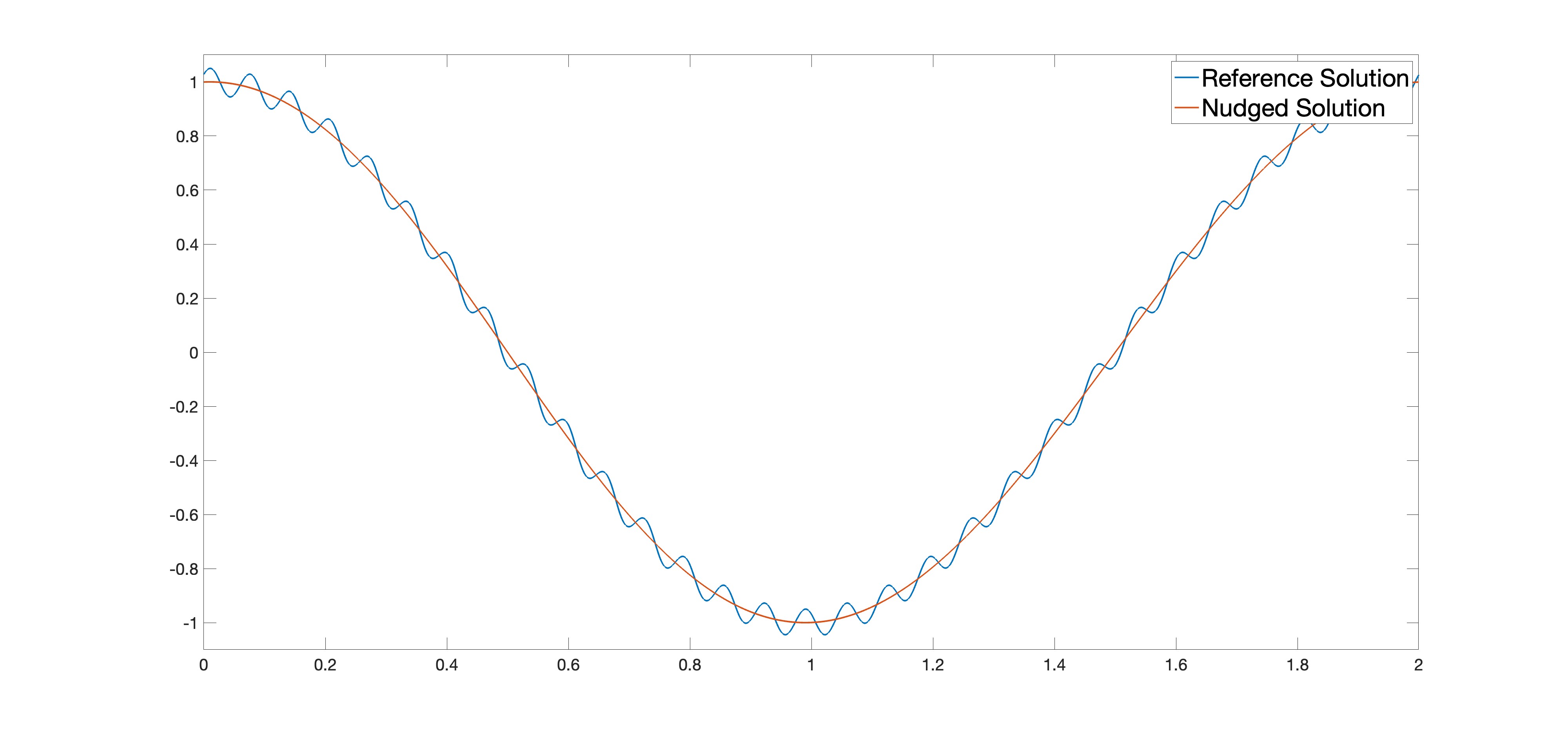}
        \caption{\small Inviscid Burgers equation}
        \label{fig:soln partial ob:IB}
    \end{subfigure}
    \hfill
    \begin{subfigure}[b]{.49\textwidth}
        \centering
        \includegraphics[width=\textwidth]{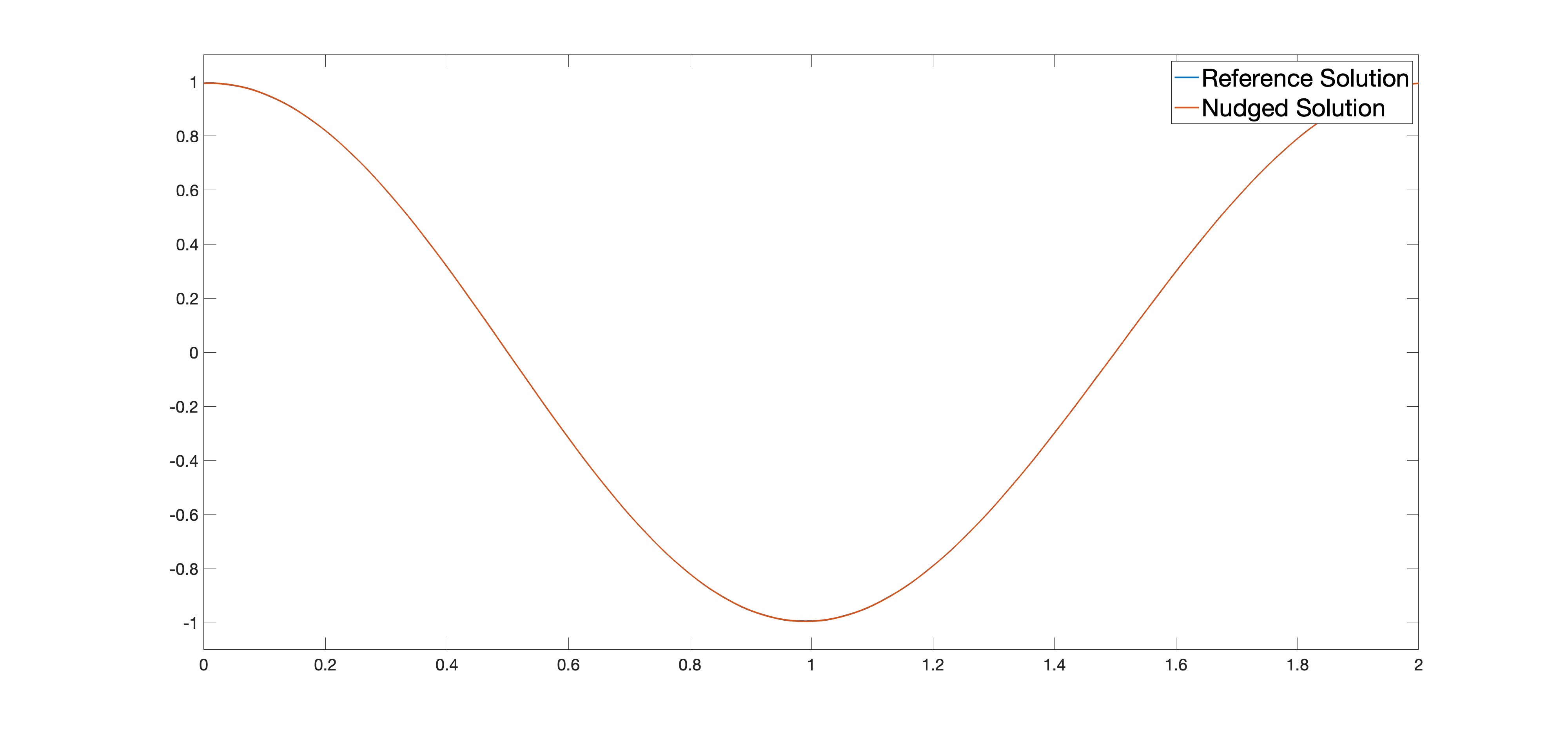}
        \caption{\small Viscous Burgers equation}
        \label{fig:soln partial ob:VB}
    \end{subfigure}
    \caption{\small Recovered initial condition using low-mode observations.}
    \label{fig:soln partial ob_B}
\end{figure}

\section{Diffusive and Voigt-Regularized BFN Algorithms}

Having established the fundamental limitations of the standard BFN algorithm using simplified 1D models, we now turn to practical strategies for modifying the algorithm to stabilize and regularize the backward evolution. In this section, we investigate two regularization variants: the established Diffusive BFN (dBFN) \cite{Auroux_Blum_Ruggiero_2016} and the newly proposed Voigt-regularized BFN (VBFN). We apply these methods to genuinely multiscale dynamics—specifically the 2D Navier-Stokes equations (NSE) and 1D viscous Korteweg-de Vries (KdV) equations—to determine the extent to which regularization can mitigate the ill-posedness inherent in the backward step.

\subsection{dBFN and VBFN for the 2D NSE}\label{sect:NSE}
We consider the 2D NSE given by \cref{eq:NSE} on the periodic domain $\Omega = [-\pi,\pi]^2$, subject to periodic boundary conditions. We will consider the case where the prescribed body force, $f$, is presumed to be constant in time.

Applying the BFN algorithm to the 2D NSE presents significant challenges due to the ill-posed nature of the backward-in-time equations. Specifically, reversing the sign of time renders the viscous term anti-diffusive, causing errors in high-frequency modes to grow exponentially and often resulting in numerical instability. To address this, we compare three variations of the backward nudging step: the Standard BFN, the Diffusive BFN (dBFN), and our newly proposed Voigt-regularized BFN (VBFN). The VBFN is the standard BFN algorithm applied to the Voigt-regularized NSE (\cref{eq:NSE:VOIGT}) to stabilize the backward step.

Representing the 2D NSE generically as in \cref{eq:generic system}, with the operator $F$:
\begin{equation}
    F(u) = \nu \Delta u - (u \cdot \nabla) u - \nabla p + f,
\end{equation}
the corresponding forward (F) and backward variations are given by:
\begin{subequations}
\begin{align}
\text{(F)}\quad &
\begin{cases}
    v_t^n = F(v^n) + \mu P_M(u - v^n),\\
    v^n(0) = v_0^n := \tilde{v}^{n-1}(0),
\end{cases}\label{eq:NSE:F}\\
\text{(Standard)}\quad &
\begin{cases}
    \tilde{v}_t^n = F(\tilde{v}^n) - \mu P_M(u - \tilde{v}^n),\\
    \tilde{v}^n(T) = \tilde{v}^n_T := v^n(T),
\end{cases}\label{eq:NSE:B}\\
\text{(Diffusive)}\quad &
\begin{cases}
    \tilde{v}_t^n = F(\tilde{v}^n) - 2\nu \Delta \tilde{v}^n - \mu P_M(u - \tilde{v}^n),\\
    \tilde{v}^n(T) = \tilde{v}^n_T := v^n(T),
\end{cases}\label{eq:NSE:diff-B}\\
\text{(Voigt)}\quad &
\begin{cases}
    (I - \alpha^2 \Delta) \tilde{v}_t^n = F(\tilde{v}^n) - \mu P_M(u - \tilde{v}^n),\\
    \tilde{v}^n(T) = \tilde{v}^n_T := v^n(T).
\end{cases}\label{eq:NSE:Voigt-B}
\end{align}
\end{subequations}
Here, $P_M$ denotes the projection onto Fourier modes with wavevectors $k \in \mathbb{Z}^2$ satisfying $|k| \leq M$.

Equations \cref{eq:NSE:B,eq:NSE:diff-B,eq:NSE:Voigt-B} represent distinct formulations of the backward pass. \Cref{eq:NSE:B} corresponds to the Standard BFN, which preserves the original dynamics but is susceptible to exponential instabilities in the high modes. \Cref{eq:NSE:diff-B} employs the Diffusive BFN (dBFN) strategy \cite{Auroux_Blum_2008, Auroux_2009}, in which the sign of the viscous term is reversed to enforce well-posedness, effectively smoothing the solution during the backward pass.

The VBFN algorithm in \cref{eq:NSE:Voigt-B} leverages the inviscid regularization framework \cite{oskolkov1973, Cao_Lunasin_Titi_2006, kalantarov2009}. The primary motivation for this approach is the global well-posedness of Voigt-regularized equations, \cref{eq:NSE:VOIGT}, in both forward and backward time \cite{Cao_Lunasin_Titi_2006, kalantarov2009}. In contrast to dBFN—which achieves stability via artificial viscosity—VBFN maintains the physical dissipative structure while using the regularization term $(I - \alpha^2 \Delta)$ to mollify the nonlinear term and inhibit the cascade of energy to higher modes \cite{Levantetal2010}. By tempering the growth rate of small scales, VBFN ensures numerical stability and prevents blow-up while preserving the turbulent energy spectrum up to a wavenumber proportional to $1/\alpha$.

Finally, we introduce two alternative algorithms that modify the dBFN and VBFN schemes by restricting the regularization to unobserved spectral modes. These \textit{Filtered} variants are defined as follows:
\begin{align}
\text{(Filtered Diffusive)}\quad &
\begin{cases}
    \tilde{v}_t^n = F(\tilde{v}^n) - 2\nu \Delta Q_M \tilde{v}^n - \mu P_M(u - \tilde{v}^n),\\
    \tilde{v}^n(T) = \tilde{v}^n_T := v^n(T),
\end{cases}\label{eq:NSE:fdiff-B}\\
\text{(Filtered Voigt)}\quad &
\begin{cases}
    (I - \alpha^2 \Delta Q_M) \tilde{v}_t^n = F(\tilde{v}^n) - \mu P_M(u - \tilde{v}^n),\\
    \tilde{v}^n(T) = \tilde{v}^n_T := v^n(T),
\end{cases}\label{eq:NSE:fVoigt-B}
\end{align}
where $Q_M = I - P_M$ is the orthogonal projection onto unobserved modes. While numerical stability strictly requires regularization only on the highest wavenumbers, it is natural to align the spectral filter with the observational cutoff $M$. This modification aims to improve the reconstruction of the initial condition by eliminating model error on the observed Fourier modes while maintaining the stabilizing effects of the dBFN and VBFN algorithms on unobserved modes.

\subsubsection{Numerical methods and experimental setup}\label{sect:NSE:numerics} 
Numerical simulations for this study were conducted using a fully dealiased pseudo-spectral code based on the vorticity-stream function formulation. Specifically, the 2D NSE \cref{eq:NSE} is is equivalent to its vorticity formulation given by:
\begin{subequations}
\begin{align}
\psi_t + \Delta^{-1}(u \cdot \nabla \omega) &= \nu \Delta \psi + \Delta^{-1} \nabla^{\perp} \cdot f, \\
u := \nabla^{\perp}\psi = \binom{\partial_y\psi}{-\partial_x\psi} \quad &\text{and} \quad \omega := \partial_{x}u_2 - \partial_yu_1 = -\Delta \psi.
\end{align}
\end{subequations}
The inverse Laplacian $\Delta^{-1}$ is defined with respect to periodic boundary conditions under the mean-free constraint. For an extensive overview of pseudo-spectral methods, we refer the reader to, e.g., \cite{Canuto_Hussaini_Quarteroni_Zang_2006, Peyret_2013_spectral_book, Shen_Tang_Wang_2011}.

Spatial derivatives were computed in the spectral domain, while time integration was performed using a first-order integrating factor Euler scheme. The linear dissipative term was integrated exactly using an exponential integrating factor, while the nonlinear advection and nudging terms were treated explicitly. A constant time step of $\Delta t = 0.001$ was employed. Initial data for the reference solution $u$ was generated by evolving the system from a zero state until a statistically steady state was reached at $t = 25,000$. As illustrated in \cref{fig:Spectrum:NSE}, the energy spectrum of the initial data is well-resolved and remained so throughout the duration of the simulation.

The simulations utilized a spatial resolution of $N^2 = 1,024^2$, which was sufficient to resolve the energy spectrum of the forward simulations to machine precision before the 2/3 dealiasing cutoff. Following the parameters in \cite{Biswas_Bradshaw_Jolly2023}, we set the kinematic viscosity to $\nu = 10^{-4}$ and use the same forcing $f$ prescribed in \cite{Gesho_Olson_Titi_2015}. This forcing was scaled to achieve a Grashof number of $G = 10^6$, where $G = \frac{\|f\|_{L^2}}{\nu^2}$ is the dimensionless ratio of external forcing to viscous forces.

To evaluate the data assimilation methods, we continuously observed the solution $u$ of \cref{eq:NSE} over the time interval $t \in [0, 0.1]$. Observations were restricted to low-mode Fourier components satisfying $|\mathbf{k}| \leq M$, with a truncation level of $M=50$. For the VBFN simulations, we selected a regularization parameter of $\alpha = 10^{-3}$, determined to be sufficient through preliminary testing. The results are presented in \cref{fig:L2 error:50 modes: 3 methods,fig:Ens error:50 modes: 3 methods,fig:error:all modes: 3 methods}. In these trials, even iterations correspond to the error evolution of the forward algorithm, while odd iterations correspond to the backward algorithm.

To quantify reconstruction accuracy, we monitor the error in both the $L^2$ norm and the $H^1$ semi-norm, representing errors at the energy and enstrophy levels, respectively. We decompose these metrics into contributions from the observed modes ($|\mathbf{k}| \leq M$) and the unobserved modes ($|\mathbf{k}| > M$) to isolate the algorithm's capability to reconstruct the missing portion of the spectrum.

\subsubsection{Computational Results: 2D NSE}

In our short-interval simulations, the Standard BFN algorithm performs comparably to the VBFN. While the inherent instability of the backward evolution equations is expected to induce significant errors over longer time intervals of simulation, this effect can be mitigated in practice by restricting the length of the assimilation window.

However, distinct differences emerge when analyzing the error distribution across the observed scales. As shown in \cref{fig:L2 error:50 modes: 3 methods}, the dBFN algorithm introduces significant errors into the observed Fourier modes. This degradation is anticipated, as dBFN fundamentally alters the model physics; initializing the algorithm with the exact solution of \cref{eq:NSE} would still yield an inaccurate estimate after a single iteration due to the model error introduced by the reversed viscosity. While this discrepancy could potentially be managed by tuning the viscosity parameter in \cref{eq:NSE:diff-B}, we adhere to the standard formulation in this study.

We observe that VBFN achieves lower error levels in both the $L^2$ norm and enstrophy compared to dBFN. Nevertheless, none of the methods achieve accuracy in the observed modes below $10^{-6}$ in these simulations. It is worth noting that the error in the observed modes could likely be further reduced by implementing the nudging term implicitly and increasing the nudging parameter $\mu$, though we do not pursue that optimization here. 

The model discrepancy is further highlighted in the limiting case where the entire solution is observed. As seen in \cref{fig:error:all modes: 3 methods}, VBFN outperforms dBFN even when full observations are available. While full-state observation renders data assimilation unnecessary in practice, this comparison serves to illustrate that the dBFN algorithm introduces significantly larger model errors than the VBFN.

To address the bias inherent in the dBFN approach, we examine the filtered variants. As shown in \cref{fig:L2 error:50 modes: filtered,fig:Ens error:50 modes: filtered}, the Filtered dBFN algorithm yields a significant reduction in low-mode error for both energy and enstrophy compared to the standard dBFN, effectively mitigating the model error on the resolved scales.

Crucially, these numerical results for the 2D Navier-Stokes equations reinforce the fundamental limitations identified in our simpler models. Even when the BFN algorithm is stabilized by VBFN or dBFN—and further refined with spectral filtering—it remains unable to reconstruct the unobserved high-frequency wavemodes. As evidenced by the high-mode error evolution in \cref{fig:L2 error:50 modes: filtered,fig:Ens error:50 modes: filtered}, none of the methods provide significant improvement in recovering the unobserved portion of the initial condition spectrum. Instead, the error in the high modes oscillates with each iteration without converging toward the true initial state.

\begin{figure}
  \centering
	\includegraphics[width=.9\textwidth]{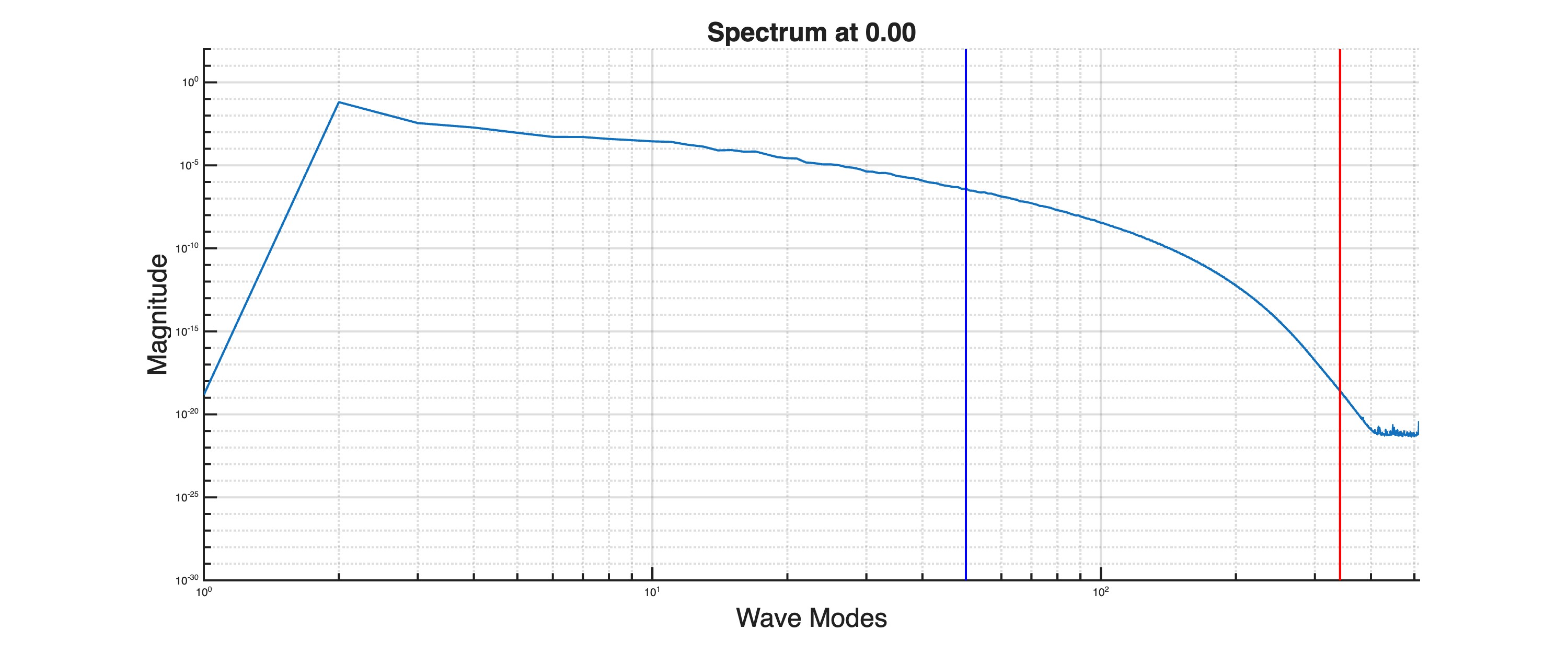}
\caption{\small Energy spectrum of the initial data. The vertical red line is the 2/3's dealiasing cutoff at $\frac23\frac{N}{2}=341.\overline{3}$.}
 \label{fig:Spectrum:NSE}
\end{figure}

 \begin{figure}
 \begin{subfigure}[b]{.49\textwidth}
  \centering
	\includegraphics[width=\textwidth]{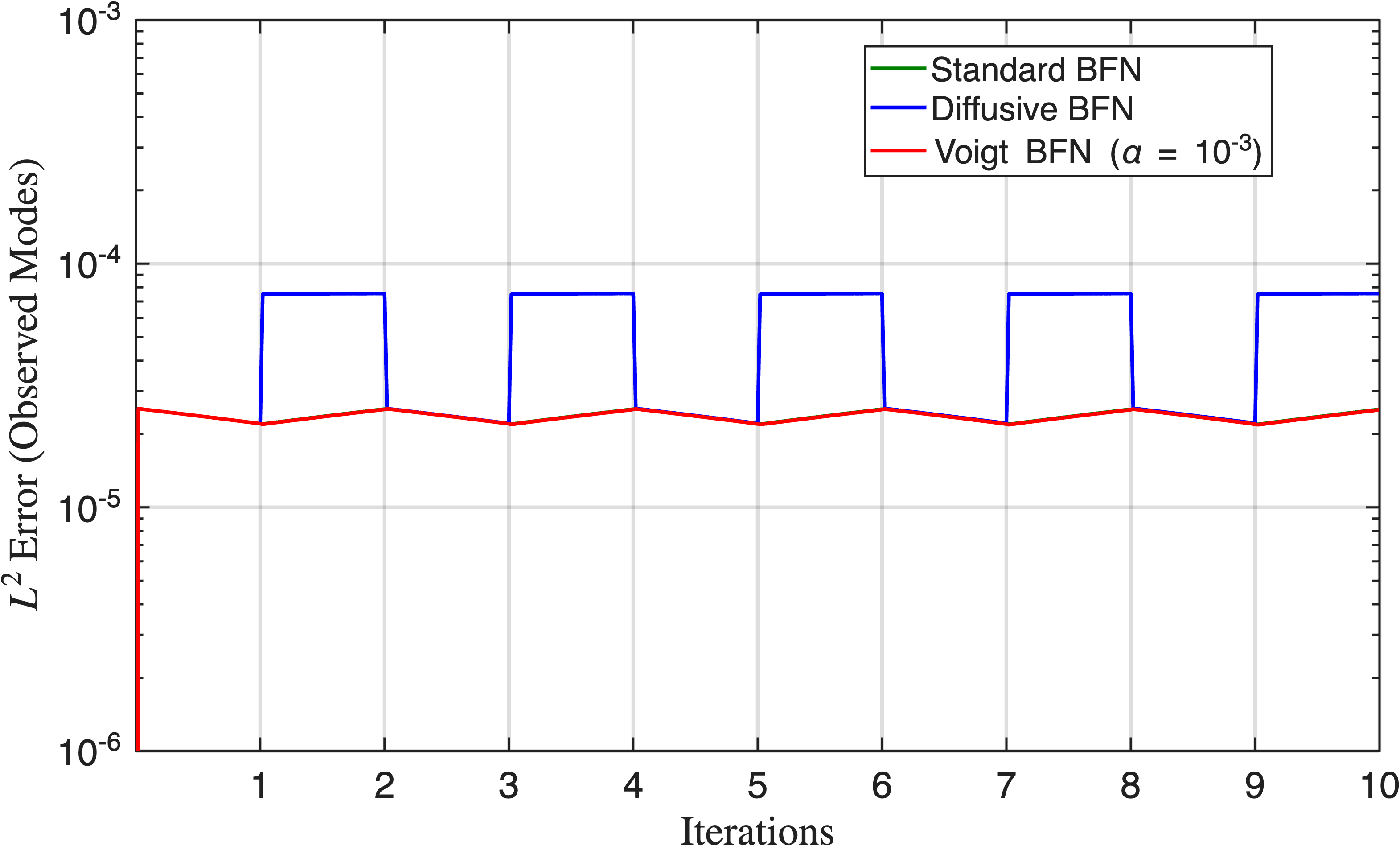}
 \end{subfigure}
 \begin{subfigure}[b]{.49\textwidth}
  \centering
	\includegraphics[width=\textwidth]{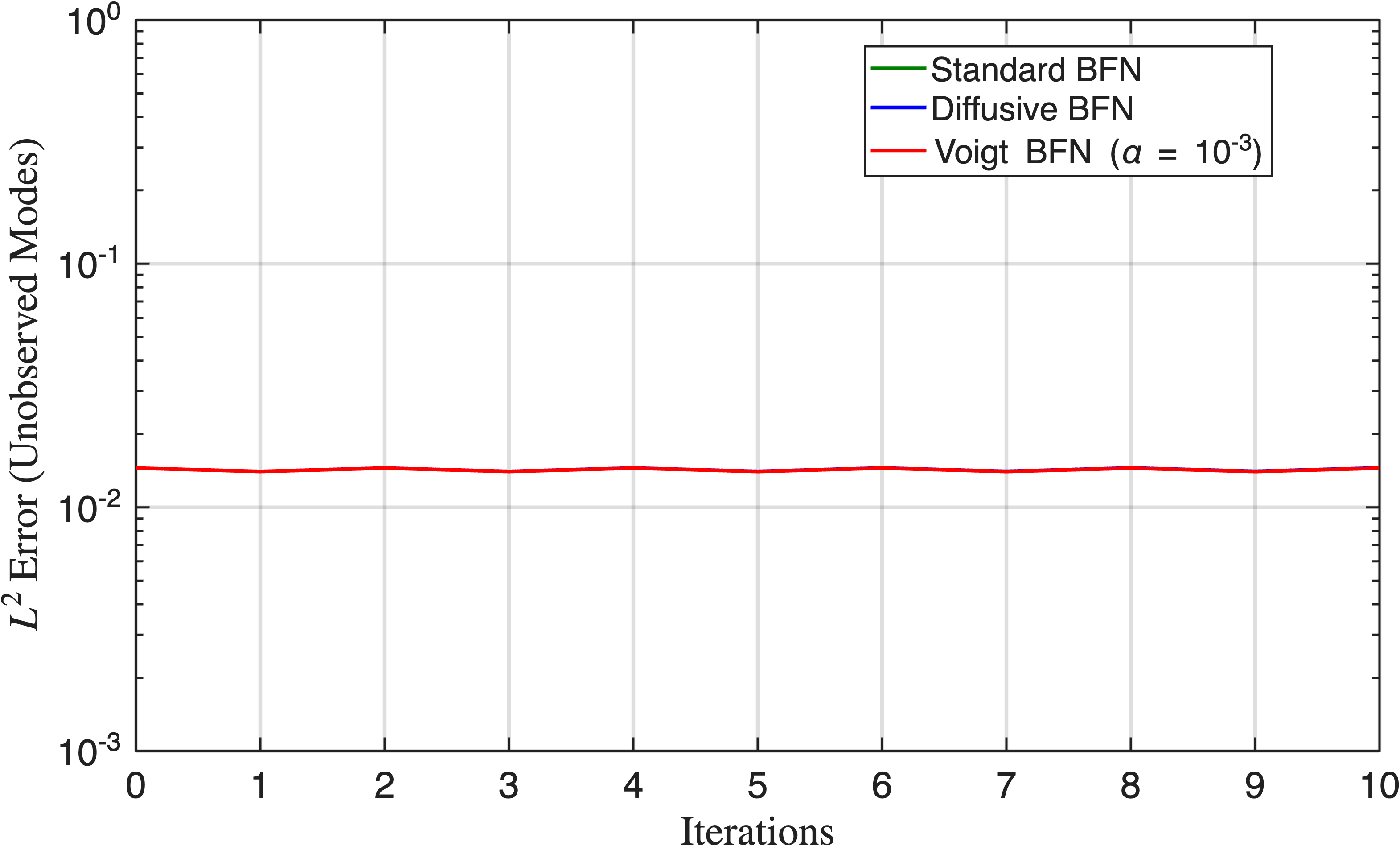}
 \end{subfigure}
 \caption{ \small $L^2$ error evolution for 2D NSE.}
 \label{fig:L2 error:50 modes: 3 methods}
\end{figure}

 \begin{figure}
 \begin{subfigure}[b]{.49\textwidth}
  \centering
	\includegraphics[width=\textwidth]{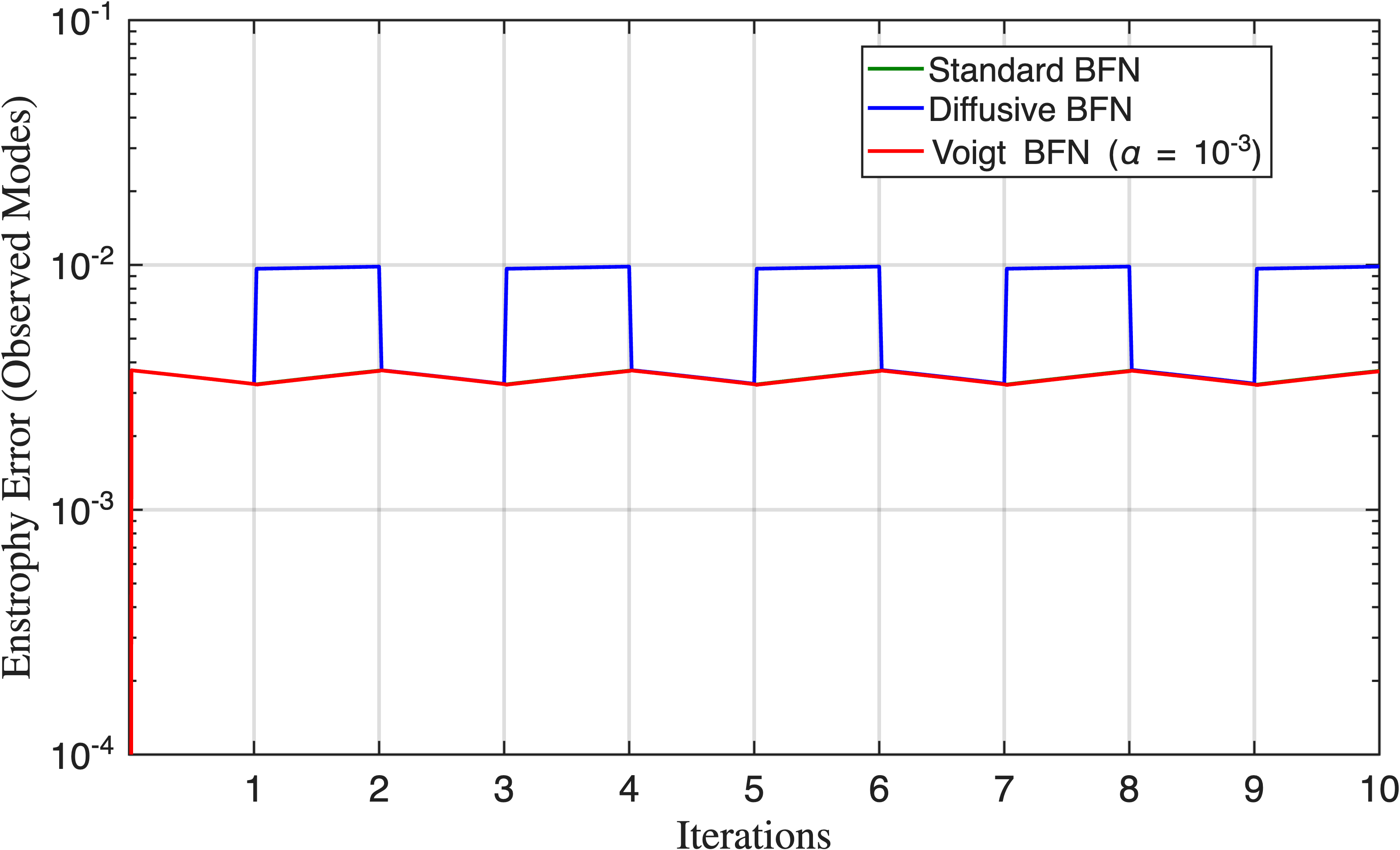}
 \end{subfigure}
 \begin{subfigure}[b]{.49\textwidth}
  \centering
	\includegraphics[width=\textwidth]{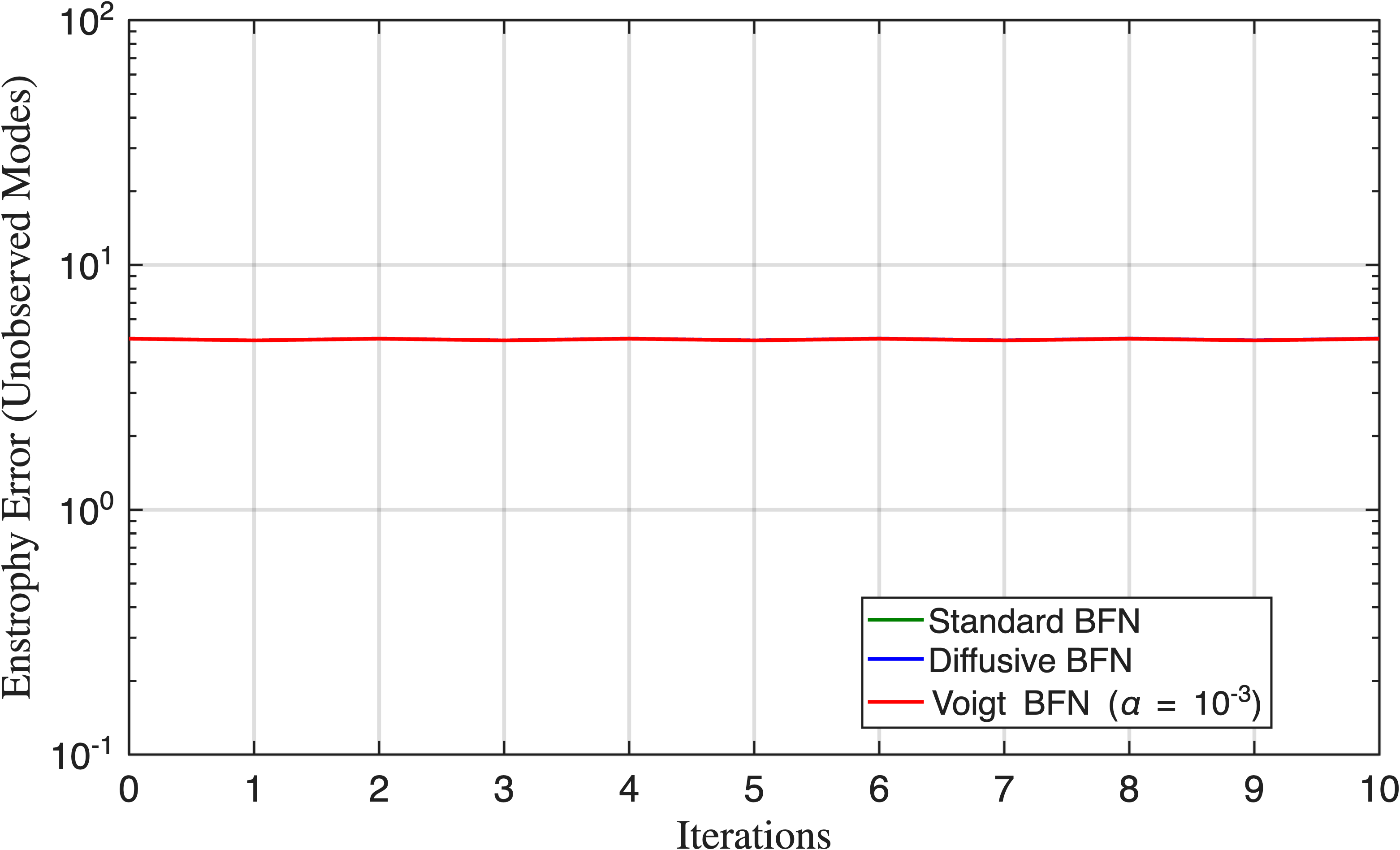}
 \end{subfigure}
 \caption{ \small Enstrophy error evolution for 2D NSE. For observations.}
 \label{fig:Ens error:50 modes: 3 methods}
\end{figure}

 \begin{figure}
 \begin{subfigure}[b]{.49\textwidth}
  \centering
	\includegraphics[width=\textwidth]{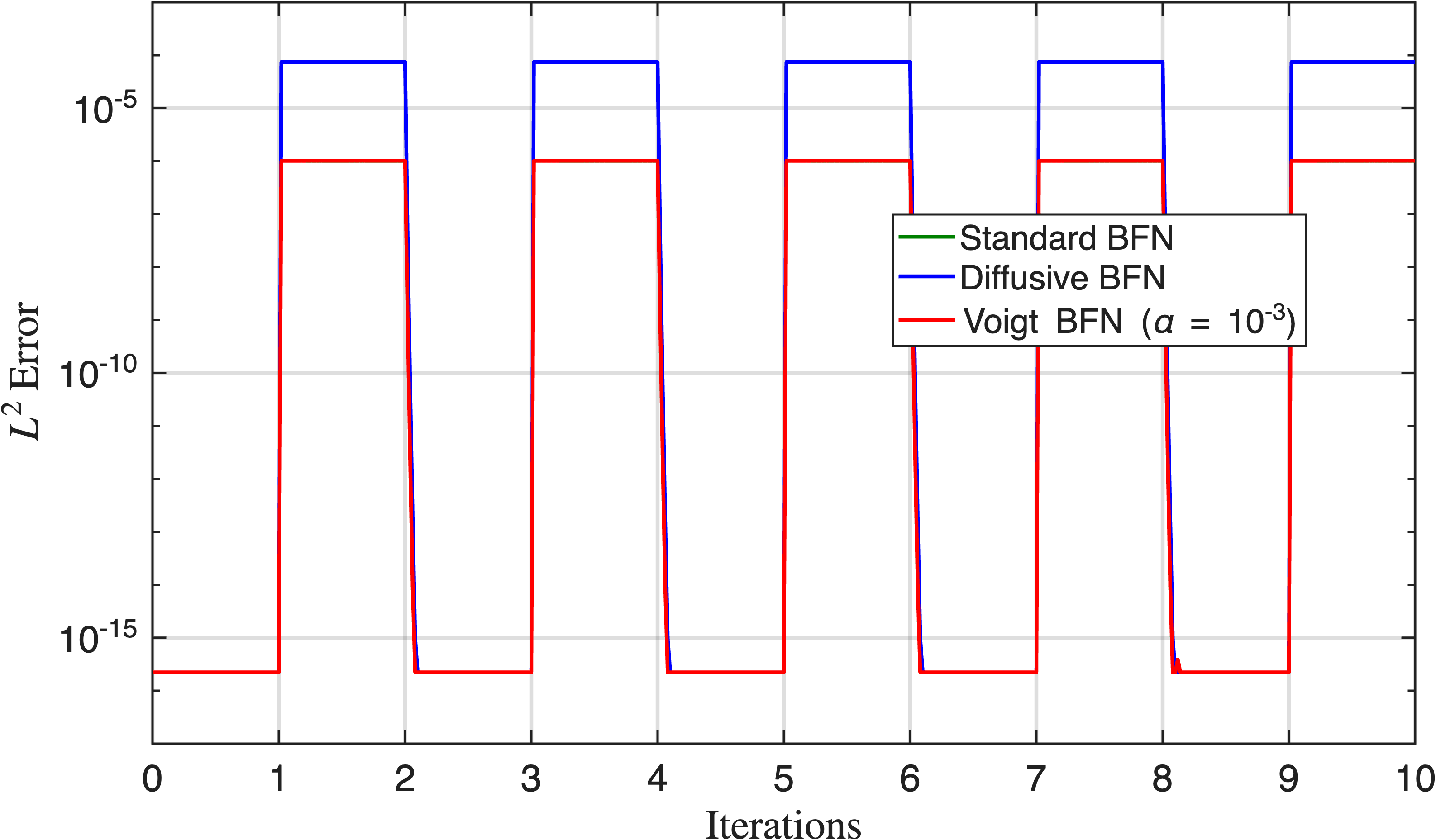}
 \end{subfigure}
 \begin{subfigure}[b]{.49\textwidth}
  \centering
	\includegraphics[width=\textwidth]{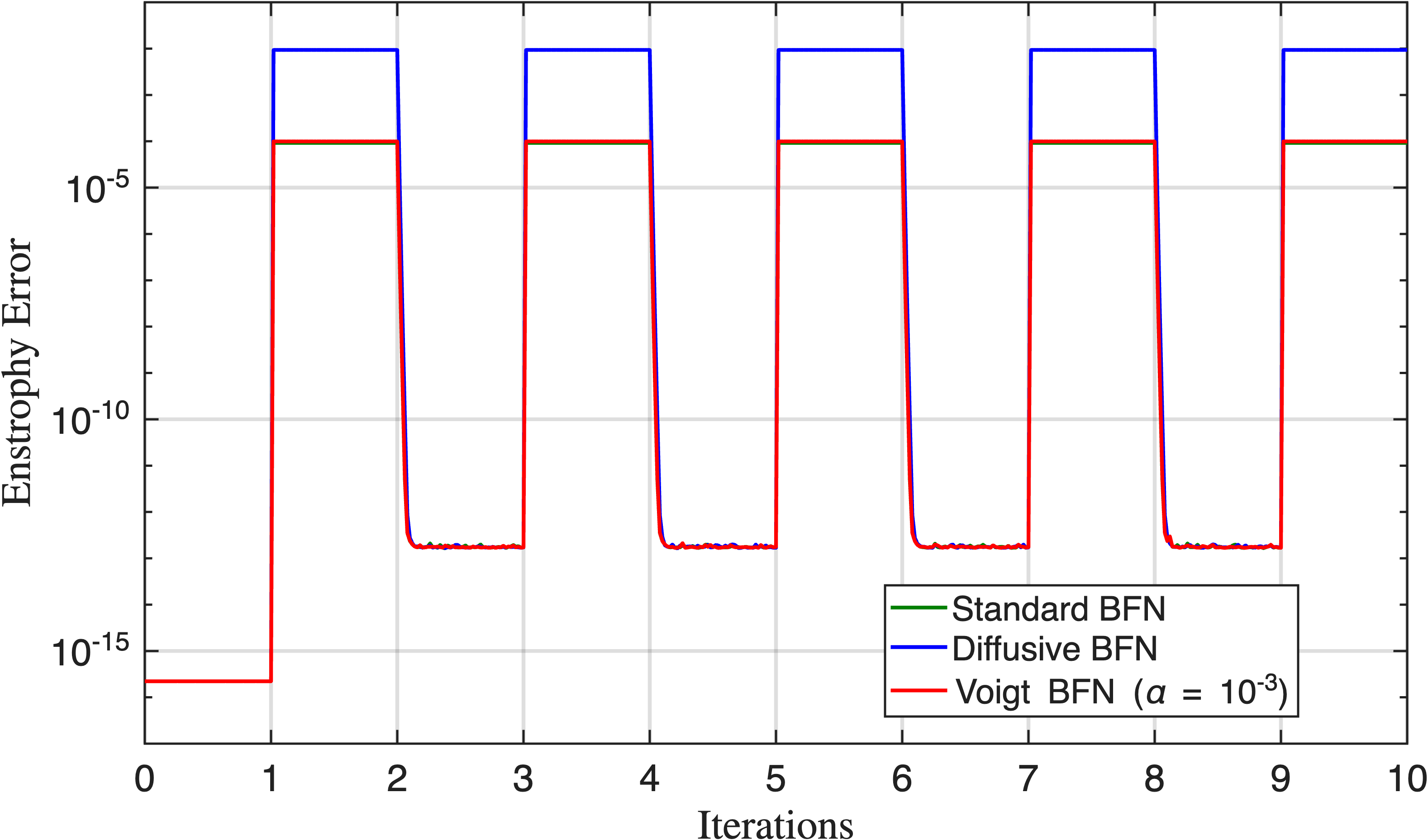}
 \end{subfigure}
 \caption{ \small $L^2$ and Enstrophy error evolution for 2D NSE. The entire solution was observed.}
 \label{fig:error:all modes: 3 methods}
\end{figure}

 \begin{figure}
 \begin{subfigure}[b]{.49\textwidth}
  \centering
	\includegraphics[width=\textwidth]{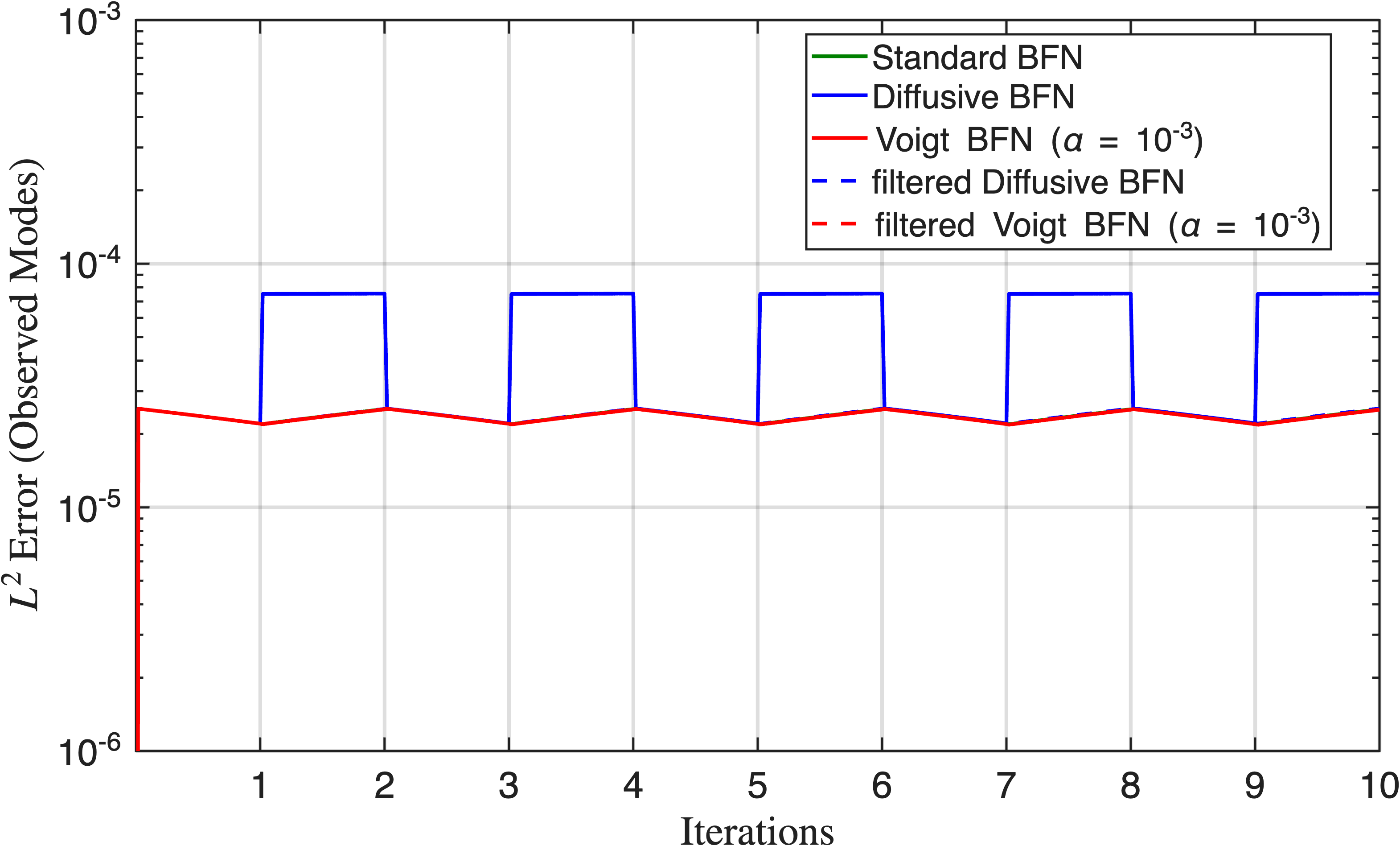}
 \end{subfigure}
 \begin{subfigure}[b]{.49\textwidth}
  \centering
	\includegraphics[width=\textwidth]{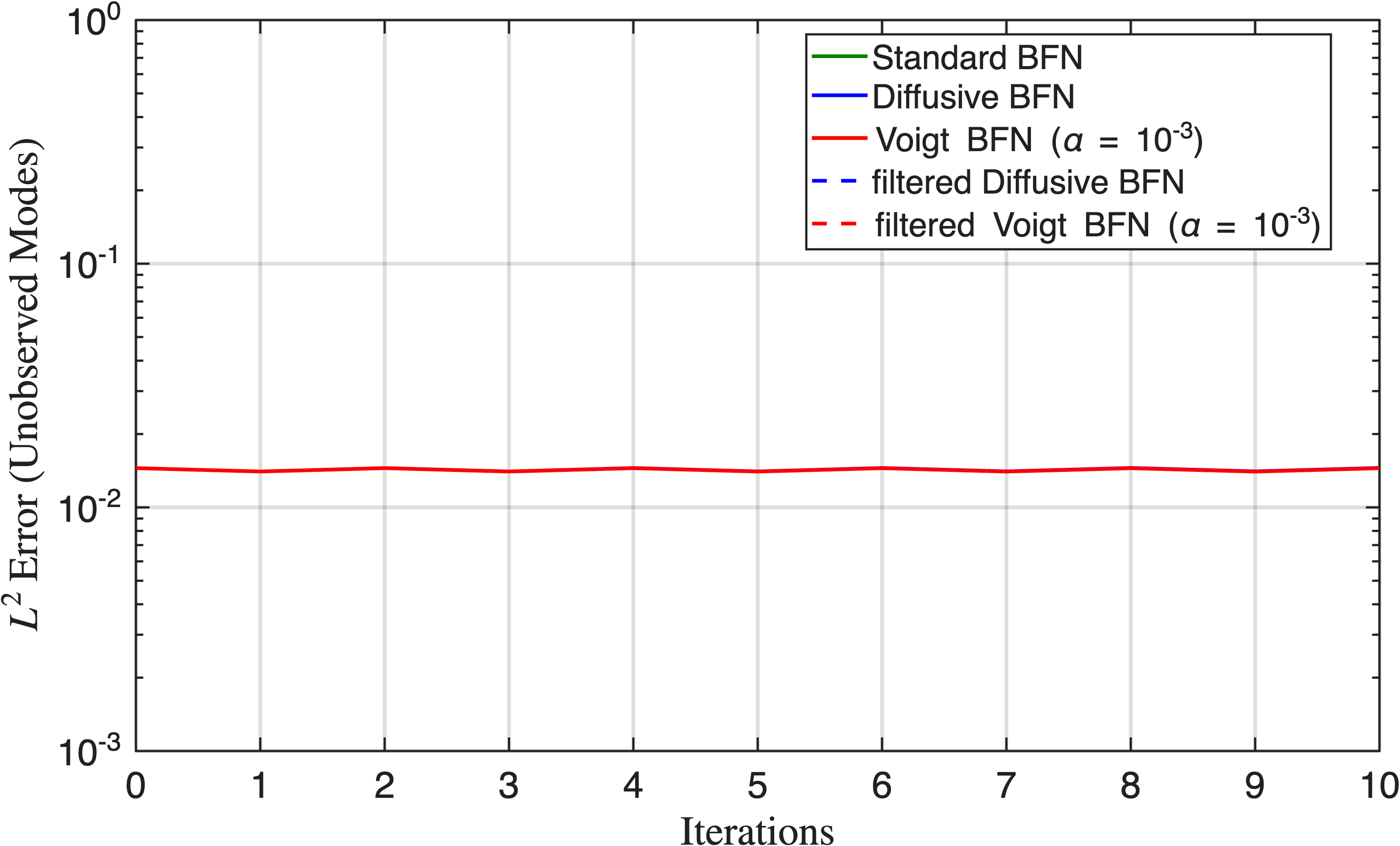}
 \end{subfigure}
 \caption{ \small $L^2$ error evolution for 2D NSE. }
 \label{fig:L2 error:50 modes: filtered}
\end{figure}

 \begin{figure}
 \begin{subfigure}[b]{.49\textwidth}
  \centering
	\includegraphics[width=\textwidth]{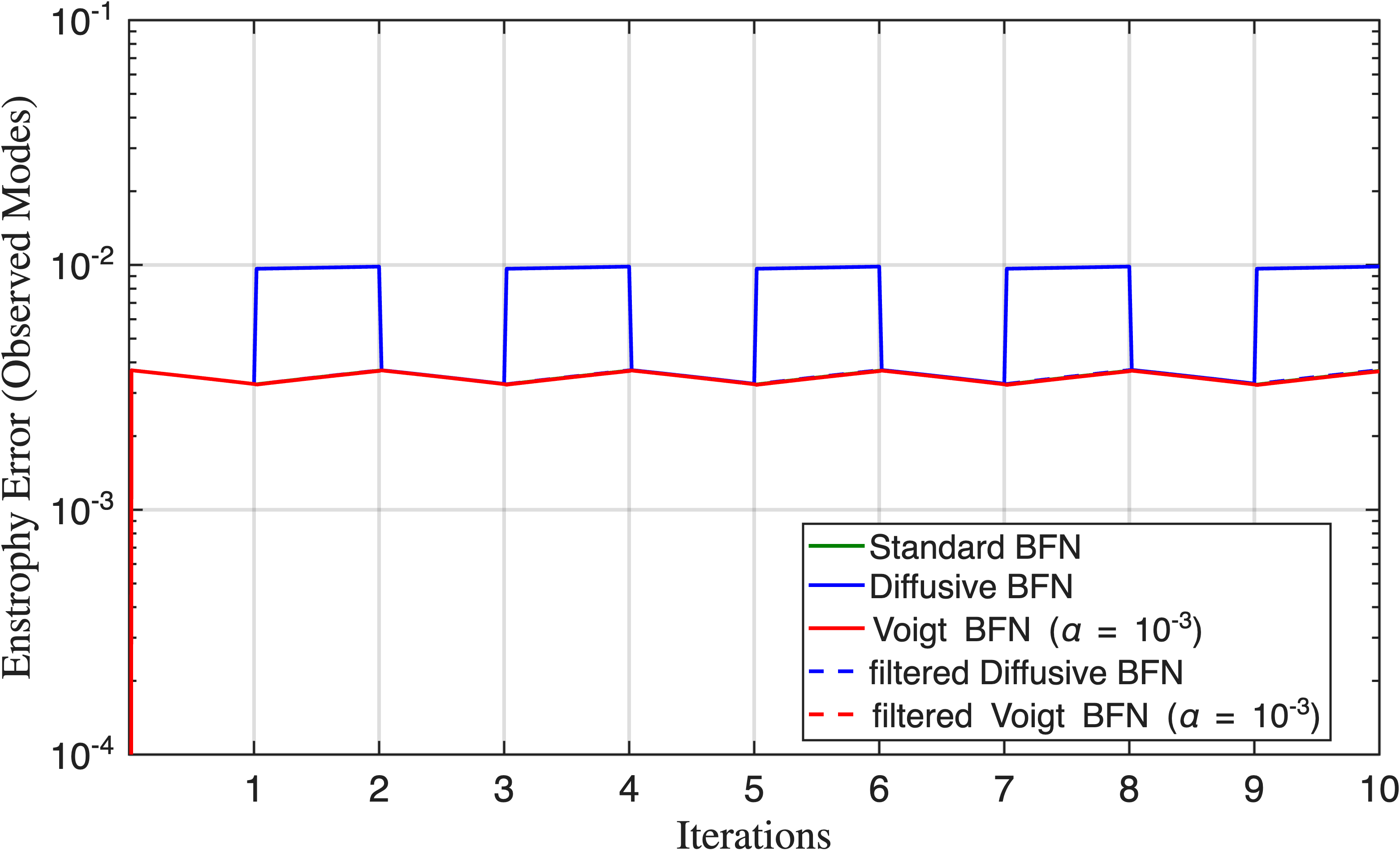}
 \end{subfigure}
 \begin{subfigure}[b]{.49\textwidth}
  \centering
	\includegraphics[width=\textwidth]{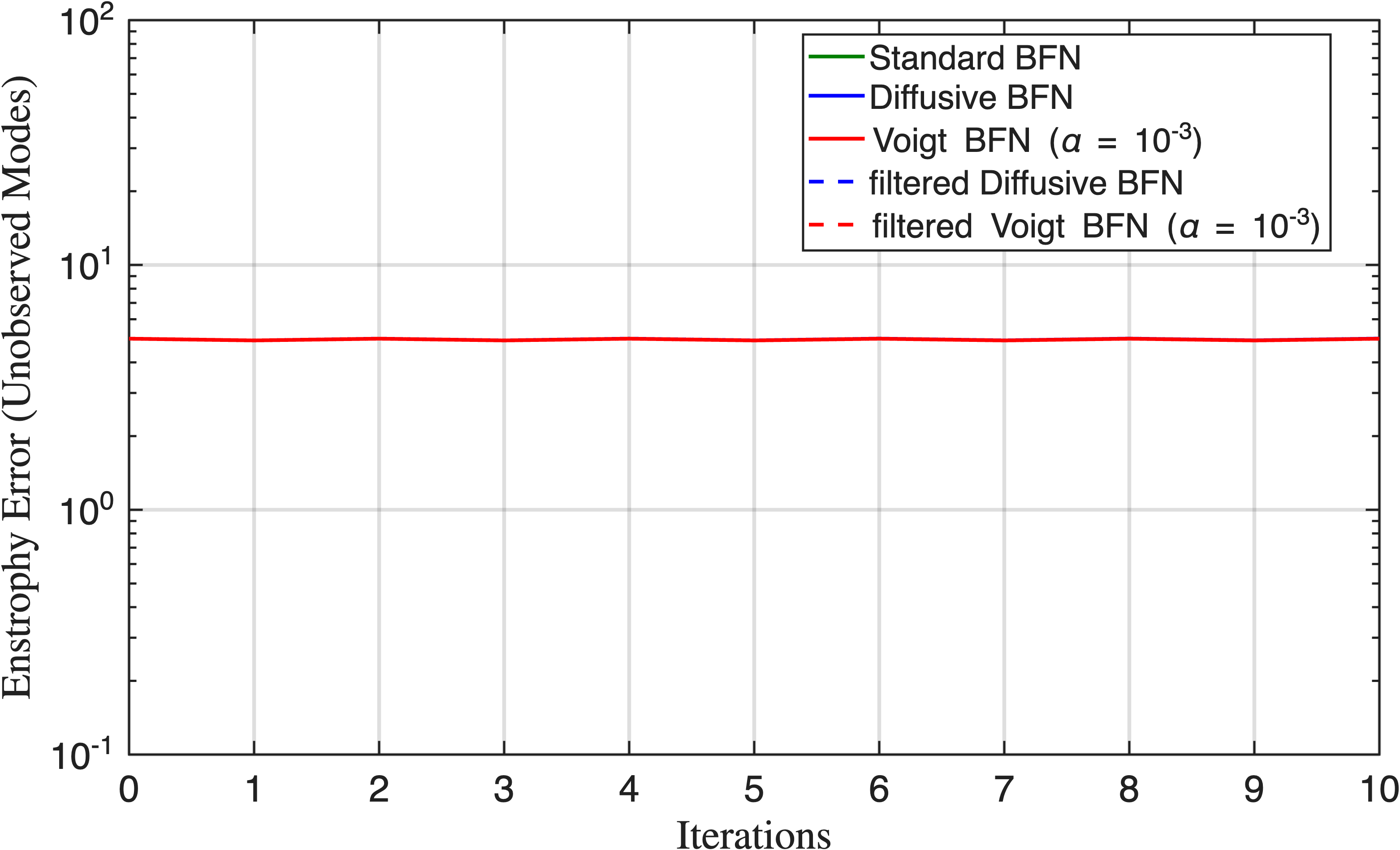}
 \end{subfigure}
 \caption{ \small Enstrophy error evolution for 2D NSE.}
 \label{fig:Ens error:50 modes: filtered}
\end{figure}


\subsection{dBFN and VBFN for KdV-type equations}\label{sect:KdV}
We now extend our investigation of stabilized backward steps to one-dimensional dispersive models. Specifically, we consider Korteweg-de Vries (KdV) dynamics on the periodic domain $\Omega = [0, 2\pi]$, restricted to mean-free solutions. Following the notation in \cref{sect:1D}, the observational operator is defined as the Fourier projection $P_M$, with $Q_M = I - P_M$ representing the projection onto the unobserved modes.

We examine two variations of the forced KdV equation, both expressed in the form $u_t = F(u)$ and driven by a time-independent forcing $f$. First, we consider the \textit{Damped-Driven KdV} equation with damping parameter $\gamma > 0$:
\begin{equation}\label{eq:KdV:dd}
    u_t + u u_x + u_{xxx} + \gamma u = f.
\end{equation}
In this case, the operator is defined as $F_{\gamma}(u) := -u u_x - u_{xxx} - \gamma u + f$.

Second, we consider the \textit{Viscous KdV} (or KdV-Burgers) equation with viscosity $\nu > 0$:
\begin{equation}\label{eq:KdV:visc}
    u_t + u u_x + u_{xxx} - \nu u_{xx} = f.
\end{equation}
The corresponding operator  for \cref{eq:KdV:visc} is $F_{\nu}(u) := -u u_x - u_{xxx} + \nu u_{xx} + f$.

\subsubsection{BFN and Stabilized Backward Variants}
Fix an assimilation window $[0,T]$ and assume the reference solution $u$ is observed continuously in time via $P_M u(t)$. Similar to the 2D NSE, direct application of the BFN algorithm to systems \cref{eq:KdV:dd,eq:KdV:visc} is problematic due to the ill-posed nature of the backward evolution in \cref{eq:KdV:visc}. During the backward step, the linear dissipative terms---damping ($\gamma u$) or diffusion ($-\nu u_{xx}$)---become energy-generating (anti-dissipative), inducing exponential growth in high wavenumbers in the case of \cref{eq:KdV:visc}.

To address this, we implement and compare three variations of the backward nudging step: the Standard BFN, the Stabilized (Diffusive/Damped) BFN, and the Voigt-regularized BFN. The forward and backward iterations are given by:
\begin{subequations}
\begin{align}
\text{(F)}\quad &
\begin{cases}
    v_t^n = F(v^n) + \mu P_M(u - v^n),\\
    v^n(0) = \tilde v^{n-1}(0),
\end{cases}\label{eq:KdV:BFN:F}
\\
\text{(Standard)}\quad &
\begin{cases}
    \tilde v_t^n = F(\tilde v^n) - \mu P_M(u - \tilde v^n),\\
    \tilde v^n(T) = v^n(T).
\end{cases}\label{eq:KdV:BFN:B}
\end{align}
\end{subequations}
Here, $F$ represents either $F_\gamma$ or $F_\nu$. \Cref{eq:KdV:BFN:B} represents the \textit{Standard BFN}, which retains the unaltered dynamics and is consequently susceptible to numerical instability.

To stabilize the backward evolution, we consider modifications that explicitly cancel the destabilizing anti-dissipative terms. This is analogous to the dBFN approach used for the NSE. For the \textit{Viscous KdV} \cref{eq:KdV:visc}, we employ a \textit{Diffusive BFN} strategy where the sign of the viscous term is reversed:
\begin{equation}
\text{(Diffusive)}\quad 
\begin{cases}
    \tilde v_t^n = F_{\nu}(\tilde v^n) - 2\nu \tilde v^n_{xx} - \mu P_M(u - \tilde v^n),\\
    \tilde v^n(T) = v^n(T).
\end{cases}\label{eq:KdV:dBFN:B}
\end{equation}
This effectively replaces backward anti-diffusion with forward diffusion, ensuring well-posedness at the cost of model consistency. Similarly, for the \textit{Damped KdV} \cref{eq:KdV:dd}, we employ a \textit{Damped BFN} strategy by reversing the sign of the linear damping term:
\begin{equation}
\text{(Damped)}\quad 
\begin{cases}
    \tilde v_t^n = F_{\gamma}(\tilde v^n) + 2\gamma \tilde v^n - \mu P_M(u - \tilde v^n),\\
    \tilde v^n(T) = v^n(T).
\end{cases}\label{eq:KdV:dampflip:B}
\end{equation}

Finally, we introduce the \textit{Voigt-regularized BFN} (VBFN) for these dispersive models. Rather than altering the physical parameters of the equation (e.g., flipping signs of viscosity or damping), we apply Helmholtz smoothing to the time derivative:
\begin{equation}
\text{(Voigt)}\quad 
\begin{cases}
    (I-\alpha^2\partial_{xx})\tilde v_t^n = F(\tilde v^n) - \mu P_M(u - \tilde v^n),\\
    \tilde v^n(T) = v^n(T),
\end{cases}\label{eq:KdV:Voigt:B}
\end{equation}
with regularization parameter $\alpha > 0$. As discussed in \cref{sect:NSE}, this regularization mollifies the nonlinear term and inhibits the energy cascade to high modes during backward dynamics while preserving the underlying structure of the dispersive operator.

\begin{Rmk}
We also investigated filtered variants for the KdV equations, where the stabilization (either diffusive or Voigt) is applied only to the unobserved modes via the projection $Q_M$. The resulting behavior was consistent with the 2D NSE results presented in \cref{eq:NSE:fdiff-B,eq:NSE:fVoigt-B}; therefore, we omit the detailed formulations here.
\end{Rmk}

\subsubsection{Numerical Methods and Experimental Setup}\label{sect:KdV:numerics}

All numerical simulations for the KdV-type equations were performed on a periodic domain $\Omega = [0, L_x]$ with $L_x = 2\pi$, discretized on an equispaced grid of $N=512$ points. Spatial derivatives were computed using a Fourier pseudo-spectral method. Nonlinear terms were evaluated in physical space and transformed back to the spectral domain using the $2/3$-rule for dealiasing.

Time integration was executed using a fourth-order integrating-factor Runge--Kutta (IFRK4) scheme. The linear dispersive term was integrated exactly via an exponential integrating factor in Fourier space, while the nonlinear term and the forcing were treated explicitly. The prescribed forcing is smooth and time-independent, defined as:
\begin{equation}
    f(x) = f_0 e^{\cos(x)} -\int_0^{2\pi} f_0 e^{\cos(x)} dx.
\end{equation}

We utilize an ``identical twin'' experimental framework. A reference solution $u(t)$ is generated by integrating the chosen KdV model forward in time over the interval $t \in [0, T]$ with $T=1$ and a time step of $\Delta t = 10^{-5}$. This reference trajectory is stored to provide continuous-in-time observations.

In these experiments, the assimilating state is initialized with the exact initial condition, $v^1(0) = u(0)$. While this is nonstandard for initial condition recovery problems, it serves to isolate the stability properties of the backward step; since the forward leg is resolved to machine precision, any subsequent error growth is attributable directly to numerical instability or model discrepancy within the regularized backward dynamics.

Observations are restricted to the low-mode Fourier components $P_M u(t)$ with a cutoff of $M=15$ (excluding the mean). The nudging term is implemented explicitly: for each forward leg, $\mu P_M(u-v)$ is added to the right-hand side, and for each backward leg, the corresponding sign change required by the BFN algorithm is applied. Unless otherwise specified, we employ a nudging strength of $\mu=1000$ and perform five back-and-forth cycles.

To quantify reconstruction performance, we monitor the total error and its decomposition into observed and unobserved modes:
\begin{equation}
    \|P_M(u-v)\|_{L^2} \quad \text{and} \quad \|Q_M(u-v)\|_{L^2}.
\end{equation}
In the following figures, even legs correspond to forward evolution and odd legs correspond to backward evolution. Results are plotted over a rescaled ``iterated time'' axis to visualize the complete BFN cycle structure.

---

\subsubsection{Computational Results: Damped and Driven KdV}\label{sect:KdV:results:damped}

For this set of experiments, we use the parameters $T=1$, $\Delta t=10^{-5}$, $M=15$, and $\mu=1000$, with a forcing amplitude of $f_0=1$ and a damping parameter of $\gamma=0.5$. In all simulations presented in this section we used the spatial resolution $N = 512$. We compare the standard BFN with the stabilized backward variants described in \cref{sect:KdV}.

The standard BFN iteration (\cref{fig:KdV:damped:BFN}) exhibits a characteristic ``sawtooth'' error pattern typical of ill-posed backward problems. During the forward leg (even intervals), the nudging term effectively pulls the solution toward the reference state, reducing the error. However, during the backward leg (odd intervals), the anti-damping term $+\gamma u$ amplifies errors exponentially, erasing the gains made during the forward pass and preventing convergence.

In contrast, both the Damped stabilization (\cref{fig:KdV:damped:dampedBFN}) and Diffusive stabilization (\cref{fig:KdV:damped:diffusiveBFN}) prevent this exponential blow-up, maintaining bounded errors throughout the backward steps. Nevertheless, a distinct ``error floor'' is observed. This residual error stems from the fact that the backward model is no longer physically consistent with the forward dynamics; the stabilization terms introduce a model discrepancy (bias) that nudging cannot fully rectify.

Notably, the Voigt-regularized BFN (\cref{fig:KdV:damped:VoigtBFN}) achieves a similar stabilizing effect but produces a qualitatively different error structure in the high modes. By mollifying the nonlinearity rather than directly altering the damping coefficient, the VBFN preserves the structural integrity of the low-frequency dynamics better than the diffusive approach, resulting in a slightly lower asymptotic error floor.

\begin{figure}[t]
\centering
\begin{subfigure}[t]{0.49\textwidth}
\centering
\includegraphics[width=\linewidth]{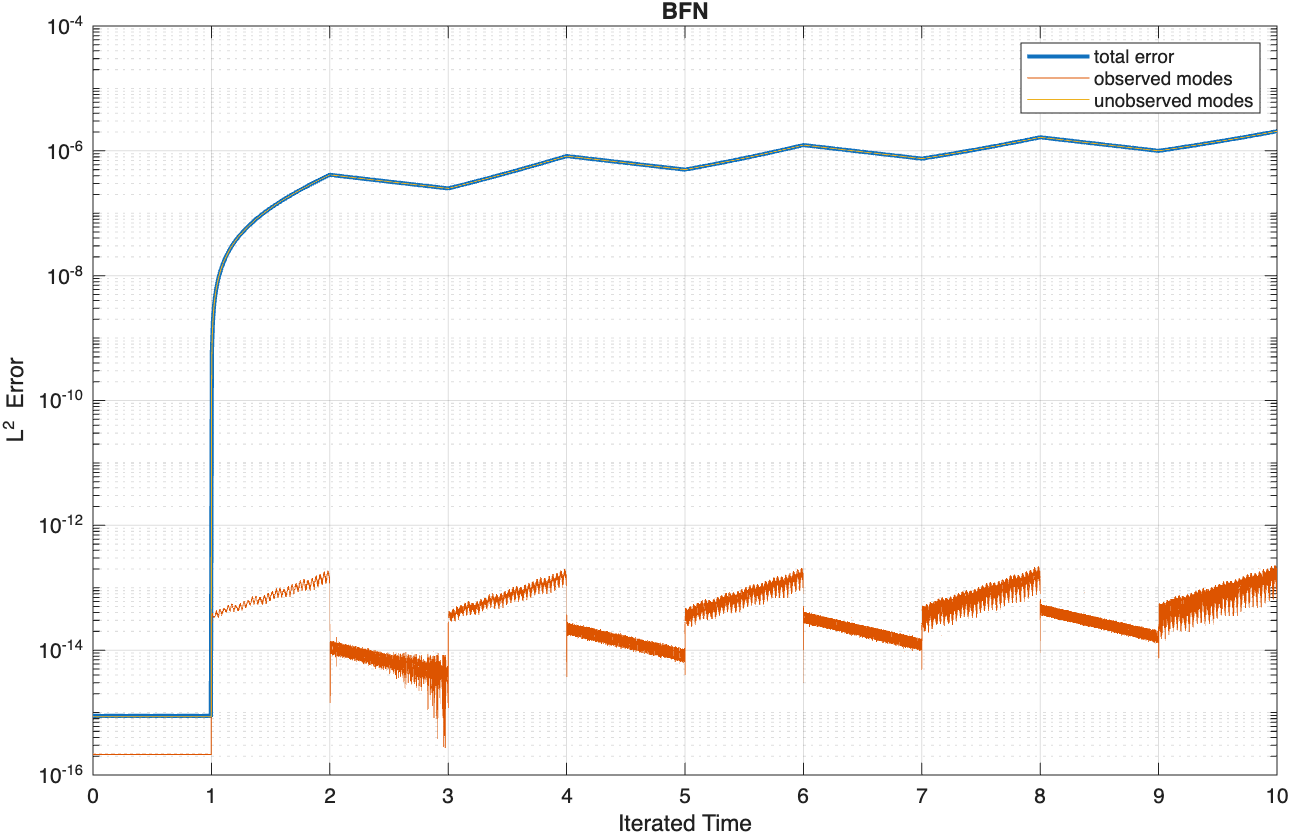}
\caption{\small\textbf{Standard} backward step.}
\label{fig:KdV:damped:BFN}
\end{subfigure}
\hfill
\begin{subfigure}[t]{0.49\textwidth}
\centering
\includegraphics[width=\linewidth]{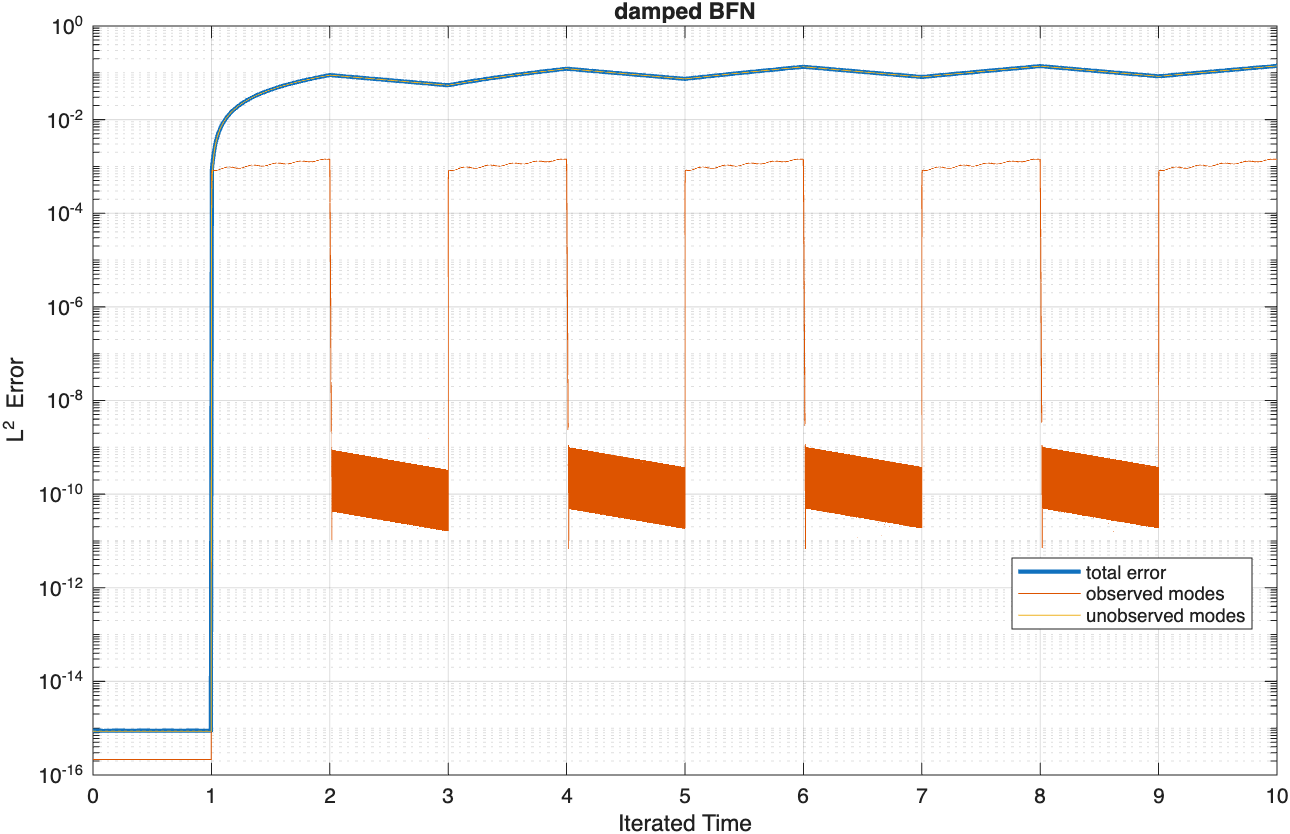}
\caption{\small \textbf{Damped} backward stabilization.}
\label{fig:KdV:damped:dampedBFN}
\end{subfigure}
\vspace{0.75em}
\begin{subfigure}[t]{0.49\textwidth}
\centering
\includegraphics[width=\linewidth]{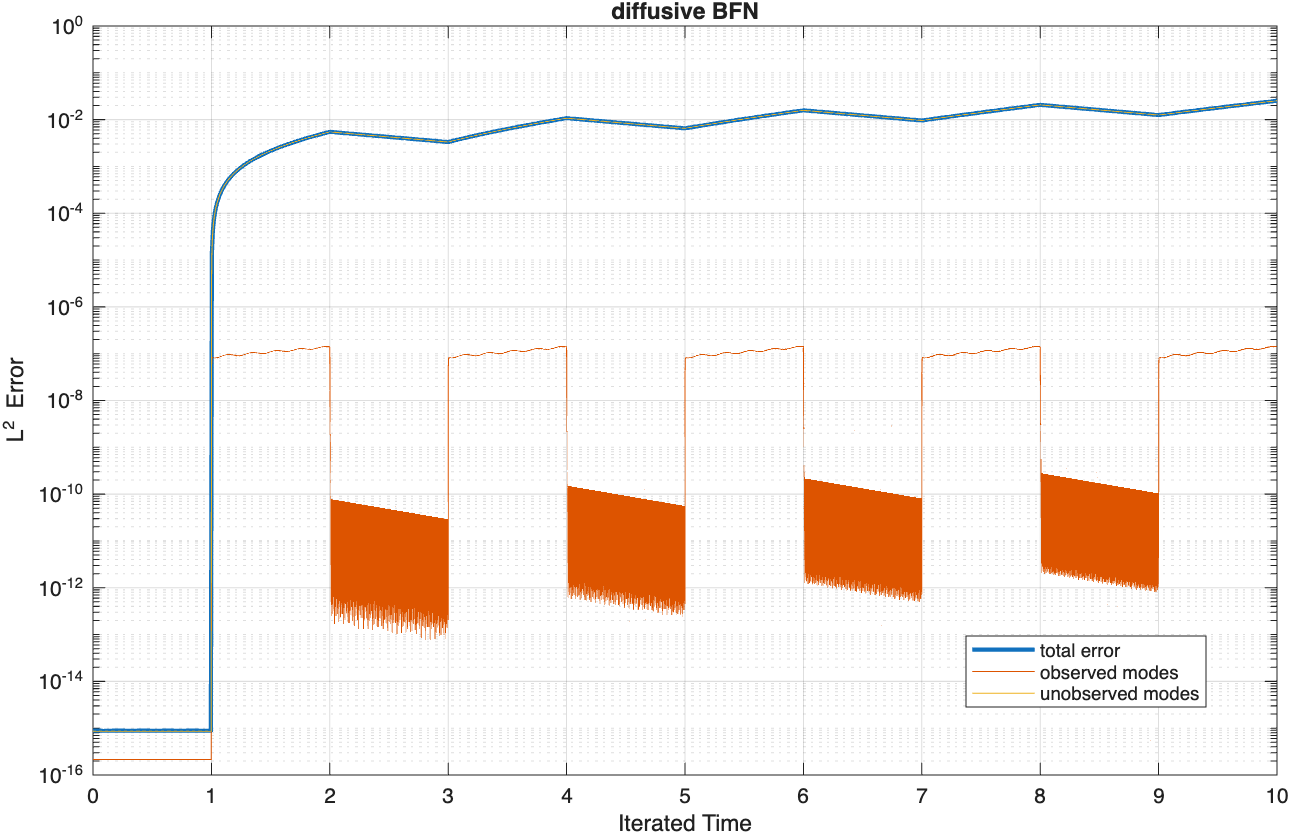}
\caption{\small \textbf{Diffusive} backward stabilization.}
\label{fig:KdV:damped:diffusiveBFN}
\end{subfigure}
\hfill
\begin{subfigure}[t]{0.49\textwidth}
\centering
\includegraphics[width=\linewidth]{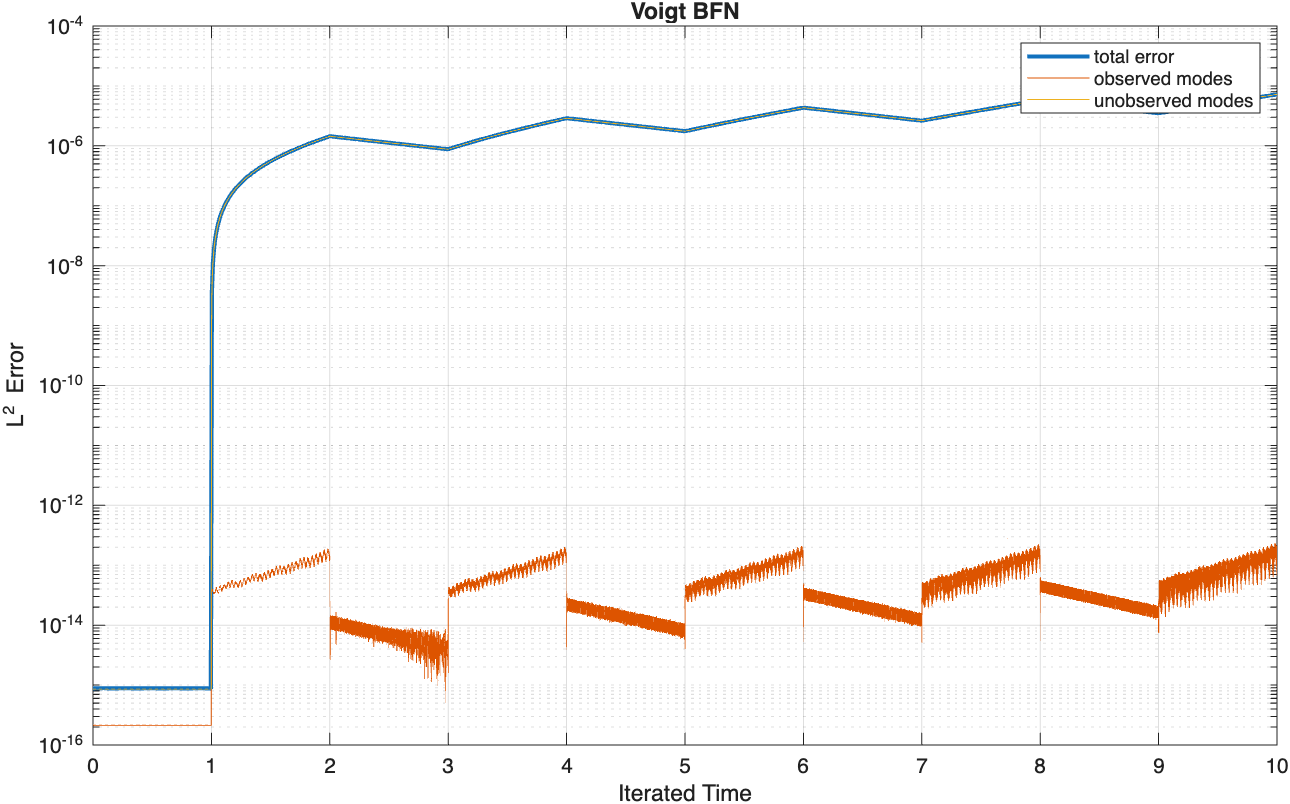}
\caption{\textbf{\small Voigt-regularized} backward step.}
\label{fig:KdV:damped:VoigtBFN}
\end{subfigure}
\caption{Damped/forced KdV: Error evolution for four backward reconstructions.}
\label{fig:KdV:damped:allBFN}
\end{figure}

\subsubsection{Computational Results: Viscous and Driven KdV}\label{sect:KdV:results:viscous}

For the viscous regime, we maintain the parameters $T=1$, $\Delta t=10^{-5}$, $M=15$, and $\mu=1000$, with a forcing amplitude of $f_0=1$ and a viscosity of $\nu=10^{-3}$. In all simulations presented in this section we used the spatial resolution $N = 512$.

The instability of the standard BFN is most pronounced in the viscous case, as illustrated in \cref{fig:KdV:viscous:BFN}. The backward heat operator acts as an anti-diffusive process with a growth rate proportional to $e^{\nu k^2 t}$. For the parameters chosen here, this growth is catastrophic; the error explodes by over 100 orders of magnitude within a single backward iteration, rendering the standard algorithm entirely unusable.

In contrast, the stabilized backward step (\cref{fig:KdV:viscous:dampedBFN}), which replaces anti-diffusion with damping (Diffusive BFN), successfully tames this instability. The error decreases initially before settling into a bounded oscillation near $10^{-2}$. This floor error represents the irreducible model discrepancy introduced by reversing the sign of the viscosity. Similarly, the Voigt-regularized approach (\cref{fig:KdV:viscous:VoigtBFN}) demonstrates robust stability. By filtering the time derivative, the VBFN prevents the blow-up of high-wavenumber modes while maintaining a bounded error profile comparable to that of the diffusive stabilization.

\begin{figure}[t]
\centering
\begin{subfigure}[t]{0.49\textwidth}
\centering
\includegraphics[width=\linewidth]{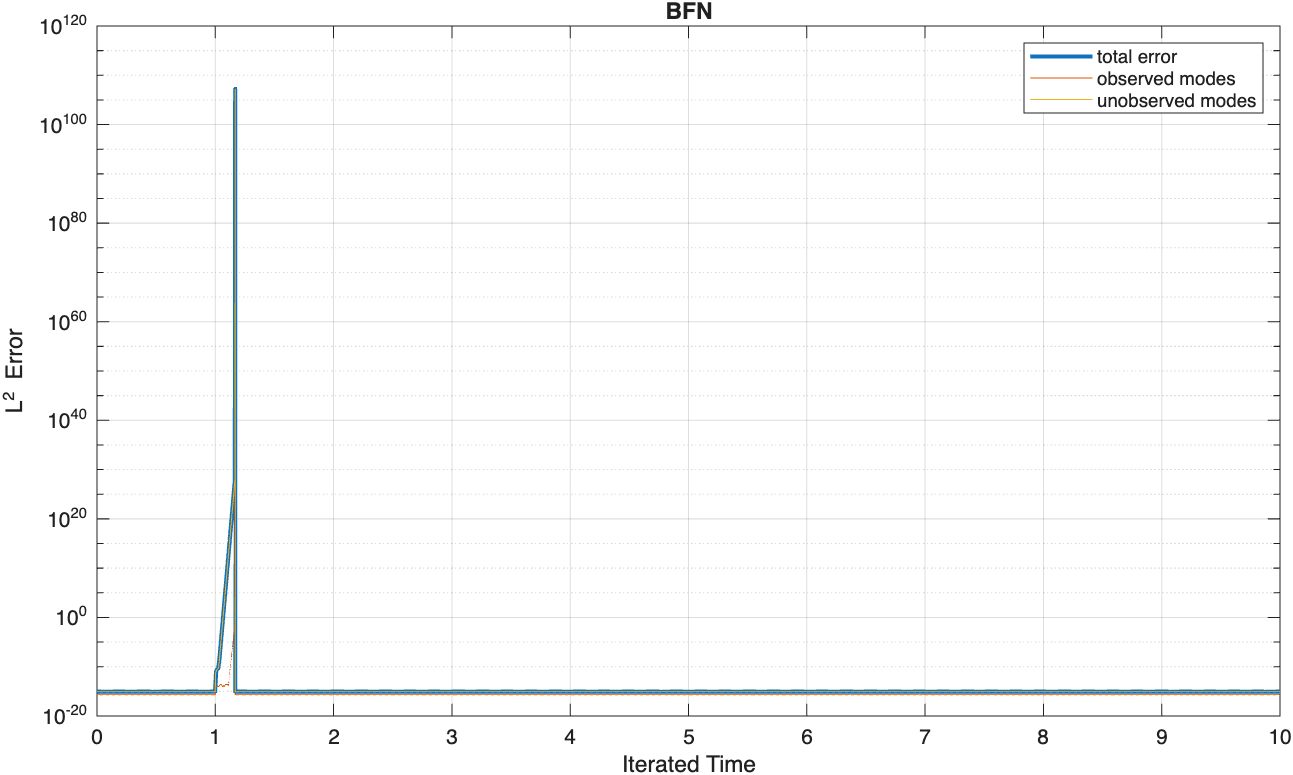}
\caption{\small \textbf{Standard} backward step.}
\label{fig:KdV:viscous:BFN}
\end{subfigure}
\hfill
\begin{subfigure}[t]{0.49\textwidth}
\centering
\includegraphics[width=\linewidth]{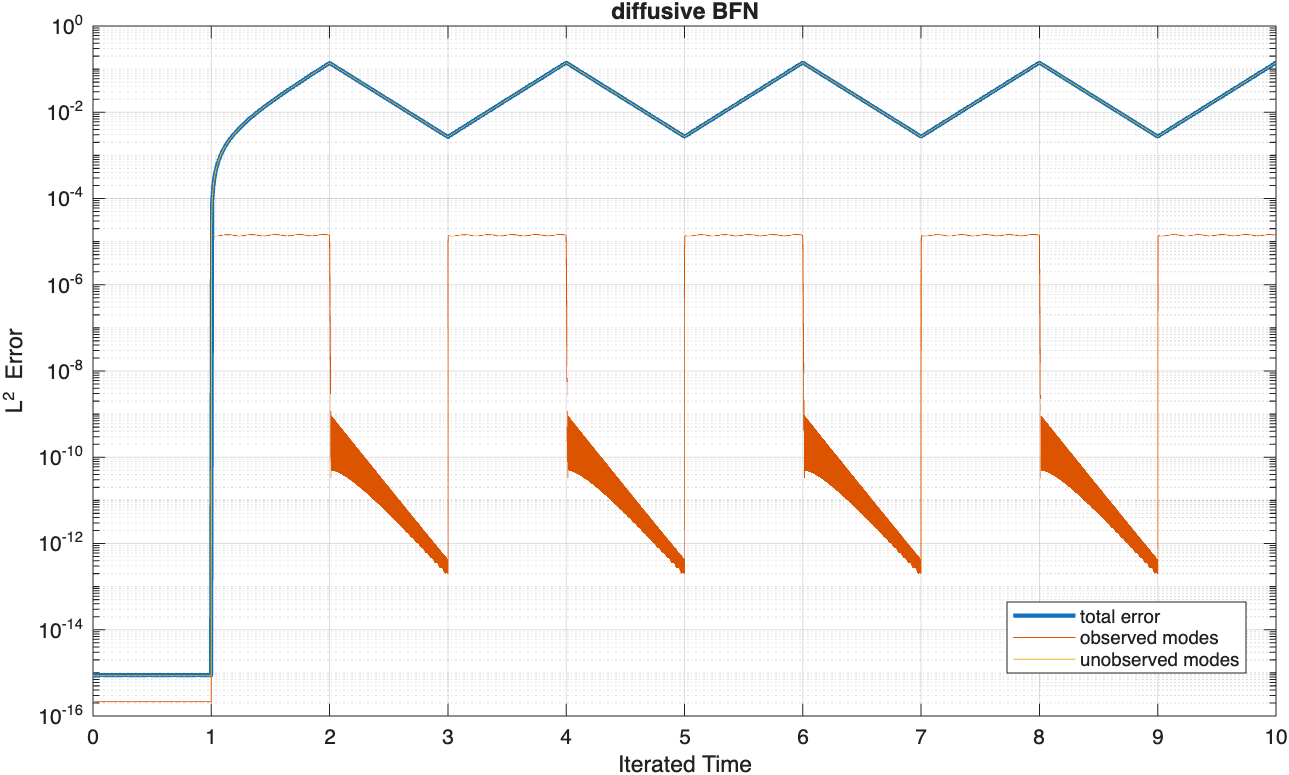}
\caption{\small \textbf{Stabilized} backward step with modified dissipation.}
\label{fig:KdV:viscous:dampedBFN}
\end{subfigure}
\vspace{0.75em}
\begin{subfigure}[t]{0.65\textwidth}
\centering
\includegraphics[width=\linewidth]{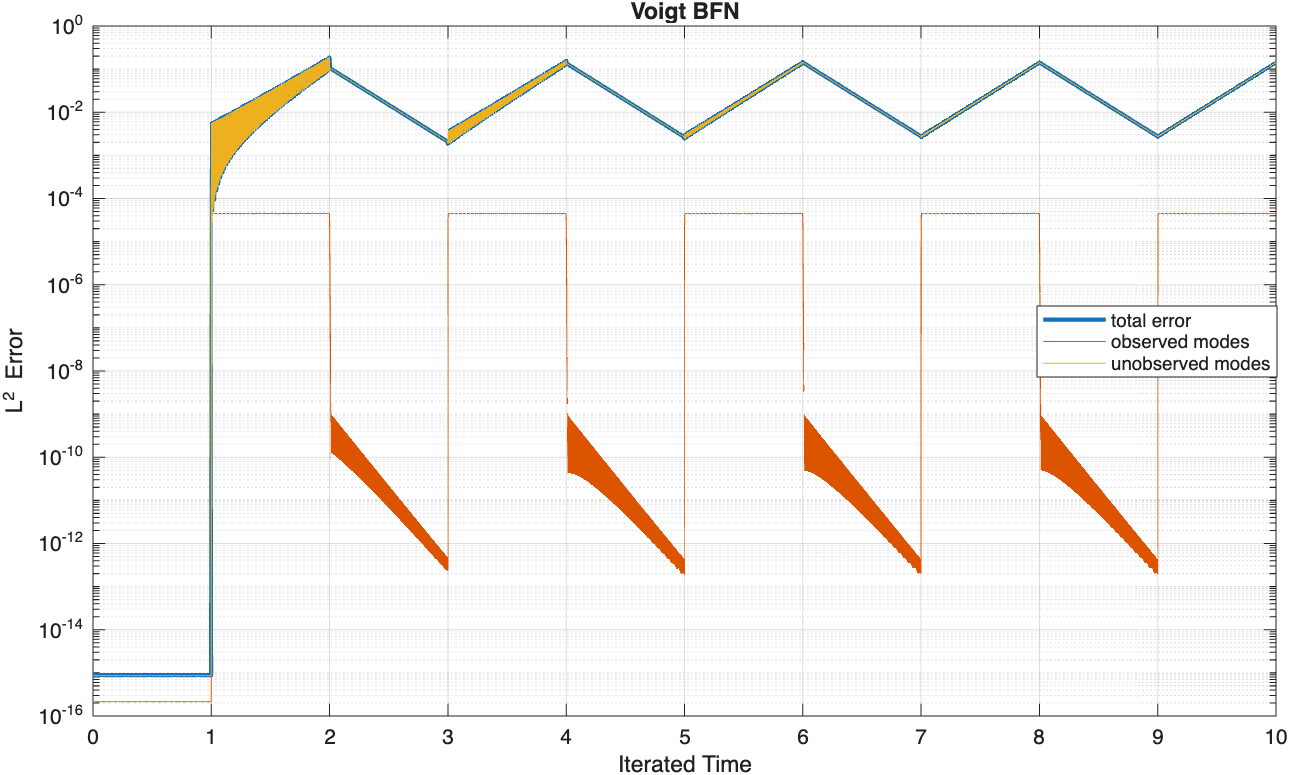}
\caption{\small \textbf{Voigt-regularized} backward step.}
\label{fig:KdV:viscous:VoigtBFN}
\end{subfigure}
\caption{\small Viscous/forced KdV: Error evolution for three backward reconstructions.}
\label{fig:KdV:viscous:allBFN}
\end{figure}

\section{Conclusion}\label{sect:Conclusion}

In this work, we critically examined the efficacy of the back-and-forth nudging (BFN) algorithm across a range of dynamical systems, including the Lorenz 1963 system and various 1D and 2D partial differential equations (PDEs). Our investigation focused on the algorithm's ability to reconstruct initial states from sparse, low-mode observations.

First, we analyzed the BFN algorithm for the Lorenz 1963 system both analytically and numerically. While prior studies established that observing the first two components of the Lorenz system is sufficient for forward-nudging to recover the true solution asymptotically, we demonstrated that the BFN algorithm fails to recover the true initial state for specific classes of solutions. A critical omission in previous literature \cite{Peng_Wu_Shiue_2023} was the failure to reconstruct the unobserved third component; our numerical simulations corroborate this, showing persistent error in the third component when only two are observed.

In the context of 1D PDEs—specifically the linear transport and Burgers equations in both viscous and inviscid regimes—our analytical findings reveal that the BFN algorithm is fundamentally limited under low-mode observations. We constructed counterexamples demonstrating that distinct initial conditions can produce identical sparse spatial observational data, rendering the recovery of the true initial state mathematically impossible. 

Numerical experiments supported these findings. In trials where observations were identically zero, the nudged solution remained trivial. Conversely, when the entire solution was observed, as in \cite{Auroux_Nodet_2012}, the algorithm recovered the initial condition to single precision. However, for more realistic initial data containing both observed low frequencies and unobserved high frequencies, the error reached a non-zero floor, with no development of high-frequency content in the nudged solution. This suggests that while BFN can succeed in regimes where the full state is already observed, such success occurs only when data assimilation is effectively redundant and extraneous.

Finally, we extended our analysis to the 2D incompressible Navier-Stokes equations and 1D KdV-type equations to test the algorithm in complex dissipative and dispersive settings. To address the inherent ill-posedness of the backward evolution, we introduced and compared two regularization strategies: the established Diffusive BFN (dBFN), which stabilizes the system by reversing the sign of the dissipation, and a proposed Voigt-regularized BFN (VBFN), which utilizes Helmholtz smoothing.

Our experiments showed that while standard BFN fails catastrophically in these regimes due to the exponential growth of high-wavenumber errors, both regularized variants successfully stabilize the backward iterations. However, this stability introduces model inconsistency; both methods reach an error saturation floor determined by the regularization parameter. Notably, the VBFN scheme outperformed dBFN by introducing significantly less bias into the observed large-scale modes. Furthermore, applying ``spectral filtering'' to restrict regularization to the unobserved modes did not alleviate the core issue: all variants remained unable to reconstruct the unobserved high-frequency spectrum.

In summary, these results underscore that even for dissipative systems, the BFN algorithm is fundamentally limited when only sparse spatial observational information is available. These findings are not restricted to the specific models studied here; the arguments can be readily extended to other dynamical systems, such as 2D magnetohydrodynamics (MHD) equations, highlighting a universal challenge in reconstructing transient states from incomplete data.


\section*{Acknowledgments}
\noindent
The work of AF was supported in part by the National Science Foundation through DMS 2206493.  The work of EST has benefited from the inspiring environment of the CRC 1114 “Scaling Cascades in Complex Systems”, Project Number 235221301, Project A02, funded by Deutsche Forschungsgemeinschaft (DFG). Moreover, this work was also supported in part by the DFG Research Unit FOR 5528 on Geophysical Flow.

\bibliographystyle{plain}
\bibliography{VictorBiblio}

\end{document}